\documentclass[11pt,letterpaper]{amsart}
\usepackage[foot]{amsaddr}

\usepackage{mathtools}
\usepackage[T1]{fontenc}
\usepackage[latin9]{inputenc}
\usepackage{lmodern} 
\usepackage{geometry}
\usepackage{multicol}
\usepackage{graphicx}
\usepackage{soul}
\usepackage{xcolor}
\usepackage{amssymb}
\usepackage{placeins}
\usepackage{cite}
\usepackage{caption}
\usepackage{enumerate}
\usepackage{afterpage}
\usepackage{enumitem}
\usepackage{bmpsize}
\usepackage{hyperref}
\usepackage{tabu}
\usepackage{enumitem}
\numberwithin{equation}{section}
\usepackage{stmaryrd}
\usepackage{tikz}
\usetikzlibrary{matrix,graphs,arrows,positioning,calc,decorations.markings,shapes.symbols}

\def\Re{ \mathsf{Re}}
\def\Im{ \mathsf{Im}}
\def\i{ \mathsf{i}}

\def \rr{\mathsf{r} }
\def \rl{\mathsf{c}}

\def \arg{\mathsf{arg}}

\makeatletter
\renewcommand{\email}[2][]{%
  \ifx\emails\@empty\relax\else{\g@addto@macro\emails{,\space}}\fi%
  \@ifnotempty{#1}{\g@addto@macro\emails{\textrm{(#1)}\space}}%
  \g@addto@macro\emails{#2}%
}
\makeatother

\newtheorem{theorem}{Theorem}[section]
\newtheorem{lemma}[theorem]{Lemma}
\newtheorem{proposition}[theorem]{Proposition}

{ \theoremstyle{definition}
\newtheorem{definition}[theorem]{Definition}}
{ \theoremstyle{remark}
\newtheorem{remark}[theorem]{Remark}}

\def\ZP{{Z^{M,N}}}

\newcommand{\R}{\mathbb{R}}

\title[Fluctuations of the log-gamma polymer free energy]{Fluctuations of the log-gamma polymer free energy with general parameters and slopes}

\author[G. Barraquand]{Guillaume Barraquand}
\address{G. Barraquand,
Laboratoire de physique de l'{\'e}cole normale sup\'erieure, ENS, Universit{\'e} PSL, CNRS, Sorbonne Universit{\'e}, Universit{\'e} de Paris, Paris, France}
\email{guillaume.barraquand@ens.fr}
\author[I. Corwin]{Ivan Corwin}
\address{I. Corwin, Department of Mathematics, Columbia University, New York, NY 10027, USA} \email{ivan.corwin@gmail.com}
\author[E. Dimitrov]{Evgeni Dimitrov}
\address{E. Dimitrov, Department of Mathematics, Columbia University, New York, NY 10027, USA} \email{esd2138@columbia.edu}

\begin{document}

\maketitle

\begin{abstract}
We prove that the free energy of the log-gamma polymer between lattice points $(1,1)$ and $(M,N)$ converges to the GUE Tracy-Widom distribution in the $M^{1/3}$ scaling, provided that $N/M$ remains bounded away from zero and infinity. We prove this result for the model with inverse gamma weights of any shape parameter $\theta>0$ and furthermore establish a moderate deviation estimate for the upper tail of the free energy in this case. Finally, we consider a non i.i.d. setting where the weights on finitely many rows and columns have different parameters, and we show that when these parameters are critically scaled the limiting free energy fluctuations are governed by a generalization of the Baik-Ben Arous-P\'ech\'e distribution from spiked random matrices with two sets of parameters.  
\end{abstract}

\tableofcontents

%
\section{Introduction and main results}
\label{Section1}

\subsection{The log-gamma polymer}\label{Section1.2}

The log-gamma polymer was introduced by Sepp\"al\"ainen \cite{Sep12} and is the only vertex-disorder 1+1 dimensional directed polymer model that is known to be exactly solvable, in the sense that the distribution of its free energy can be computed explicitly. The contribution of our present work is to establish asymptotics of this model's free energy fluctuations over a very wide range of parameters which control the dimensions of the polymer as well as the nature of its disorder. Proving these general asymptotic results require us to considerably rework the fundamental starting formula for this model, i.e. the Fredholm determinant Laplace transform formula. Our asymptotic results have applications that are being pursued in a number of contexts including showing tightness of the log-gamma line ensemble \cite{BCDpaperII}, showing a phase transition for the maximum of the log-gamma polymer free energy landscape \cite{BCDtoappear,kotowski2019tracy} and showing that the log-gamma polymer converges to the KPZ fixed point \cite{virag2020heat}.

%

\begin{definition} A continuous random variable $X$ is said to have the inverse-gamma distribution with parameter $\theta > 0 $ if its density is given by
\begin{equation}\label{S1invGammaDens}
f_\theta(x) = \frac{{\bf 1} \{ x > 0 \} }{\Gamma(\theta)} \cdot x^{-\theta - 1} \cdot \exp( - x^{-1}).
\end{equation}
The log-gamma polymer model that we consider in this paper depends on two integers $M, N \geq 1$ and two sets of parameters $\vec{\alpha} = (\alpha_1, \dots, \alpha_M) \in \mathbb{R}_{>0}^M$ and $\vec{a} = (a_1, \dots, a_N) \in \mathbb{R}^N$ such that $\theta_{i,j} = \alpha_i - a_j > 0$
for all $i = 1, \dots, M$ and $j = 1, \dots, N$.

We let $w = \left( w_{i,j} :  1 \leq i \leq M, 1 \leq j \leq N\right)$ denote the random matrix such that $w_{i,j}$ are independent random variables with density $f_{\theta_{i,j}}$ as in (\ref{S1invGammaDens}). We denote by $(\Omega, \mathcal{F}, \mathbb{P})$ the probability space on which $w$ is defined. A {\em directed lattice path} is a sequence of vertices $(x_1, y_1), \dots, (x_k, y_k) \in \mathbb{Z}^2$ such that $x_1 \leq x_2 \leq \cdots \leq x_k$, $y_1 \leq y_2 \leq \cdots \leq y_k$ and $(x_i - x_{i-1}) + (y_i - y_{i-1}) = 1$ for $i = 2, \dots, k$. In words, a directed lattice path is an up-right path on $\mathbb{Z}^2$, which makes unit steps in the coordinate directions (see Figure \ref{halfspaceloggamma}). We let $\Pi_{M,N}$ denote the set of directed paths $\pi$ in $\mathbb{Z}^2$ from $(1,1)$ to $(M,N)$. Given a directed path $\pi$ we define its {\em weight}
\begin{equation}\label{PathWeight}
w(\pi) = \prod\nolimits_{(i,j) \in \pi} w_{i,j}.
\end{equation}
We also define
\begin{equation}\label{PartitionFunct}
\ZP = \sum\nolimits_{ \pi \in \Pi_{M,N}} w(\pi),
\end{equation}
to be the {\em partition function} of the model.
\end{definition}

\begin{figure}
	\begin{center}
		\begin{tikzpicture}[scale=0.7]
		\begin{scope}
		\draw (0,0) node[anchor = north east]{{\footnotesize $(1,1)$}};
		\draw (9,6) node[anchor=south west]{{\footnotesize $(M,N)$}};
		\draw[->, >=stealth'] (9,2.5) node[anchor=west]{{ $w_{i,j}$ with density $f_{\theta_{i,j}}$}} to[bend right] (7,2);
		\draw (12,1.8) node{$\theta_{i,j}=\alpha_i-a_j$};
		\draw[dotted, gray] (0,0) grid (9,6);
		\draw[ultra thick] (0,0) -- (1,0) -- (2,0) -- (2,1) -- (2,2) -- (3,2) -- (4,2) -- (4,3) -- (5,3) -- (6,3) -- (6,4) -- (7,4) -- (8,4) -- (8,6) -- (9,6);
		\draw (7, -0.4) node{$i$};
		\draw (-0.4,2) node{$j$};
		\foreach \x in {1,2,...,10}
		\draw (\x-1,-1) node{$\alpha_{\x}$};
		\foreach \x in {1,2,...,7}
		\draw (-1,\x-1) node{$a_{\x}$};
		\end{scope}
		\end{tikzpicture}
	\end{center}
	\caption{A directed lattice path $\pi\in \Pi_{M, N}$ in the log-gamma polymer model. }
	\label{halfspaceloggamma}
\end{figure}
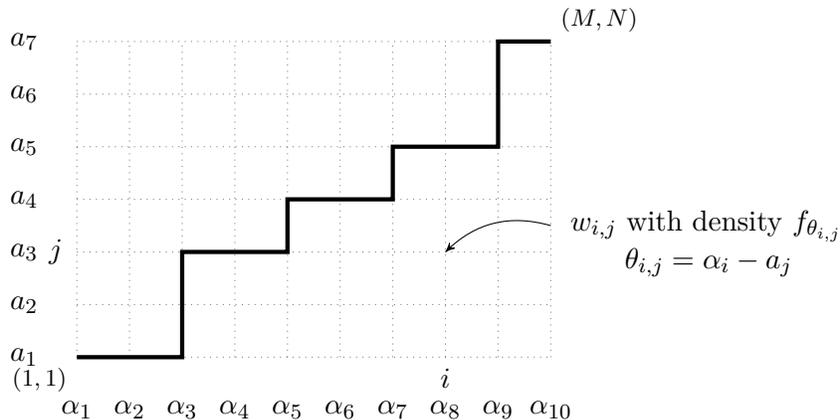

\bigskip
The main goal of this paper is to understand the asymptotic behavior of $\log \ZP$ as $N,M \rightarrow \infty$ while $N/M$ remains bounded away from $0$ and $\infty$. We describe these asymptotic results now.

%
\subsection{GUE Tracy-Widom asymptotics} \label{Section1.3}

In this section we consider the log-gamma polymer with i.i.d. weights. Fix $\theta>0$ and assume that $a_j = 0$ for all $j = 1,\dots, N$ and $\alpha_i=\theta$ for all $i =1, \dots, M$. Before we state the result we introduce some necessary notation. Define the function
\begin{equation}\label{DefLittleg}
g(z) = \frac{\sum_{n =0}^\infty \frac{1}{(n+\theta - z)^2}}{ \sum_{n = 0}^\infty \frac{1}{(n+z)^2}},
\end{equation}
and observe that $g$ is a smooth, increasing bijection from $(0, \theta)$ to $(0, \infty)$. Clearly, the inverse function $g^{-1}: (0, \infty) \rightarrow (0,\theta)$ is also a strictly increasing smooth bijection. We now define for $x \in (0,\infty)$ the function
\begin{equation}\label{HDefLLN}
\begin{split}
h_\theta(x) = x \cdot  \Psi(g^{-1}(x)) + \Psi( \theta - g^{-1}(x)),
\end{split}
\end{equation}
where  $\Psi(x)$ is the digamma function, i.e.
\begin{equation}\label{digammaS1}
\Psi(z) = \frac{\Gamma'(z)}{\Gamma(z)} = - \gamma_{E} + \sum_{n = 0}^\infty \left[\frac{1}{n + 1} - \frac{1}{n+z} \right],
\end{equation}
and $\gamma_{E}$ is the Euler constant. One readily observes that $h_\theta$ is a smooth function on $(0, \infty)$. Finally, for any $x > 0$ we define
\begin{equation}\label{DefSigma}
\sigma_\theta(x) := \left[ \sum_{n = 0}^\infty \frac{x}{(n+g^{-1}(x))^3} +  \sum_{n = 0}^\infty \frac{1}{(n+\theta - g^{-1}(x))^3} \right]^{1/3}.
\end{equation}
For all $N,M\geq 1$, we define the rescaled free energy
\begin{equation}
\label{eq:rescaledpartitionfunction}
\mathcal F(M,N) = \frac{\log \ZP+ M h_{\theta}(p_M)}{M^{1/3}\sigma_{\theta}(p_M)}, \;\;\;p_M=N/M.
\end{equation}

We are now able to state the convergence to the Tracy-Widom distribution.
\begin{theorem}\label{LGPCT} Let $\theta > 0$ and $\delta \in (0,1)$. Let $\mathcal F(M,N)$ be as in  \eqref{eq:rescaledpartitionfunction} and assume that we let $M,N$ go to infinity in such a way that the sequence $p_M=N/M \in [\delta, \delta^{-1}]$ for all sufficiently large $M$. Then for all $y\in \R$ we have
\begin{equation}\label{HLConv}
\lim_{M \rightarrow \infty} \mathbb{P} \left( \mathcal F(M,N)  \leq y\right) = F_{\rm GUE}(y),
\end{equation}
where $F_{\rm GUE}$ is the GUE Tracy-Widom distribution \cite{TWPaper}. 
\end{theorem}
\begin{remark}\label{S1RemTW}
Note that $Z^{M,N}$ and $Z^{N,M}$ are equal in distribution. Using the fact that $g^{-1}(1/x)=\theta-g^{-1}(x)$, this implies that $\mathcal F (M,N)$ and $\mathcal F(N,M)$ are also equal in distribution. Owing to this symmetry, it suffices to prove Theorem \ref{LGPCT} for $p_M\in [\delta,1]$.
\label{rem:symmetry}
\end{remark}
\begin{remark}\label{rem:thm13}
Some special cases of Theorem \ref{LGPCT} had been established in previous works. First, it was proved in \cite[Theorem 1]{BCR} that the free energy $\log \ZP$ has Tracy-Widom GUE fluctuations with a restrictive assumption on the shape parameters of the gamma distribution ($\theta$ was assumed to be sufficiently small), and only for the free energy along the main diagonal $M=N$, i.e. $p_M\equiv 1$. The result was then extended in  \cite[Theorem 2.1]{KQ} to cover all values $\theta > 0$, still restricting to the diagonal direction. Theorem \ref{LGPCT} extends the  result of \cite{BCR, KQ} to any sequence of slopes $p_M$ that is uniformly bounded away from $0$ and $\infty$. We mention that the asymptotic analysis becomes notably more involved once $p_M \neq 1$ -- see Remark \ref{RemBBP2}.
\end{remark}
\begin{remark}
	Theorem \ref{LGPCT} constitutes an important input of a work in preparation \cite{BCDpaperII} proving the tightness of the top curve of a discrete version of the Airy line ensemble \cite{CorHamA} naturally associated to the log-gamma polymer \cite{Wu19,JonstonOConnell}. This top curve describes the free energy of polymer paths of fixed length, seen as a function of the endpoint, and is expected to converge to the Airy$_2$ process as the length of polymer paths goes to infinity.
\end{remark}
\begin{remark}\label{rem:thm13v2}
We mention that if $N/M \rightarrow 0$, there are two cases that are understood: if $N$ remains bounded the fluctuations are of Gaussian type and described by the O'Connell-Yor semi-discrete directed polymer \cite{oconnell2001brownian}; if $N=M^{1/\alpha}$ with $0<\alpha<3/14$ the fluctuations follow the GUE Tracy-Widom distribution as follows from the notably more general setup of \cite{auffinger2012universality}.
\end{remark}

Theorem \ref{LGPCT} is proved in Section \ref{Section4.1}. The starting point of our analysis is a formula, Theorem \ref{LGPT1}, for the Laplace transform of $\ZP$. This formula expresses the Laplace transform as a Fredholm determinant of a convolution kernel that depends on the parameters $\vec{\alpha}$ and $\vec{a}$ of the model.
The proof of Theorem \ref{LGPCT} is then based on the asymptotic analysis of this Fredholm determinant formula via the steepest descent method. The main challenges that we overcome in this pursuit are finding suitable contours that allow us to perform a steepest descent argument and obtaining the necessary estimates for the kernel in our Fredholm determinant along these contours. The behavior of the integrand in the definition of the kernel along various contours is studied in Section \ref{Section9} and estimates on the kernel itself are given in Section \ref{Section10.3}. 

Our starting formula in Theorem \ref{LGPT1} leverages an earlier result of \cite{BCFV} (which in turn built upon earlier work of \cite{COSZ,BorCor,BCR}) which contain a similar Laplace transform formula. Namely, the functions in \cite{BCFV} are the same, though the contours on which they are integrated and on which the Fredholm determinant is defined are rather different and unsuitable for the broad range of asymptotics that we investigate here. Reworking that formula into one which is appropriate for asymptotics requires significant effort and is the content of Section \ref{Section2}. The preface to Section \ref{Section2} provide  discussion on this point and the relation of our work to previous literature.

%
%


%

\subsection{Moderate deviation estimate}\label{Section1.4}

The careful asymptotic analysis of the Fredholm determinant carried out for the proof of Theorem \ref{LGPCT} also yields, as a byproduct, moderate deviation bounds for the one point fluctuations of the free energy. These bounds hold uniformly in the slope. Such tail bounds are of independent interest, and we will provide one application in a paper in preparation \cite{BCDtoappear}. That work considers the maximum of the free energy landscape for the log-gamma polymer. In other words, we allow the starting and ending points for the log-gamma polymer to vary within an $N$ by $M$ box and study maximum for the logarithm of the associated partition functions. In \cite{BCDtoappear}, we prove a phase transition for how this maximum is achieved as controlled by $\theta$. This is partially motivated by recent results of \cite{kotowski2019tracy} which relates this question to the study of a singular values for a random Schr\"odinger operator on the honeycomb lattice.

We continue with the same notation as in the previous section. In particular, $\theta>0$ is fixed and $\alpha_i=\theta$ for $i =1, \dots, M$ and $a_j = 0$ for $j = 1,\dots, N$. The below result is proved in Section \ref{Section4.2}.
\begin{theorem}\label{S1LDE1}Let $\theta > 0$ and $\delta \in (0,1)$ be given. Then there exist constants $C_1, C_2,c_1, c_2 > 0$ all depending on $\theta, \delta$ such that the following holds for all $M, N \in \mathbb{N}$ with $N/M \in [\delta, \delta^{-1}]$ and $x \geq 0$
\begin{equation}\label{S1LDEeq}
\mathbb{P} \left( \mathcal F(M,N)  \geq x\right) \leq C_1 e^{-c_1 M} + C_2 e^{-c_2 x^{3/2}}.
\end{equation}
\end{theorem}
\begin{remark}A different upper-tail estimate has been previously obtained in \cite{georgiou2013large}, which proves a large deviation principle for $\log \ZP$ with an explicit rate function. In the present setting, this corresponds to  scaling $x$ as $x=M^{2/3}  y$. More precisely, for $ y>0$ fixed, \cite[Theorem 2.2]{georgiou2013large} shows
$$\mathbb{P} \left( \mathcal F(M,N)  \geq  y M^{2/3}\right) \leq  C e^{-M I( y)},$$
 where $I(\cdot)$ is an explicit rate function. It was noted in \cite[Remark 2.9]{georgiou2013large} that $I(\epsilon) = C \epsilon^{3/2}+ o(\epsilon^{3/2})$ as $\epsilon$ goes to $0$, which is consistent with Theorem \ref{S1LDE1}. Here $C>0$ is a constant depending on $\theta$ and the ratio of $N$ and $M$. The inequality in Theorem \ref{S1LDE1} becomes relevant (i.e. not implied by the results from\cite{georgiou2013large}) in the moderate deviation regime, that is when we consider deviations of $\mathcal F(M,N)$ by $x$, whose size is smaller than $M^{2/3}$ and much larger than $0$.
\end{remark}
\begin{remark}
For very large $M$, the distribution of $\mathcal F(M,N)$ is close to the Tracy-Widom distribution by Theorem \ref{LGPCT}, and the right-hand side of \eqref{S1LDEeq} is dominated by $C_2 e^{-c_2 y^{3/2}}$ which, up to the values of $C_2$ and $c_2$ corresponds to the upper-tail behavior of the Tracy-Widom distribution.
\end{remark}
\begin{remark}
Some upper-tail estimates of one-point fluctuations have been obtained for a variety of models in the KPZ universality class, such as the KPZ equation itself \cite{corwin2013crossover, das2019fractional}, the asymmetric simple exclusion process \cite{damron2018coarsening}, the semi-discrete O'Connell-Yor polymer model \cite{seppalainen2010bounds},  longest increasing subsequences of a random permutation \cite{seppalainen1998large, lowe2001moderate}, or the Laguerre Unitary ensemble \cite{ledoux2010small} (which relates directly to exponential last passage percolation, i.e. the zero temperature limit of the log-gamma polymer). In particular, \cite{lowe2001moderate} proves a moderate deviations principle with the optimal constant $c_2=4/3$.
\end{remark}

%
\subsection{Baik-Ben Arous-P\'ech\'e asymptotics} \label{Section1.5} In the following section we describe a different scaling limit for the free energy of the log-gamma polymer model from Section \ref{Section1.2}, which is obtained when a finite number of the parameters $a_i$ and $\alpha_j$ are suitably scaled. As will be explained, one obtains another limit, which is a finite rational (in terms of the determinantal kernel) perturbation of the GUE Tracy-Widom distribution. We call the resulting distribution a {\em Baik-Ben Arous-P\'ech\'e (BBP) distribution with two sets of parameters}. We define this distribution in the following section.

BBP asymptotics with one set of parameters were first defined in the study of exponential last passage percolation in the original BBP paper \cite{baik2005phase} and their two parameter deformation was then introduced by Borodin-P\'ech\'e in \cite{BorPech}. This class of distributions have been shown to arise in some interacting particle systems \cite{baik2005phase, aggarwal2016phase, barraquand2014phase} and in the O'Connell-Yor semi-discrete directed polymer \cite{borodin2012free, TV20} when initial data or boundary conditions are suitable perturbed (as presently). They should similarly arise in all Kardar-Parisi-Zhang universality class models.

For the log-gamma polymer model, our results are the first instance of BBP asymptotics that have been obtained. In fact, one of our motivations for filling this gap now is that there is a significant interest recently in BBP type asymptotics due to the key role that they play in the general convergence theorem announced in \cite{virag2020heat}. In a sense, the results of \cite{virag2020heat} suggest that once the BBP asymptotics are established for the log-gamma polymer, it should be possible to prove convergence of the model to the KPZ fixed point.  We state exactly those needed asymptotic result as Theorem \ref{BBP} below.

\subsubsection{Borodin-P\'ech\'e's extension of the BBP distribution with two sets of parameters} We introduce the following useful notation. For $a \in \mathbb{C}$ and $\phi \in (0, \pi)$ we define the contour $C_{a,\phi}$ to be the union of $\{a + ye^{-\i\phi} \}_{y \in \mathbb{R}^+}$ and $\{a + ye^{\i \phi} \}_{y \in \mathbb{R}^+}$ oriented to have increasing imaginary part. Throughout the paper we write $\i = \sqrt{-1}$ and reserve the usual letter $i$ for indexing purposes.

\begin{definition}\label{DefBBP} Let $\rr, \rl \in \mathbb{Z}_{\geq 0}$ and fix $\vec{x} = (x_1, \dots, x_{\rr}) \in \mathbb{R}^{\rr}$ and $\vec{y} = (y_1, \dots, y_{\rl}) \in \mathbb{R}^{\rl}$ such that $\min(\vec{y}) > \max(\vec{x})$.  In addition, suppose that $a, b \in \mathbb{R}$ are such that $\min(\vec{y}) > b > a > \max(\vec{x})$ (by convention $\max(\vec{x}) = -\infty$ if $\rr = 0$ and $\min(\vec{y}) = \infty$ if $\rl = 0$) . With this data set for $r \in \mathbb{C}$
\begin{equation}\label{S1FDet}
F_{{ \rm BBP}; \vec{x}, \vec{y}}(r) = \det \left( I + K_r^{\rm BBP} \right)_{L^2(C_{a ,3\pi/4})},
\end{equation}
where $K_r^{\rm BBP}$ is defined in terms of its integral kernel
\begin{equation}\label{kernelMainS1}
\begin{split}
 &K_r^{\rm BBP}(v,v') = \frac{1}{2\pi \i} \int_{C_{b, \pi/4}} \prod_{n = 1}^{\rr} \frac{( {w} - x_n )}{({v}  - x_n  )} \prod_{m = 1}^{\rl} \frac{( y_m -  {v} )}{(y_m  -{w} )}\cdot  \frac{\exp(-v^3/3 + w^3/3 - r w + r v)}{(v-w)(w - v')} dw.
\end{split}
\end{equation}
\end{definition}

In equation (\ref{S1FDet}) the expression on the right-hand side is a {\em Fredholm determinant}; we refer the reader to Section \ref{Section2.1} for some basic definitions and facts about the latter. Explicitly, we have
\begin{equation}\label{S1FDet2}
\det \left( I + K_r^{\rm BBP} \right)_{L^2(C_{a ,3\pi/4})} = 1 + \sum_{n = 1}^\infty \frac{1}{n!} \int_{C_{a ,3\pi/4}} \cdots \int_{C_{a ,3\pi/4}} \det \left[ K_r^{\rm BBP}(v_i, v_j)\right]_{i,j = 1}^n \prod_{i = 1}^n \frac{dv_i}{2\pi \i}.
\end{equation}

The following result summarizes the basic properties we will require for $F_{{ \rm BBP}; \vec{x}, \vec{y}}$ -- its proof can be found in Section \ref{Section4.3}.
\begin{lemma}\label{S1Lemma} The integral in (\ref{kernelMainS1}) is absolutely convergent, each summand in (\ref{S1FDet2}) is an absolutely convergent integral and the resulting series is absolutely convergent for each $r \in \mathbb{C}$; in particular, the function $F_{{ \rm BBP}; \vec{x}, \vec{y}}$ from Definition \ref{DefBBP} is well-defined. Moreover, $F_{{ \rm BBP}; \vec{x}, \vec{y}}(r)$ does not depend on the choice of $a, b$ as long as $\min(\vec{y}) > b > a > \max(\vec{x})$, is entire in $r$ and defines a cumulative distribution function, when restricted to $\mathbb{R}$.
\end{lemma}

We call $F_{{ \rm BBP}; \vec{x}, \vec{y}}$ the {\em Baik-Ben Arous-P\'ech\'e distribution with two sets of parameters}. The name comes from the fact that when $\rl = 0$ we have that $F_{{ \rm BBP}; \vec{x}, \vec{y}}$ becomes the usual Baik-Ben Arous-P\'ech\'e distribution \cite{baik2005phase} as can be deduced from \cite[Lemma C.2]{BCF}. When $\rr = \rl = 0$ then $F_{{ \rm BBP}; \vec{x}, \vec{y}}$ becomes the Tracy-Widom distribution. Note that in its full generality, $F_{{ \rm BBP}; \vec{x}, \vec{y}}$ was first introduced in \cite{BorPech}, hence we could call it the Borodin-P\'ech\'e distribution, though for shorthand we will refer to all one and two parameter distributions as BBP.

\begin{remark} In \cite{BorPech} the authors considered a generalization of the extended Airy kernel, which they called the {\em extended Airy kernel with two sets of parameters}. If one sets the time parameters $t_1 = t_2$ in \cite[equation (4)]{BorPech} one precisely arrives at $F_{{ \rm BBP}; \vec{x}, \vec{y}}$. In particular, $F_{{ \rm BBP}; \vec{x}, \vec{y}}$ arises from the extended Airy kernel with two sets of parameters the same way the GUE Tracy-Widom distribution arises from the usual extended Airy kernel.
\end{remark}

\begin{remark}
There are other asymptotic problems which one could pursue for the log-gamma polymer based on our preasymptotic formula. This includes studying the model where inhomogeneity parameters are decaying or growing according to a certain profile (see related work in \cite{johansson2008some, barraquand2020stochastic}), or taken random under some distribution (see related work in \cite{emrah2016limit, krajenbrink2020tilted}).
\end{remark}
\subsubsection{Asymptotic result} \label{sec:BBPloggamma}
Going back to the log-gamma polymer from Section \ref{Section1.2} we consider a different scaling regime of the model. Let us first state the scalings and asymptotic result, and we will provide some explanations for this specific choice of scalings in Section \ref{sec:heuristic} below.

We fix $\rr, \rl \in \mathbb{Z}_{\geq 0}$ and assume $M, N \geq 1$. We fix $\vec{x} = (x_1, \dots, x_{\rr}) \in \mathbb{R}^{\rr}$ and $\vec{y} = (y_1, \dots, y_{\rl}) \in \mathbb{R}^{\rl}$ such that $\min(\vec{y}) > \max(\vec{x})$ (if $\rr = 0$ or $\rl = 0$ this condition is void). We  also fix $\theta > 0$, $p_{M} = N/M$ and $g, h_\theta, \sigma_\theta$ as in Section \ref{Section1.3}. With this choice of parameters we consider the following scaling
\begin{equation}\label{scaleParS1}
\begin{split}
&a_j(M,N) = g^{-1}(p_M) + x_j \sigma_\theta^{-1}(p_M) M^{-1/3} + o(M^{-1/3}) \mbox{ for $j = 1, \dots, \rr$}, \\
&\alpha_i(M,N) = g^{-1}(p_M) + y_i \sigma_\theta^{-1}(p_M)  M^{-1/3} + o(M^{-1/3}) \mbox{ for $i = 1, \dots, \rl$},\\
&a_j(M,N) = 0 \mbox{ for $j = \rr+1, \dots, N + \rr$} \hspace{2mm} \alpha_i(M,N) = \theta \mbox{ for $i = \rl+1, \dots,  M + \rl$},
\end{split}
\end{equation}
where we always assume that $M$ is sufficiently large so that $\alpha_i - a_j > 0$ for $j = 1, \dots, N + \rr$ and $i = 1, \dots, M + \rl$ (this is possible by our assumption that $\min(\vec{y}) > \max(\vec{x})$ and $p_M$ is bounded away from $0$ and $\infty$).

In this section we consider the one-point asymptotics of the rescaled free energy
\begin{equation}
\label{eq:rescaledpartitionfunctionV2}
\mathcal F(M + \rl, N + \rr) = \frac{\log Z^{M + \rl, N + \rr}+ M h_{\theta}(p_M)}{M^{1/3}\sigma_{\theta}(p_M)}.
\end{equation}
when $N,M \rightarrow \infty$ and $p_M = N/M$ remains bounded away from $0$ and $\infty$. We mention here that the distribution of $Z^{M + \rl, N + \rr}$ is invariant under separate permutations of the parameters $( \alpha_1, \dots,\alpha_{M + \rl}) $ and $(a_1, \dots, a_{N + \rr})$. This can be deduced from \cite{COSZ}, which relates the distribution of the partition function $Z^{M + \rl, N + \rr}$ to class-one $\mathfrak{gl}_n(\R)$-Whittaker functions or alternatively from Theorem \ref{LGPT1}. Thus in (\ref{scaleParS1}) we could have rescaled any $\rr$ (resp.  $\rl$) of the $a_j$ (resp. $\alpha_i$) parameters, and not necessarily those of lowest index,  without affecting the distribution of $Z^{M + \rl, N + \rr}$ and correspondingly $\mathcal F(M + \rl, N + \rr)$.

The below result about this scaling regime is proved in Section \ref{Section4.3}.
\begin{theorem}\label{BBP}
Let $\theta > 0$ and $\delta \in (0,1)$ and assume the same notation as in the beginning of the section. Let $\mathcal F(M + \rl, N + \rr)$ be as in  \eqref{eq:rescaledpartitionfunctionV2} and assume that we let $M,N$ go to infinity in such a way that the sequence $p_M=N/M \in [\delta, \delta^{-1}]$ for all sufficiently large $M$. Then for all $r\in \R$
\begin{equation}\label{BBPConv}
\lim_{M \rightarrow \infty} \mathbb{P} \left( \mathcal F(M + \rl, N + \rr)  \leq r\right) = F_{{ \rm BBP}; \vec{x}, \vec{y}}(r) ,
\end{equation}
where $F_{{ \rm BBP}; \vec{x}, \vec{y}}$ is as in Definition \ref{DefBBP}.
\end{theorem}
\begin{remark}\label{S1RemBBP}
The distribution of $\mathcal F(M,N)$ and $\mathcal F(N,M)$ are equal, up to changing the parameters $(\vec x, \vec y)$ into $(-\vec y, -\vec x)$. Furthermore, the BBP distribution is left invariant by such an exchange of parameters, as can be seen by performing the change of variables  $(v,w)\to(-w,-v)$ in \eqref{S1FDet2}. Hence,
as in Remark \ref{rem:symmetry}, it suffices to prove Theorem \ref{BBP} for $p_M\in [\delta, 1]$.
\end{remark}
\begin{remark}\label{RemBBP} Convergence to the one parameter BBP distribution (i.e. $F_{{ \rm BBP}; \vec{x}, \vec{y}}$ with $\rl = 0$) has been previously established for the O'Connell-Yor semi-discrete polymer in \cite{BCF}. In addition, the O'Connell-Yor semi-discrete polymer with vertical and horizontal boundaries is considered in \break \cite{BCFV, TV20}, and \cite{TV20} established convergence to the two parameter BBP distribution $F_{{ \rm BBP}; \vec{x}, \vec{y}}$. The starting Laplace transform Fredholm determinant formulas that are used in  \cite{BCFV,TV20}  are of a similar form as those that we use here. For the O'Connell-Yor  polymer, the core of the asymptotic analysis is based on understanding the behavior of the real part of $F_{\kappa}(z) = \log \Gamma(z) - \kappa z^2/2 + f_\kappa z$ along various contours, where $\kappa>0$ is a slope parameter and $f_{\kappa}$ is some appropriate constant.  In contrast, the core of our paper's asymptotic analysis is based on understanding the behavior of the real part of  $G_{\kappa, \theta} (z) =\kappa \Psi (z) -  \Psi (\theta - z) - g_{\theta, \kappa} z$, where $\Psi$ is the digamma function of (\ref{digammaS1}),  $\kappa>0$ is a slope parameter, $\theta>0$ is the parameter of the inverse gamma variables, and  $g_{\theta, \kappa}$ is some appropriately chosen constant. Unlike $F_{\kappa}(z)$, the function $G_{\kappa, \theta}(z)$ involves two parameters $(\kappa, \theta)$ and two digamma functions, which makes the study of its real part along different contours in Section \ref{Section9} quite involved.	
\end{remark}

\begin{remark}\label{RemBBP2}
The result in \cite{KQ} corresponds to setting $M = N$ and $\rr = \rl = 0$ in Theorem \ref{BBP}. In \cite{KQ} the authors dealt with the case $\kappa = 1$ of our formulas, which also significantly simplifies the analysis and they manage to use the same contours as in \cite{BCFV} since one gets nice cancellations coming from the fact that the coefficients in front of $\Psi(z)$ and $-\Psi(\theta - z)$ are the same in $G_{\kappa, \theta} (z)$. When $\kappa \neq 1$ a completely different choice of contours and analysis becomes necessary to obtain useful estimates on the real part of $G_{\kappa, \theta}(z)$.
\end{remark}

\subsection{Heuristic explanation for the scalings}
\label{sec:heuristic}
Let us now explain why we have scaled the parameters $a_j$ and $\alpha_i$ as in \eqref{scaleParS1} and provide some heuristic interpretations.  For simplicity of the discussion, assume that only one column has a different shape parameter, that is, we assume that $\rr=0, \rl=1$ and $\alpha_1 \in (0, \theta)$, such that the shape parameters of weights  are $\theta_{1,j}=\alpha_1<\theta$ along the fist column, and for all $i\geq 2, j\geq 1$, $\theta_{i,j}=\theta$. This means that the polymer weights $w_{i,j}$ will be typically larger on the first column, and polymer paths may spend more time there.

Using the law of large numbers for the free energy of the i.i.d. model (see \cite{Sep12} or Theorem \ref{HLConv} in our present work), and the fact that $\mathbb E\left[ \log w_{1,j} \right]=-\Psi(\alpha_1)$, one expects that for $N=\lfloor pM\rfloor$ the contribution to the partition function of paths which exit the first column near the location $(1, xM)$ to be approximately
$$\exp\Big( -M x\Psi(\alpha_1) - M h_{\theta}(p-x)\Big).$$
Thus, we expect the law of large numbers for the free energy will be given by
\begin{equation}
\lim_{M\to\infty} \frac{\log Z^{M+1, N} }{M}  = \max_{x\in (0,p)} \Big\lbrace - x\Psi(\alpha_1) - h_{\theta}(p-x)\Big\rbrace.
\label{eq:LLNGaussiancase}
\end{equation}
When $\alpha_1$ is large enough, the maximum in \eqref{eq:LLNGaussiancase} is achieved at $x=0$, but when  $\alpha_1<\theta_c$ for a certain critical value $\theta_c$, the maximum is attained at some $x>0$. The critical value  $\theta_c$ is such that
$$ -\Psi(\theta_c)  = -h'_{\theta}(p), \quad \textrm{or equivalently}\quad \theta_c=g^{-1}(p). $$
This argument can be extended to several rows and columns being perturbed, and leads to the scalings considered in \eqref{scaleParS1}.

%
%
%

When $\alpha_1<\theta_c$, we expect that polymer paths yielding the main contribution to the free energy will stay on the first column for a length $xM$, where $x$ is the maximizer in \eqref{eq:LLNGaussiancase}. The fluctuations of weights on the first column would dominate and this would result in Gaussian asymptotics. This could be proved through an asymptotic analysis of the Fredholm determinant as in the proof of Theorem \ref{BBP}, or via more probabilistic methods, but we leave this out of the scope of the paper.

When, however, $\alpha_1>\theta_c$, the first column being different will have no effect at large scale in the scalings and statistics (this can be rigorously deduced from our Theorems \ref{LGPCT} and \ref{BBP}).

Consider now the critical case and let $\alpha_1=g^{-1}(p)+ y_1 \sigma_\theta^{-1}(p_M) M^{-1/3}$ as in \ref{scaleParS1}. It is expected that paths yielding the main contribution to the partition function will exit the first column at a random vertical coordinate of order $M^{2/3}$ (this heuristic claim is partly justified by the results of \cite{Sep12} on the transversal exponent of paths under the quenched polymer measure). This $M^{2/3}$ is precisely the scale which balances the energy/entropy competition: the contribution to the free energy of the weights on the first column will have roughly the same variance as the contribution of bulk weights, the relative proportion being controlled by the parameter $y_1$. This results in BBP asymptotics interpolating between Gaussian and GUE Tracy-Widom statistics.

\begin{remark}
Using results from \cite{georgiou2013large} on the i.i.d model, there is another way to interpret the value of the critical point. Denote by $Z(m,n; M,N)$ the partition function between points $(m,n)$ and $(M,N)$ of the i.i.d. log-gamma polymer model as in Section \ref{Section1.3}. It is proved in \cite[Theorem 4.1]{georgiou2013large} that for any $n,m \geqslant 1$, the vertical increments ratios  $Z(m,n; M,N)/Z(m,n+1;M,N)$ almost surely converge, as $M$ and $N=\lfloor pM\rfloor$ go to infinity, to an inverse gamma random variable of parameter $g^{-1}(p)$ (while the horizontal increments ratios $Z(m,n; M,N)/Z(m+1, n;M,N)$ converge to inverse gamma random variables of parameter $\theta-g^{-1}(p)$). Hence, the critical value of $\alpha_1$ is precisely such that the weights on the first column have the same size as the vertical increments ratios in the i.i.d model.
\end{remark}

\subsection{Outline}

Section \ref{Section2} establishes the Fredholm determinant Laplace transform formula stated as Theorem \ref{LGPT1}. Section \ref{Section3} contains the asymptotic analysis of this Fredholm determinant formula. Section \ref{Section4} concludes with proofs of Theorems \ref{LGPCT}, \ref{S1LDE1}, \ref{BBP} and Lemma \ref{S1Lemma}.
Appendix \ref{Section9} records properties of the digamma function $\Psi$ and proves some technical lemmas about the function $G_{M,N}$ arising in Section \ref{Section3}. Finally, Appendix \ref{Section10} contains classic estimates for the gamma function and uses them to prove a number of estimates used in Section \ref{Section2} and Section \ref{Section3}.

\subsection*{Acknowledgments}
The authors wish to thank B\'alint Vir\'ag for discussions related to \cite{virag2020heat} and for pointing out the need to prove BBP asymptotics for the log-gamma polymer. E.D. was partially supported by a Minerva Foundation Fellowship.  I. C. was partially supported by the NSF grants DMS:1811143 and DMS:1664650 and the Packard  Foundation  through  a  Packard  Fellowship  for  Science  and Engineering.

\section{Laplace transform formula}\label{Section2}
The goal of this section is to obtain a formula for the Laplace transform $\mathbb{E}\big[\exp(-u\ZP)\big]$ of the partition function for the log-gamma polymer which is {\em suitable} for asymptotic analysis. This result is given as Theorem \ref{LGPT1} and  expresses $\mathbb{E}\big[\exp(-u\ZP)\big]$ in terms of a Fredholm determinant (see Section \ref{Section2.1} for a definition and basic properties of the latter).
To be suitable for asymptotic analysis means that the contour for the $L^2$ space of the Fredholm determinant, as well as the contour used in defining the kernel (via a contour integral) should be steep descent contours for the integrand functions in the scaling $N, M \rightarrow \infty$ that we eventually perform.

There are two instances prior to this paper of Fredholm determinant formulas for $\mathbb{E}\big[\exp(-u\ZP)\big]$, which have yielded asymptotics. The first is in \cite{BCR} where a {\em small} contour integral formula was derived under the assumption that $a_i$ are sufficiently small and positive. The asymptotic analysis in \cite{BCR} was performed when $\alpha_i$ are all equal to a fixed $\theta > 0$ and $a_i = 0$, under the assumption that $\theta$ is sufficiently close to zero. The second is in \cite{KQ}, where starting from the formulas in \cite{BCFV} the authors were able to derive a {\em large} contour integral formula that allows one to perform the asymptotic analysis when $\alpha_i = \theta$ and $a_i = 0$ without any restriction on $\theta$. We mention that the asymptotic analysis in both \cite{BCR} and \cite{KQ} has only been performed when $M = N$.

It turns out that neither of the formulas in \cite{BCR} and \cite{KQ} are very good for asymptotics when $M \neq N$ and $\theta > 0$ is arbitrary. The small contour formulas of \cite{BCR} are immediately inapplicable as the contours only exist under the assumption that $\theta$ is small enough, which leaves the large contour formulas from \cite{KQ, BCFV}. When $M = N$ the integrand functions in these formulas enjoy a certain symmetry, that dramatically simplifies their analysis and allows one to establish that they behave nicely along certain vertical contours, which are suitable for asymptotics. The approach in \cite{KQ} is to work with these vertical contours.  When $M \neq N$ the behavior of the integrand functions along vertical contours becomes unclear -- we believe (although we haven't proved this) that the vertical lines are no longer descent contours. On this account, for the general $M \neq N$ case we needed to develop a formula with different contour choices.

We discovered that the kernel in the Fredholm determinant formula for $\mathbb{E}\big[\exp(-u\ZP)\big]$ behaves nicely along lines of slope $\pm 1$ and this works for any $M, N \geq 1$. The necessary estimates of the kernel along these contours will be introduced later in Section \ref{Section3} as Lemmas \ref{analGS3}, \ref{ReleL1S3}, \ref{horGS3} and Proposition \ref{PKBS3}, whose proofs take up the majority of Section \ref{Section9} and Section \ref{Section10.3}.\\

Once we have discovered a suitable choice of contours for our asymptotic analysis we face the immediate obstacle of ensuring that $\mathbb{E}\big[\exp(-u\ZP)\big]$ can be written as a Fredholm determinant involving precisely these contours -- there are several problems in doing this and the purpose of this section is to overcome them all. We elaborate on these problems below.

Firstly, there is no generic formula for $\mathbb{E}\big[\exp(-u\ZP)\big]$ and instead one is forced to start with the formulas in \cite{BCFV} that are for a certain polymer model that mixes Sepp{\" a}l{\"a}inen's log-gamma polymer \cite{Sep12} and the O'Connell-Yor semi-discrete polymer \cite{OCY}. This mixed polymer is a one-parameter generalization of the log-gamma polymer, which in addition to the $\alpha_i$'s and $a_j$'s depends on a positive parameter $\tau$. This model is recalled in Section \ref{Section2.2}, where it is also established that its partition function $\ZP(\tau)$ converges weakly in the $\tau \rightarrow 0+$ limit to the partition function $\ZP$ of the log-gamma polymer. The benefit of working with $\ZP(\tau)$ is that in \cite{BCFV} the authors established a general formula for its Laplace transform $\mathbb{E}\big[\exp(-u\ZP(\tau))\big]$. The idea then is to start from the formula for $\mathbb{E}\big[\exp(-u\ZP(\tau))\big]$, deform the contours in that formula to the suitable contours we want in our case and send $\tau$ to $0+$. Since the functions that are being integrated over these contours are meromorphic, one has a certain level of flexibility in deforming one set of contours to another that {\em is} suitable; however, as we are dealing with infinite contours, this necessitates that we justify the deformation near infinity.  In addition, sending $\tau$ to $0+$ requires certain estimates near infinity that would allow the appropriate application of the dominated convergence theorem. Obtaining estimates near infinity is a second problem we face, which is resolved with the technical Propositions \ref{PE1v2}, \ref{PE2v2} and \ref{PE3v2}.

A third problem we face is that the formula for $\mathbb{E}\big[\exp(-u\ZP(\tau))\big]$ in \cite{BCFV} is only correct under the additional assumption that $\min(\vec{\alpha}) - \max(\vec{a}) > 1$ while we are interested in the fully general case of $\min(\vec{\alpha}) - \max(\vec{a}) > 0$. We recall the formula from \cite{BCFV} for the correct parameter range in Proposition \ref{BCT1} and explain the issues with this formula for general parameters in Remark \ref{RemCont}. In order to fix the third problem we derive a new formula for $\mathbb{E}\big[\exp(-u\ZP(\tau))\big]$ in Theorem \ref{BCT1v2}, which holds in the general case $\min(\vec{\alpha}) - \max(\vec{a}) > 0$. The way Theorem \ref{BCT1v2} is proved is fairly involved and takes up the majority of Section \ref{Section2.4}. The basic idea of the proof, which will be elaborated significantly later, is to start from Proposition \ref{BCT1}, which works when $\min(\vec{\alpha}) - \max(\vec{a}) > 1$ and do the deformation of the contours for that range. Then one performs an analytic continuation argument to extend the validity of the result to the full range of $\min(\vec{\alpha}) - \max(\vec{a}) > 0$. As part of the proof, we required a statement detailing when a Fredholm determinant depending on a complex parameter $z$ is analytic in that parameter -- see Lemma \ref{AT}.

The above few paragraphs give a sketch of the approach we will be taking in this section. The argument is fairly involved, but the guiding principle is that we want to start from the formulas for $\mathbb{E}\big[\exp(-u\ZP(\tau))\big]$ in Proposition \ref{BCT1}, deform the contours in that formula, perform an analytic continuation in a suitable parameter and then send $\tau \rightarrow 0+$. Although simple in appearance, this sequence of steps requires one to obtain fairly detailed estimates of the integrand functions in our formulas and constantly be careful in avoiding the large number of poles coming from gamma and sine functions that are part of these functions.\\

%
\subsection{Background on Fredholm determinants}\label{Section2.1}

We present a brief background on Fredholm determinants and state and prove a few key technical lemmas that we will use throughout this section. A reader may skip this background material and refer back to it when we begin to manipulate Fredholm determinants in the later parts of this section.
 For a general overview of the theory of Fredholm determinants, the reader is referred to \cite{Simon} and \cite{Lax}. For our purposes the definition below is sufficient and we will not require additional properties.

\begin{definition}
Fix a Hilbert space $L^2(X, \mu)$, where $X$ is a measure space and $\mu$ is a measure on $X$. When $X = \Gamma$, a simple (oriented) piecewise smooth contour in $\mathbb{C}$, we write $L^2(\Gamma)$ where for $z \in \Gamma$, $d\mu(z)$ is understood to be $\frac{dz}{2\pi \i }$. Here and throughout the paper, we denote the imaginary unit $\i = \sqrt{-1}$, to distinguish it from $i$ which will often be used as an index parameter.
\end{definition}
Let $K$ be an {\em integral operator} acting on $f\in L^2(X,\mu)$ by $Kf(x) = \int_XK(x,y)f(y)d\mu(y)$.  $K(x,y)$ is called the {\em kernel} of $K$ and we will usually assume that $K(x,y)$ is continuous in both $x$ and $y$. If $K$ is a {\em trace-class} operator then one defines the Fredholm determinant of $I + K$, where $I$ is the identity operator, via
\begin{equation}\label{fredholmDefS2}
\det(I + K)_{L^2(X)} := 1 + \sum_{n = 1}^\infty \frac{1}{n!} \int_X \cdots \int_X \det \left[ K(x_i, x_j)\right]_{i,j = 1}^n \prod_{i = 1}^nd\mu(x_i),
\end{equation}
where the latter sum can be shown to be absolutely convergent (see \cite{Simon}).

The expression appearing on the RHS of (\ref{fredholmDefS2}) can be absolutely convergent even if $K$ is not trace-class. Whenever $X$ and $K$ are such that the RHS in (\ref{fredholmDefS2}) is absolutely convergent, we will still call it $\det (I + K)_{L^2(X)}$. The latter is no longer a Fredholm determinant, but some numeric quantity we attach to the kernel $K$. Of course, if $K$ is the kernel of a trace-class operator on $L^2(X)$ this numeric quantity agrees with the Fredholm determinant. Doing this allows us to work on the level of numbers throughout most of the paper, and avoid constantly checking if the kernels we use represent a trace-class operator.

In some of our applications we will have that the kernel $K(x,y)$ depends in addition on a complex parameter $z$, and then we write $K_z(x,y)$ to reflect this dependence. A basic question which will be important in our analysis is when can one deduce that $\det(I + K_z)_{L^2(X)}$ is not just well-defined (meaning that each summand in (\ref{fredholmDefS2}) is well-defined and the sum is absolutely convergent) but also analytic in $z$. We provide sufficient conditions for this to hold in Lemma \ref{AT} below. Our proof of that lemma uses the following result.
\begin{proposition}\cite[Theorem 1]{Mattner}\label{Mattner}
Let $(\Omega, \mathcal{A}, \mu)$ be a measure space, $G \subset \mathbb{C}$ an open set and $\mathcal{B}(G)$ the $\sigma$-algebra generated by all open subsets of $G$. Suppose that $f: G \times \Omega \rightarrow \mathbb{C}$ is a function that satisfies the following assumptions:
\begin{enumerate}
\item $f(z, \cdot)$ is $\mathcal{A}$-measurable for every $z \in G$;
\item $f(\cdot, \omega)$ is holomorphic in $G$ for every $\omega \in \Omega$;
\item $\int_{\Omega} |f(\cdot, w)| d\mu(\omega)$ is locally bounded, that is, for all $z_0\in G$ there exists $\delta>0$ such that
$$ \sup_{z \in G, |z-z_0| \leq \delta} \int_\Omega |f(z,w)|d\mu(\omega) < \infty.$$
\end{enumerate}
Then $\int_{\Omega}f(z, \omega) d\mu(\omega)$ is holomorphic in $G$ and may be differentiated under the integral. More precisely, we have for any $n \geq 0$ the following conclusions:
\begin{enumerate}
\item $ \partial_z^n f$ is $\mathcal{B}(G) \otimes \mathcal{A}$-measurable and, for every $A \subset G$, $\sup_{z \in A}|\partial_z^n f(z, \cdot)|$ is $\mathcal{A}$-measurable;
\item for all compact $K\subset G$
$$\int_{\Omega}\, \sup_{z \in K}  \left| \partial_z^n f(z,\omega) \right|d\mu(\omega) < \infty;$$
\item $\int_{\Omega} f(z, \omega) d\mu(\omega)$ is holomorphic in $G$ and for each $z\in G$,
$$\partial_z^n \int_{\Omega} f(z,\omega) d\mu(\omega) = \int_{\Omega} \partial_z^n f(z,\omega)d\mu(\omega).$$
\end{enumerate}
\end{proposition}

\begin{lemma}\label{AT}
Suppose that $\Gamma$ is a piecewise smooth contour in $\mathbb{C}$ and $G \subset \mathbb{C}$ is open. Suppose that $K_z(x,y)$ is a function such that
\begin{enumerate}
\item $K_z(\cdot, \cdot)$ is a measurable function on  $\Gamma \times \Gamma$ for each $z \in G$;
\item for each $x,y \in \Gamma$ the function $K_z(x, y)$ is analytic in $z$ on $G$.
\end{enumerate}
In addition, suppose that for every $z_0 \in G$ there exist a $\delta > 0$ and a measurable function $F(x)$ on $\Gamma$ (depending on $z_0$ and $\delta$) such that
\begin{equation}\label{supBoundK}
\sup_{z \in G, |z-z_0| < \delta}\,\,\sup_{y\in \Gamma}|K_z(x,y)| \leq F(x)\,\,\, \mbox{ and } \,\,\,\int_{\Gamma}F(x) |d\mu(x)| = M < \infty.
\end{equation}
Then for each $z \in G$ one has that $\det \left[ K_z(x_i, x_j)\right]_{i,j = 1}^n$ is integrable on $\Gamma^n$, so that in particular $\int_{\Gamma} \cdots \int_{\Gamma} \det \left[ K_z(x_i, x_j)\right]_{i,j = 1}^n \prod_{i = 1}^nd\mu(x_i)$ is well defined. Moreover, for $z \in G$ the series $\det(I + K_z)_{L^2(\Gamma)}$ in (\ref{fredholmDefS2})
is absolutely convergent and defines a holomorphic function of $z$ in $G$.
\end{lemma}
\begin{proof}
Let $z_0 \in G$ and $\delta > 0$ be sufficiently small so that $N_\delta(z_0) := \{ z\in \mathbb{C}: |z - z_0| < \delta\}$ is contained in $G$ and the condition in (\ref{supBoundK}) is satisfied for some function $F$. In particular, we have by Hadamard's inequality
that if $x_1, \dots, x_n \in \Gamma$, then
$$ \sup_{z \in G, |z-z_0| < \delta} \left| \det \left[ K_z(x_i, x_j)\right]_{i,j = 1}^n\right| \leq n^{n/2} \prod_{i = 1}^nF(x_i),$$
and the latter is integrable by assumption. It follows from Theorem \ref{Mattner} that
$$B_n(z): = \int_{\Gamma} \cdots \int_{\Gamma} \det \left[ K_z(x_i, x_j)\right]_{i,j = 1}^n \prod_{i = 1}^nd\mu(x_i)$$
is holomrphic in $N_\delta(z_0)$ and by (\ref{supBoundK}) $|B_n(z)| \leq n^{n/2} \cdot M^n$ for $z \in N_\delta(z_0)$. We conclude that
$$\det(I + K_z)_{L^2(\Gamma)} = 1 + \sum_{n = 1}^\infty \frac{B_n(z)}{n!}$$
is absolutely convergent, and hence also holomorphic in $N_\delta(z_0)$ as the uniform limit of holomorphic functions, cf. \cite[Chapter 2, Theorem 5.2]{Stein}. Since $z_0 \in G$ was arbitrary we conclude that $\det(I + K_z)_{L^2(\Gamma)}$ is holomorphic in $G$.
\end{proof}

The following lemmas provide a framework for proving convergence of Fredholm determinants, based on pointwise convergence of their defining kernels and estimates on those kernels. The proofs can be found in \cite[Section 2.5]{ED}.
\begin{lemma}\label{FDCT}
Suppose that $\Gamma$ is a piecewise smooth contour in $\mathbb{C}$ and $K^N(x,y)$, for $N \in \mathbb{N}\cup \{\infty\}$, are measurable kernels on $\Gamma \times \Gamma$ such that  $\lim_{N \rightarrow \infty} K^N(x,y) = K^\infty(x,y)$ for all $x,y\in \Gamma$. In addition, suppose that there exists a non-negative, measurable function $F(x)$ on $\Gamma$ such that
$$\sup_{N \in \mathbb{N}}\sup_{y\in \Gamma}|K^N(x,y)| \leq F(x) \mbox{ and } \int_{\Gamma}F(x) |d\mu(x)| = M < \infty.$$
Then, for each $n \geq 1$ and $N \in \mathbb{N}\cup \{\infty\}$, one has that $\det \left[ K^N(x_i, x_j)\right]_{i,j = 1}^n$ is integrable on $\Gamma^n$, so that in particular $\int_{\Gamma} \cdots \int_{\Gamma} \det \left[ K^N(x_i, x_j)\right]_{i,j = 1}^n \prod_{i = 1}^nd\mu(x_i)$ is well defined. Moreover, for each $N \in \mathbb{N}\cup \{\infty\}$ the series
$\det(I + K^N)_{L^2(\Gamma)}$ in (\ref{fredholmDefS2}) is absolutely convergent and $\lim_{N \rightarrow \infty} \det(I + K^N)_{L^2(\Gamma)} = \det(I + K^\infty)_{L^2(\Gamma)}.$
\end{lemma}

\begin{lemma}\label{FDKCT}
Suppose that $\Gamma_1, \Gamma_2$ are piecewise smooth contours and the function $(x,y,z)\mapsto g^N_{x,y}(z)$ is measurable for  $(x,y,z)\in \Gamma_1\times \Gamma_1 \times \Gamma_2$ for all $N \in \mathbb{N}\cup \{\infty\}$ and satisfies $\lim_{N \rightarrow \infty} g^N_{x,y}(z) = g_{x,y}^\infty(z)$ for all $(x,y,z)\in \Gamma_1\times \Gamma_1 \times \Gamma_2$. In addition, suppose that there exist bounded non-negative measurable functions $F_1$ and $F_2$ on $\Gamma_1$ and $\Gamma_2$ respectively such that
$$\sup_{N \in \mathbb{N}}\,\sup_{y\in \Gamma_1}|g^N_{x,y}(z)| \leq F_1(x)F_2(z), \,\,\,\mbox{ and }\,\,\, \int_{\Gamma_i}F_i(u)|d\mu(u)| = M_i < \infty \,\,\,\mbox{for }i=1,2.$$
Then for each $N \in \mathbb{N}\cup \{\infty\}$ one has $\int_{\Gamma_2}|g^N_{x,y}(z)| |d\mu(z)|< \infty$ and in particular $K^N(x,y):= \int_{\Gamma_2}g^N_{x,y}(z) d\mu(z)$ are well-defined. Moreover, $K^N(x,y)$ satisfy the conditions of Lemma \ref{FDCT} with $\Gamma = \Gamma_1$ and $F =  F_1 M_2$.
\end{lemma}

%
\subsection{Mixed polymer model}\label{Section2.2}

In \cite{BCFV} the authors considered a  polymer model that mixes Sepp{\" a}l{\"a}inen's log-gamma polymer \cite{Sep12} and the O'Connell-Yor semi-discrete polymer \cite{OCY}. In the language of Whittaker processes \cite{BorCor}, which we will not use or rely directly upon, this model arises as a mixture of the Plancherel and pure alpha specializations. The model depends on a parameter $\tau > 0$ as well as two sets of real parameters $\{a_i\}_{i \geq 1}$ and $\{\alpha_i\}_{i \geq 1}$, which satisfy $\alpha_m > 0$ and $\alpha_m - a_n > 0$ for all $m,n \in \mathbb{N}$.

Before we go into the definition let us briefly explain our interest in this model. As mentioned above, the model depends on parameters $\{a_i\}_{i \geq 1}, \{\alpha_i\}_{i \geq 1}$ and $\tau > 0$ and for each $M, N \in \mathbb{N}$ gives rise to a random quantity $\ZP(\tau)$, called the {\em partition function}, which depends only on $\{a_i\}_{i = 1}^N$, $\{\alpha_i\}_{i= 1}^M$ and $\tau$. As we explain around (\ref{PPAS}), this partition function converges weakly as $\tau \rightarrow 0+$ to the partition function $\ZP$ for the log-gamma polymer with parameters $\{a_i\}_{i = 1}^N, \{\alpha_i\}_{i = 1}^M$ (recall from Section \ref{Section1.2} or see below). In \cite{BCFV} the authors were able to find a formula for the Laplace transform for $\ZP(\tau)$, recalled below as Theorem \ref{BCT1}. We are ultimately interested in obtaining a formula for the Laplace transform for $\ZP$ which is suitable for asymptotic analysis as $N$ and $M$ grow. The way we go about this is by rewriting the known formula the Laplace transform of $\ZP(\tau)$ in a manner which is then amenable to the limit $\tau \rightarrow 0+$. This rewriting and $\tau \rightarrow 0+$ limit transition are performed in Sections \ref{Section2.3} and \ref{Section2.4} below. \\

\begin{definition}\label{Dmixed}
The mixed polymer model is a probability distribution on up-right paths connecting points $(-M, 1)$ and $(\tau, N)$. These paths consist of two parts: a discrete portion $\phi^d$ and an adjoined semi-discrete portion $\phi^{sd}$. A {\em discrete up-right path} $\phi^d$ from $(i_1, j_1)$ to $(i_\ell, j_\ell)$ (written as $\phi^d: (i_1, j_1) \nearrow (i_{\ell}, j_{\ell})$) is an ordered set of points $((i_1, j_1), (i_2, j_2), \dots, (i_{\ell}, j_{\ell}))$ with each $(i_k, j_k) \in \mathbb{Z}^2$ and each increment $(i_k, j_k) - (i_{k-1}, j_{k-1}) \in \{ (1,0), (0,1)\}$. A {\em semi-discrete up-right path} $\phi^{sd}$ from $(0,n)$ to $(\tau, N)$ (written as $\phi^{sd}: (0,n) \nearrow (\tau,N)$) is a union of horizontal line segments $((0,n) \rightarrow (s_n, n)) \cup ((s_n, n+1) \rightarrow (s_{n+1}, n+1)) \cup \cdots \cup ((s_{N-1}, N) \rightarrow (\tau,N))$ where $0 \leq s_n < s_{n+1} < \cdots < s_{N-1} \leq \tau$.

\begin{figure}[h]
\centering
\scalebox{1.1}{\includegraphics{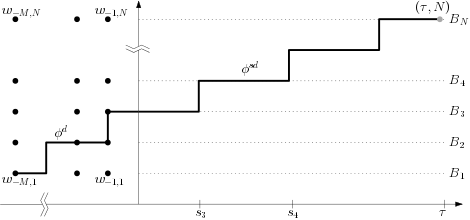}}
 \captionsetup{width=.9\linewidth}
\caption{An example of a mixed polymer path.}
\label{Fig:mixedpath}
\end{figure}

The state space of the model consists of pairs $\phi = (\phi^{d}, \phi^{sd})$ such that $\phi^{d}$ is a discrete up-right path and $\phi^{sd}$ is a semi-discrete up-right path, which are adjoined to each other in such a way that for some $1 \leq n \leq N$ we have $\phi^d : (-M, 1) \nearrow (-1,n)$ and $\phi^{sd} : (0,n) \nearrow (\tau, N)$.

We turn to defining the polymer probability measure on this space, along with its partition function. Assume that $w_{i,j}$ for $-M \leq i \leq -1$ and $1 \leq j \leq N$ are independent random variables such that $w_{i,j}$ have the same distribution as $\log X_{i,j}$ with $X_{i,j}$ inverse-gamma distributed as in (\ref{S1invGammaDens}) with parameter $\theta = \alpha_{M + i + 1} - a_j > 0$. As before we have $\vec{\alpha} = (\alpha_1, \dots, \alpha_M) \in \mathbb{R}_{> 0}^M$ and $\vec{a} = (a_1, \dots, a_N) \in \mathbb{R}^N$. In addition, assume that $B_n(t) = a_nt + W_n(t)$ for $n = 1, \dots, N$, where $W_n(t)$ are i.i.d. Brownian motions, which are independent of the $w_{i,j}$. To each path $\phi = (\phi^{d}, \phi^{sd})$ we associate the energy
\begin{equation}\label{PathEnergy}
E(\phi) := \sum_{(i,j) \in \phi^d} w_{i,j} + B_n(s_n) + \sum_{k = n+1}^N \big(B_{k}(s_{k}) - B_{k}(s_{k-1})\big).
\end{equation}
This energy is random, as it is a function of $w_{i,j}$ and the $B_k$'s. We associate a Boltzmann weight $e^{E(\phi)}$ to each path $\phi$. The polymer measure on $\phi$ is then proportional to this weight. The normalizing constant, or polymer partition function, is written as $\ZP(\tau)$ and is equal to
\begin{equation}\label{PathEnergy2}
\ZP(\tau) = \sum_{n = 1}^N\, \sum_{\phi^d:(-M,1) \nearrow (-1,n)} \int_{\phi^{sd}:(0,n) \nearrow (\tau, N)} e^{E(\phi)}d \phi^{sd},
\end{equation}
where $d\phi^{sd}$ is Lebesgue measure on the simplex $0 \leq s_n < \cdots < s_{N-1} \leq \tau$, identified with $\phi^{sd}$. \\
We remark that in \cite{BCFV} the partition function was denoted by $Z^{N,M}(\tau)$, and we have opted to swap the roles of $M$ and $N$ to make the notation closer to the one from Section \ref{Section1.2} -- apart from this change we follow the same notation as in \cite{BCFV}. 
\end{definition}

We next recall \cite[Theorem 2.1]{BCFV}, which is the main ingredient we need in deriving our Laplace transform formula for $\ZP$. In order to state the theorem we need a few definitions. Below we write $\max(\vec{x}) = \max(x_1, \dots, x_n)$ and $\min(\vec{x}) = \min(x_1, \dots, x_n)$ for $\vec{x} = (x_1, \dots, x_n) \in \mathbb{R}^n$.

\begin{definition}\label{DCont1} Let $\vec{a} = (a_1, \dots, a_N) \in \mathbb{R}^N$, $\vec{\alpha}= (\alpha_1, \dots, \alpha_M) \in \mathbb{R}_{>0}^M$ and suppose that $\min(\vec{\alpha}) - \max(\vec{a}) > 0$. Set $\mu = \frac{1}{2} \max(\vec{a}) + \frac{1}{2} \min(\vec{\alpha})$ and $\eta = \frac{1}{4} \max(\vec{a}) + \frac{3}{4} \min(\vec{\alpha})$ and fix $\phi \in [0, \pi/4)$. With the latter data we let we let $\mathcal{C}_{\vec{a}; \vec{\alpha}; \phi}$ denote the union of the contour $\{\mu + ye^{\i(\pi +\phi)} \}_{y \in \mathbb{R}^+}$ and $\{\mu + ye^{\i(\pi - \phi)} \}_{y \in \mathbb{R}^+}$ oriented so as to have increasing imaginary part. For any $v \in \mathcal{C}_{a; \alpha; \phi}$ we choose $R = - \Re(v) + \eta$, $d > 0$ and define a contour $\mathcal{D}_v$ as follows: $\mathcal{D}_v$ goes by straight lines from $R - \i \infty$ to $R - \i d$, to $1/2 - \i d$, to $1/2 + \i d$, to $R + \i d$, to $R + \i \infty$. The parameter $d$ is sufficiently small so that $v + \mathcal{D}_v$ does not intersect $\mathcal{C}_{\vec{a}; \vec{\alpha}; \phi}$. If $\phi \in [\pi/6, \pi/4]$ we may take any $d \in (0, 1/4]$. See Figure \ref{S2_2}.
\end{definition}

\begin{figure}[h]
\centering
\scalebox{0.6}{\includegraphics{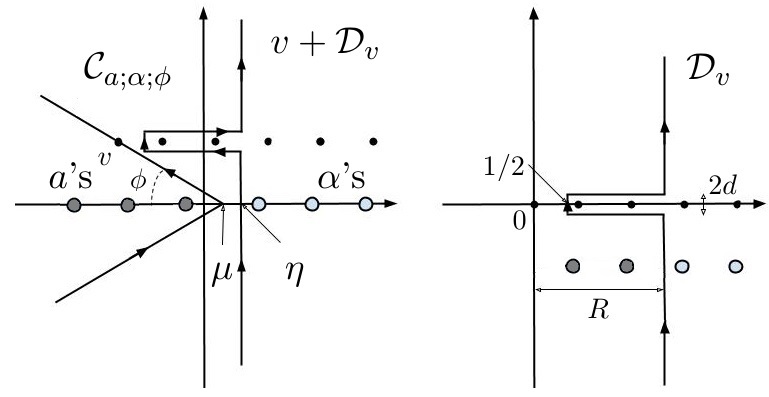}}
 \captionsetup{width=.9\linewidth}
\caption{The left part depicts the contours $\mathcal{C}_{\vec{a}; \vec{\alpha}; \phi}$ and $v+ \mathcal{D}_v$ for $v \in \mathcal{C}_{\vec{a}; \vec{\alpha}; \phi}$. The black dots denote the points $v, v+1, v+2, \dots$, the light gray ones denote $\{\alpha_i + m: i = 1, \dots, M, m \in \mathbb{Z}_{\geq 0}\}$ and the dark gray ones $\{a_i - m: i = 1, \dots, N, m \in \mathbb{Z}_{\geq 0}\}$. The right part depicts the contour $\mathcal{D}_v$. }
\label{S2_2}
\end{figure}

\begin{proposition}\label{BCT1}\cite[Theorem 2.1]{BCFV}
Fix integers $N \geq 9, M \geq 1$, a real number $\tau > 0$ and $\vec{a} = (a_1, \dots, a_N) \in \mathbb{R}^N$, $\vec{\alpha}= (\alpha_1, \dots, \alpha_M) \in \mathbb{R}_{>0}^M$. Assume that $\min(\vec{\alpha}) - \max(\vec{a}) > 1$. For all $u \in \mathbb{C}$ with $\Re(u)> 0$ we have
\begin{equation}\label{BCTe1}
\mathbb{E} \left[ e^{-u \ZP(\tau)} \right] = \det \left( I + \mathcal{K}_u \right)_{L^2(\mathcal{C}_{\vec{a};\vec{\alpha};\phi})},
\end{equation}
where the operator $\mathcal{K}_u$ is defined in terms of its integral kernel
\begin{equation}\label{kernel1}
\mathcal{K}_u(v,v') = \frac{1}{2\pi \i} \int_{\mathcal{D}_v}  \Gamma(-s) \Gamma(1+s)  \prod_{n = 1}^N \frac{\Gamma(v - a_n)}{\Gamma(s + v - a_n)} \prod_{m = 1}^M \frac{\Gamma(\alpha_m - v - s)}{\Gamma(\alpha_m - v)}\frac{u^s e^{v\tau s+ \tau s^2/2}ds}{v+ s - v'}.
\end{equation}
The contours $\mathcal{C}_{\vec{a};\vec{\alpha}; \phi}$ and $\mathcal{D}_v$ are as in Definition \ref{DCont1} with $\phi = \pi/6$ and $d = 1/4$.
\end{proposition}
\begin{remark}\label{RemCont} Proposition \ref{BCT1} is formulated precisely as \cite[Theorem 2.1]{BCFV}, except that in \cite[Theorem 2.1]{BCFV} the authors assumed that $\min(\vec{\alpha}) - \max(\vec{a}) > 0$, whereas we have assumed $\min(\vec{\alpha}) - \max(\vec{a}) > 1$. This strengthened assumption is necessary in order for the result to be valid with the prescribed choice of contours. Specifically, in order for the result to hold one needs to pick contours in such a way that the poles of the functions $\Gamma(\alpha_i - v -s)$ lie strictly to the right of the contour $\mathcal{D}_v$. If $v = \mu$ the latter function has a pole at $\min(\vec{\alpha}) - \mu$ and by the definition of $\mathcal{D}_v$ we must have that $\min(\vec{\alpha}) - \mu > \frac{1}{2}$, which upon replacing $\mu = \frac{1}{2} \max(\vec{a}) + \frac{1}{2} \min (\vec{\alpha})$ implies we must have $\min(\vec{\alpha}) - \max(\vec{a}) > 1$. We believe that the authors of \cite{BCFV} overlooked this point when making their choice of contours and the issue remained unnoticed in \cite{TV20}; however the results of the two papers hold if one assumes $\min(\vec{\alpha}) - \max(\vec{a}) > 1$. We further remark that this problem was previously observed in \cite[Theorem A.1]{KQ}, where the authors proposed a different choice of contours in order to fix the issue. The authors of \cite{KQ} worked exclusively in the homogeneous case when $\alpha_i$ are all equal to the same number $\theta > 0$ and $a_i = 0$. In this case the restriction becomes $\theta > 1$. It is worth mentioning that the choice of contours in \cite[Theorem A.1]{KQ} fixes the problem for $\theta \in (0,1]$ but creates a new problem when $\theta \geq 8$. In particular, the choice of contours in \cite[Theorem A.1]{KQ} only works when $\theta \in (0, 8)$ since unfortunately the authours failed to take into account the poles coming from $ \Gamma(-s) \Gamma(1+s) $ at $s = 1, 2, \dots,$ which all need to be to the right of $\mathcal{D}_v$.
\end{remark}

We end this section by showing that the partition function $\ZP(\tau)$ converges weakly to $\ZP$ from Section \ref{Section1.2} as $\tau \rightarrow 0+$. Let us define $Z^{M,n}(0)$ by
\begin{equation}\label{PP1}
Z^{M,n}(0)=  \sum\limits_{\phi^d:(-M,1) \nearrow (-1,n)}  e^{E(\phi^d)},
\end{equation}
and observe that this random variable has the same distribution as the log-gamma partition function $Z^{M,n}$ with parameters $\{a_i\}_{i = 1}^n$, $\{\alpha_i\}_{i= 1}^M$ from (\ref{PartitionFunct}). We claim that
\begin{equation}\label{PPAS}
\lim_{\tau \rightarrow 0+} \ZP(\tau)=  \ZP(0) \mbox{ almost surely}.
\end{equation}
To see this let $E_K$ denote the event
$$E_K = \Big\{ |B_k(s)| \leq K \mbox{ for $s \in [0,1], k \in \{ 1, \dots, N \}$ and } \lim_{s \rightarrow 0+} B_N(s) = 0 \Big\}.$$
Then on $E_K$, as the volume of the simplex $0 \leq s_n < s_{n+1} < \cdots < s_{N-1} \leq \tau$ is at most $\tau^{N-n}$,
\begin{align*}
\limsup_{\tau \rightarrow 0+} \ZP(\tau) &\leq  \limsup_{\tau \rightarrow 0+} \sum_{n = 1}^{N-1}  Z^{M,n}(0) \cdot \tau^{N-n} e^{(2N + 1)K} + \ZP(0) = \ZP(0),\\
\liminf_{\tau \rightarrow 0+} \ZP(\tau) &\geq  \liminf_{\tau \rightarrow 0+} \ZP(0) e^{B_N(\tau)} =  \ZP(0).
\end{align*}
Since $\cup_{K \in \mathbb{N}} E_K$ is a set of full probability we conclude (\ref{PPAS}).

%

\subsection{Laplace transform for the log-gamma polymer}\label{Section2.3}
In this section we derive as Theorem \ref{LGPT1} a formula for the Laplace transform of $\ZP$ from (\ref{PartitionFunct}), which is suitable for asymptotic analysis in the sense we described earlier. We continue with the same notation as in Section \ref{Section2.2} and also introduce the following notation, see also Figure \ref{S2_1}.
\begin{figure}[h]
\scalebox{0.6}{\includegraphics{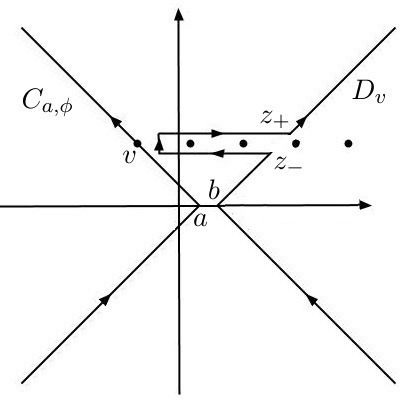}}
\captionsetup{width=\linewidth}
 \caption{The figure depicts the contour $C_{a, \phi}$ for $\phi = 3 \pi/ 4$ and $D_v(b,\pi/4, d)$ with $a, b \in \mathbb{R}$ and $b > a$. The black dots denote the points $v, v+1, v+2, \dots$. }
\label{S2_1}
\end{figure}

\begin{definition}\label{ContV}
For $a \in \mathbb{C}$ and $\phi \in (0, \pi)$ we define the contour $C_{a,\phi}$ to be the union of $\{a + ye^{-\i\phi)} \}_{y \in \mathbb{R}^+}$ and $\{a + ye^{\i \phi} \}_{y \in \mathbb{R}^+}$ oriented to have increasing imaginary part.
\end{definition}

\begin{definition}\label{ContE}
Suppose that $a \in \mathbb{C}$, $\phi \in (0, \pi)$, $d > 0$ and $v \in \mathbb{C}$. From this data we construct a contour that consists of two parts. The first, called $D_v^1$, is $C_{a,\phi} \setminus \{z: \Im(z) \in [\Im(v) - d, \Im (v) + d] \}.$ Let $z_-, z_+$ be the points on $C_{a,\phi} $ that have imaginary parts $\Im(v) - d$ and $\Im(v) + d$ respectively. The second part of the contour, called $D_v^2$, consists of straight oriented segments that connect $z_-$ to $v + 2d- \i d$ to $v + 2d + \i d$ to $z_+$. We will denote the resulting contour by $D_{v}(a,\phi, d)$ or just $D_v$ when the other parameters are clear from context.
\end{definition}

The goal of this section is to establish the following result.
\begin{theorem}\label{LGPT1}
Fix integers $N \geq 9, M \geq 1$ and $\vec{a} = (a_1, \dots, a_N) \in \mathbb{R}^N$, $\vec{\alpha}= (\alpha_1, \dots, \alpha_M) \in \mathbb{R}_{>0}^M$ such that $\min(\vec{\alpha}) - \max(\vec{a}) > 0$. Let $\min(\vec{\alpha}) > b > a > \max(\vec{a}) $, $d \in \left(0 , \min \left( 1/4, (b-a)/4\right) \right)$ and $u \in \mathbb{C}$ with $\Re(u)> 0$. Then
\begin{equation}\label{LGPTe1}
\mathbb{E} \left[ e^{-u \ZP} \right] = \det \left( I + K_u \right)_{L^2(C_{a,3\pi/4})},
\end{equation}
where the operator $K_u$ is defined in terms of its integral kernel
\begin{equation}\label{kernelMain}
K_u(v,v') = \frac{1}{2\pi \i} \int_{D_v}  \frac{\pi }{\sin(\pi (v-w))}   \prod_{n = 1}^N \frac{\Gamma(v - a_n)}{\Gamma(w - a_n)} \prod_{m = 1}^M \frac{\Gamma(\alpha_m - w)}{\Gamma(\alpha_m - v)} \frac{u^{w-v} dw}{w - v'},
\end{equation}
where $C_{a,3\pi/4}$ and $D_v = D_v(b, \pi/4, d)$ are as in Definitions \ref{ContV} and \ref{ContE} .
\end{theorem}
To prove Theorem \ref{LGPT1} we need the following result, whose proof is postponed to Section \ref{Section2.4}.

\begin{theorem}\label{BCT1v2}
Fix integers $N \geq 9, M \geq 1$, a real number $\tau > 0$ and $\vec{a} = (a_1, \dots, a_N) \in \mathbb{R}^N$, $\vec{\alpha}= (\alpha_1, \dots, \alpha_M) \in \mathbb{R}_{>0}^M$ such that $\min(\vec{\alpha}) - \max(\vec{a}) > 0$. Let $\min(\vec{\alpha}) > b > a > \max(\vec{a}) $, $d \in \left(0 , \min \left( 1/4, (b-a)/4\right) \right)$ and $u \in \mathbb{C}$ with $\Re(u)> 0$. Then
\begin{equation}\label{BCe1}
\mathbb{E} \left[ e^{-u \ZP(\tau)} \right] = \det \left( I + K^{\tau}_u \right)_{L^2(C_{a,3\pi/4})},
\end{equation}
where the operator $K^{\tau}_u$ is defined in terms of its integral kernel
\begin{equation}\label{kernel2}
K^{\tau}_u(v,v') = \frac{1}{2\pi \i} \int_{D_v} \frac{\pi }{\sin(\pi (v-w))}  \prod_{n = 1}^N \frac{\Gamma(v - a_n)}{\Gamma(w - a_n)} \prod_{m = 1}^M \frac{\Gamma(\alpha_m - w)}{\Gamma(\alpha_m - v)} \frac{u^{w-v} e^{\tau(w^2 - v^2)/2}dw}{w - v'},
\end{equation}
with $C_{a,3\pi/4}$ and $D_v = D_v(b, \pi/4, d)$ as in Definitions \ref{ContV} and \ref{ContE} .
\end{theorem}
\begin{remark}
The Fredholm determinants in Theorems \ref{LGPT1} and \ref{BCT1v2} are as in Section \ref{Section2.1}. In particular, part of the statement of these theorems is that the kernels are well-defined, each term in the series (\ref{fredholmDefS2}) is finite and the series converges for the given choice of parameters.
\end{remark}

In addition, we require the following two results, whose proofs are given in Section \ref{Section10}.
\begin{proposition}\label{PE1v2}
Fix $M, N \geq 1$, $\vec{a} = (a_1, \dots, a_N) \in \mathbb{R}^N$, $\vec{\alpha}= (\alpha_1, \dots, \alpha_M) \in \mathbb{R}^M$, $T \geq 0$, a compact set $K \subset \mathbb{C}$ and $v,u\in \mathbb{C}$ with $\Re(u) > 0$. Then there exist positive constants $L_0, C_0, c_0$ depending on all previous constants and $K$ such that if $\tau \in[0,T]$, $x,a \in K$, $w = a + z$, $\arg(z) \in [\pi/4, \pi/3] \cup [-\pi/3, -\pi/4]$, $|w| \geq L_0$
\begin{equation}\label{PE1v2e1}
\left|\frac{\pi }{\sin(\pi (v-w))}  \frac{\prod_{m = 1}^M  \Gamma(x+ \alpha_m - w )}{\prod_{n = 1}^N\Gamma(w - a_n) }  u^{w-v} e^{\tau (w^2 - v^2)/2} \right| \leq C_0 e^{- c_0 |w| \log |w|}.
\end{equation}
\end{proposition}

\begin{proposition}\label{PE2v2}
Fix  $M, N \geq 1$,  $T \geq 0$, $\tau \in [0,T]$, $\vec{a} = (a_1, \dots, a_N) \in \mathbb{R}^N$, $\vec{\alpha}= (\alpha_1, \dots, \alpha_M) \in \mathbb{R}_{>0}^M$ and $u \in \mathbb{C}$ with $\Re(u) > 0$. Put $\theta_0 = \min(\vec{\alpha}) - \max(\vec{a})$ and assume that $\theta_0 > 0$ and $\delta_0 \in (0, \min(1/4, \theta_0 /16))$. Suppose that $v, v' \in C_{a, \phi}$ as in Definition \ref{ContV} with $a \in [ \max(\vec{a}) + \delta_0,  \min(\vec{\alpha}) - 5\delta_0]$ and $\phi \in [3\pi/4, 5\pi/6]$. Finally, fix $b \in [a+ 2\delta_0, \min(\vec{\alpha}) - 3\delta_0]$ and denote by $D_v$ the contour $D_v(b, \pi/4, \delta_0)$ as in Definition \ref{ContE}. Then there exists a positive constant $ C_0$ depending on $\vec{a}, \vec{\alpha}, \delta_0, u, N, M, T$ (and not $\tau$) such that if $x \in \mathbb{C}$ with $d(x, [ 0, 1]) \leq \delta_0$ then we have
\begin{equation}
  \left| \int_{D_v} \frac{\pi d\mu(w)}{\sin(\pi (v-w))}   \prod_{n = 1}^N \frac{\Gamma(v - a_n)}{\Gamma(w - a_n)} \prod_{m = 1}^M \frac{\Gamma(\alpha_m + x - w)}{\Gamma(\alpha_m + x- v)} \frac{u^{w-v} e^{\tau (w^2 - v^2)/2}}{w - v'} \right| \leq \frac{C_0}{1 + |v|^2}.
\end{equation}
\end{proposition}

\begin{proof}[Proof of Theorem \ref{LGPT1}] From (\ref{PPAS}) we know that $\ZP(\tau)$ converges almost surely to $\ZP(0)$ as $\tau \rightarrow 0+$. By the bounded convergence theorem (here we use $\Re(u) > 0$) and the distributional equality of $\ZP(0)$ and $\ZP$ we conclude
\begin{equation}\label{LGE1}
\mathbb{E} \left[ e^{-u \ZP} \right]  = \lim_{\tau \rightarrow 0+} \mathbb{E} \left[ e^{-u \ZP(\tau)} \right].
\end{equation}
In view of Theorem \ref{BCT1v2} it suffices to show that
\begin{equation}\label{LGPe1}
 \lim_{\tau \rightarrow 0+}\det \left( I + K^{\tau}_u \right)_{L^2(C_{a,3\pi/4})} = \det \left( I + K_u \right)_{L^2(C_{a,3\pi/4})}.
\end{equation}

From Proposition \ref{PE1v2} we have that $K^{\tau}_u(v,v')$ and $K_u(v,v')$ are well-defined. Let us denote
$$F_{\tau}(w,v',v) =\frac{\pi e^{\tau (w^2 - v^2)/2} u^{w-v}}{\sin(\pi(w-v)) (w- v')} \prod_{n = 1}^N \frac{\Gamma(v - a_n)}{\Gamma(w - a_n)} \prod_{m = 1}^M \frac{\Gamma(\alpha_m - w)}{\Gamma(\alpha_m - v)} .$$
It is clear that pointwise for fixed $v,v',w$ we have $\lim_{\tau \rightarrow 0} F_{\tau}(w,v',v)  = F_{0}(w,v',v)$. By the Dominated convergence theorem with dominating function $C_0 \exp(-c_0 |w|\log|w|)$ (here we use Proposition \ref{PE1v2} with $K = [0, \theta]$, $T = 1$, and $M,N, \vec{a}, \vec{\alpha}, u, v$ as above) we conclude for $v,v' \in C_{a,3\pi/4}$
\begin{equation}\label{LGPe2}
\lim_{\tau \rightarrow 0+} K^{\tau}_u(v,v')  = \lim_{\tau \rightarrow 0+} \frac{1}{2\pi \i} \int_{D_v}  F_{\tau}(w,v',v)dw  =  \frac{1}{2\pi \i} \int_{D_v}  F_{0}(w,v',v) dw =  K_u(v,v').
\end{equation}
Let $\delta_0 = \min (d/4, a - \max(\vec{a}), (\min(\vec{\alpha}) - b)/4)$. By Cauchy's theorem we can deform $D_v(b, \pi/4, d)$ in the definition of $K^{\tau}_u(v,v')$ and $K_u(v,v')$ to $D_v(b, \pi/4, \delta_0)$ without affecting the value of the integrals. Here the poles we need to avoid during the deformation are the points $v + n$ for $n \in \mathbb{Z}$ coming from $\sin(\pi(w-v))$ in $F_\tau(w,v',v)$ and the points $\alpha_m + n$ for $n \in \mathbb{Z},$ $n \geq 0$ coming from the functions $\Gamma(\alpha_m - w)$ in $F_\tau(w,v',v)$ and by the definition of $\delta_0$ and $d$ we know that none of these poles are crossed in the process of the deformation. We continue to write $D_v$ for this new contour. From Proposition \ref{PE2v2} (applied to $T = 1$, $a,b,M,N,u,\vec{a},\vec{\alpha}$  as in the statement of this theorem, $\delta_0 =\min (d/4, a - \max(\vec{a}), (\min(\vec{\alpha}) - b)/4)$, $x = 0$, $\phi = 3\pi/4$) we know that if $\tau \in [0,1]$ there is a $\tau$-independent constant $C_0$ such that for $v,v' \in C_{a,3\pi/4}$ we have
\begin{equation}\label{LGPe3}
|K^{\tau}_u(v,v')|  \leq \frac{C_0}{1 + |v|^2}.
\end{equation}
Using (\ref{LGPe2}) and (\ref{LGPe3}) we apply Lemma \ref{FDCT} with $\Gamma = C_{a,3\pi/4}$, $F(v) = \frac{C_0}{1 + |v|^2}$ to conclude (\ref{LGPe1}). In particular, observe that Lemma \ref{FDCT} implies that the series defining the determinants on both sides of (\ref{LGPe1}) are absolutely convergent.
\end{proof}

%
\subsection{Proof of Theorem \ref{BCT1v2}}\label{Section2.4} In this section we give the proof of Theorem \ref{BCT1v2}, which we needed to prove Theorem \ref{LGPT1}. We require the following proposition, whose proof is deferred to Section \ref{Section10}.
\begin{proposition}\label{PE3v2}
Fix $M, N \geq 1$, $\vec{a} = (a_1, \dots, a_N) \in \mathbb{R}^N$, $\vec{\alpha}= (\alpha_1, \dots, \alpha_M) \in \mathbb{R}^M$, $\tau > 0$, a compact set $K \subset \mathbb{C}$ and $v,u\in \mathbb{C}$ with $\Re(u) > 0$. Then there exist positive constants $L_0, C_0, c_0$ depending on all of the previous constants such that if $w = a + z$ with $a \in K$, $\arg(z) \in [\pi/4, \pi/2] \cup [-\pi/2, -\pi/4]$, $|w| \geq L_0$ we have
\begin{equation}\label{PE3v2e1}
\left|\frac{\pi }{\sin(\pi (v-w))} \frac{\prod_{m = 1}^M  \Gamma( \alpha_m - w )}{\prod_{n = 1}^N\Gamma(w - a_n) } u^{w-v} e^{\tau (w^2 - v^2)/2}\right| \leq C_0 e^{- c_0 |w| \log |w|}.
\end{equation}
\end{proposition}
\begin{remark}
Let us briefly explain the difference between Proposition \ref{PE1v2} from the previous section and Proposition \ref{PE3v2}. Proposition \ref{PE1v2} is formulated slightly more generally in that (\ref{PE1v2e1}) holds for all $x$ in a compact set $K$, while in (\ref{PE3v2e1})we have set $x = 0$. In the second part of the proof of Theorem \ref{BCT1v2} one of the arguments will require an analytic continuation in $x$, which demands we get estimates not just at $x = 0$ but in a neighborhood of it. In the first part of the proof of Theorem \ref{BCT1v2} we will not need the same type of generality for the estimate and Proposition \ref{PE3v2} will suffice for our purposes. A more important difference is that
in Proposition \ref{PE1v2} we assumed $\arg(z) \in [\pi/4, \pi/3] \cup [-\pi/3, -\pi/4]$, while in Proposition \ref{PE3v2} we assume $\arg(z) \in [\pi/4, \pi/2] \cup [-\pi/2, -\pi/4]$. The reason behind this difference is that in Theorem \ref{BCT1} the contour $\mathcal{D}_v$ extends vertically up near infinity, and we need estimates when $|\arg(z)| = \pi/2$. Yet another difference is that the constants $C_0, c_0, L_0$ in Proposition \ref{PE1v2} only depended on $T$ as long as $\tau \in [0, T]$ -- here it is important that we get uniform estimates near $0$ since in the proof of Theorem \ref{LGPT1} we performed the limit $\tau \rightarrow 0+$. In the proof of Theorem \ref{BCT1v2} $\tau > 0$ is fixed and the constants $C_0, c_0, L_0$  in Proposition \ref{PE1v2} are only proved to exist for the range of $\arg(z) $, provided $\tau > 0$. In plain words, there is a difficulty in showing that $C_0, c_0$ in the two propositions exist either if $\tau$ is close to $0$ or when $|\arg(z)|$ is close to $\pi/2$. If one of these quantities is bounded away from its problematic value, the constants exist. Then Proposition \ref{PE1v2} assumes $|\arg(z)|$ is bounded away from $\pi/2$, while Proposition \ref{PE3v2} assumes $\tau$ is bounded away from $0$.
\end{remark}

\begin{proof}[Proof of Theorem \ref{BCT1v2}] The proof is split into two parts. In the first part we prove the theorem under the assumption that $\min(\vec{\alpha}) - \max(\vec{a}) > 1$ and in the second part we extend the result to the case when $\min(\vec{\alpha}) - \max(\vec{a}) > 0$. Each part of the proof is split into several steps, and begins with a brief description of the argument for the reader's convenience. Throughout we work with the same notation as in Theorem \ref{BCT1v2}.

{\bf \raggedleft Part I.} In this part we prove the theorem under the assumption that $\min(\vec{\alpha}) - \max(\vec{a}) > 1$. Before we go into the details let us sketch our approach. In Step 1 we apply Proposition \ref{BCT1} and reduce the proof of the theorem to the establishment of a certain integral identity. This identity essentially states that the $n$-fold contour integral of $ \det \left[ K^\tau_u(v_i, v_j)\right]_{i,j = 1}^n $ over $\mathcal{C}^n_{\vec{a}; \vec{\alpha}; \pi/6}$ equals the one over $C^n_{a, 3\pi/4}$. One would like to argue that the two integrals are equal by Cauchy's theorem. In order to implement this approach one needs to fix the $w$-contour somehow and deform the $v$-contours from $\mathcal{C}_{\vec{a}; \vec{\alpha}; \pi/6}$ to $C_{a, 3\pi/4}$, but the problem is that no fixed $w$ contour works. Presumably, one can simultaneously deform the $v$ and $w$ contours without crossing any poles of the integrand and then argue that by Cauchy's theorem the value of the integral does not change. This type of deformation is somewhat delicate because we are dealing with infinite contours and also because the $w$-contours depends on $v$'s. The way around this issue is to first truncate the $\mathcal{C}_{\vec{a}; \vec{\alpha}; \pi/6}$ and $C_{a, 3\pi/4}$ contours. For the truncated contours, one {\em can} fix the $w$-contour and perform the deformation. The challenge is in showing that the truncation and deformation steps produce a small error. The exact nature of the truncation is explained in Step 2. Step 3 shows that the truncation and deformation of the contours produces only a small error, and in Step 4 the application of Cauchy's theorem is justified for the truncated contours.
\begin{remark}
We point out that in \cite{KQ} a similar strategy was carried out with two crucial differences. Firstly, \cite[Theorem 2.1]{BCFV} is only formulated for $\phi \in (0, \pi/4)$, but when the theorem is recalled in \cite[Theorem A.1]{KQ} it is formulated for $\phi \in (0, \pi/4]$. We point out that the authors in \cite{KQ} provide the necessary decay estimates that would enable them to deform the contours to $\phi = \pi/4$ and all the arguments in that paper go through. Part of our proof of Theorem \ref{BCT1v2} is actually justifying this deformation carefully. Secondly, their contour is $D_v$ with $\phi = \pi/2$, and one reason it works for them is that they have nice estimates for the integral that come from elegant cancellations coming from the fact that they work with $M = N$. If $M \neq N$ the estimates they get no longer hold and one is better off working with $\phi = \pi/4$.
\end{remark}

We now turn to the proof. For clarity we split the proof into several steps.\\

{\bf \raggedleft Step 1.} In this step we reduce the proof of the theorem to establishing the following equality for every $n \geq 1$
\begin{equation}\label{R1}
 \int_{\mathcal{C}_{\vec{a}; \vec{\alpha}; \pi/6}} \hspace{-4mm}\cdots \int_{\mathcal{C}_{\vec{a}; \vec{\alpha}; \pi/6}} \hspace{-4mm} \det \left[ \tilde{K}(v_i, v_j)\right]_{i,j = 1}^n  \prod_{i = 1}^n dv_i  - \int_{C_{a,3\pi/4}} \hspace{-4mm}\cdots \int_{C_{a,3\pi/4}} \hspace{-4mm} \det \left[ \tilde{K}(v_i, v_j)\right]_{i,j = 1}^n  \prod_{i = 1}^n dv_i = 0,
\end{equation}
where $\tilde{K}(v,v')$ is defined in terms of its integral kernel
\begin{equation}\label{kernelAux}
\tilde{K}(v,v') = \frac{1}{2\pi \iota} \int_{\tilde{D}_v} \frac{\pi }{\sin(\pi (v-w))}   \prod_{n = 1}^N \frac{\Gamma(v - a_n)}{\Gamma(w - a_n)} \prod_{m = 1}^M \frac{\Gamma(\alpha_m - w)}{\Gamma(\alpha_m - v)}\frac{u^s e^{v\tau (w-v)+ \tau (w-v)^2/2}dw}{w - v'},
\end{equation}
where $\tilde{D}_v = D_v(\tilde{b}, \pi/4, \tilde{d})$ as in Definition \ref{ContE} with $\min(\vec{\alpha}) > \tilde{b} > \max( \mu, a)$ and $\tilde{d} \in (0, \min(1/4, a - \max(\vec{a}), (\tilde{b} -  \max( \mu, a))/8, (\min(\vec{\alpha}) - \tilde{b})/8)$. Recall $\mu = (\max(\vec{a}) + \min(\vec{\alpha}))/2$ from Definition \ref{DCont1}.

As explained in the proof of Theorem \ref{LGPT1} we have that $K_u^{\tau}(v,v')$ in Theorem \ref{BCT1v2} is a well-defined measurable function on $\Gamma \times \Gamma$ where $\Gamma = C_{a, 3\pi/4}$. In addition, $\det \left( I + K^{\tau}_u \right)_{L^2(C_{a,3\pi/4})}$ is a well-defined absolutely convergent series as in (\ref{fredholmDefS2}) for any $\tau \geq 0$.

As shown in the proof of Proposition \ref{BCT1} (see \cite[Theorem 2.1]{BCFV}) the right side of (\ref{BCTe1}) is given by an absolutely convergent series as in (\ref{fredholmDefS2}) for any $\tau > 0$ -- here we use the fact that $\min(\vec{\alpha}) - \max(\vec{a}) > 1$, which is the case under which Proposition \ref{BCT1} holds. Consequently, to prove the theorem it suffices to show that these sums are termwise equal, i.e. for each $n \in \mathbb{N}$ we have
\begin{equation}\label{R11}
 \int_{\mathcal{C}_{\vec{a}; \vec{\alpha}; \pi/6}} \hspace{-1mm}\cdots  \int_{\mathcal{C}_{\vec{a}; \vec{\alpha}; \pi/6}}\det \left[ \mathcal{K}_u(v_i, v_j)\right]_{i,j = 1}^n  \prod_{i = 1}^n dv_i  =  \int_{C_{a,3\pi/4}} \hspace{-4mm}\cdots \int_{C_{a,3\pi/4}}\det \left[ K^\tau_u(v_i, v_j)\right]_{i,j = 1}^n  \prod_{i = 1}^n dv_i
\end{equation}

Notice that we can apply a change of variables $w = s+ v$ in (\ref{kernel1}) and by Cauchy's theorem deform the resulting contour $v+\mathcal{D}_v$ to $\tilde{D}_v$, without affecting the value of the kernel. The decay estimates necessary to deform the contour near infinity come from Proposition \ref{PE1v2} applied to $K = \{0 \} \cup [\max(\vec{a}), \min(\vec{\alpha})]$, $T =\tau$, and $ M,N, \vec{a}, \vec{\alpha}, u$ as in the statement of the theorem. Here we implicitly used the functional equation
\begin{equation}\label{gammaFE}
\Gamma(s)\Gamma(1-s) = \frac{\pi}{\sin(\pi s)},
\end{equation}
which can be found in \cite[Chapter 6, Theorem 1.4]{Stein}. This proves that the left side of (\ref{R11}) is equal to the first term in (\ref{R1}). Analogously, by Cauchy's theorem we can deform the contour $D_v$ in the definition of $K^\tau_u(v, v')$ to the contour $\tilde{D}_v$ without affecting the value of the kernel. The decay estimates necessary to deform the contour near infinity come from Proposition \ref{PE1v2} applied to $K = \{0 \} \cup [\max(\vec{a}), \min(\vec{\alpha})]$, $T = \tau$ and $M,N, \vec{a}, \vec{\alpha}, u$ as in the statement of the theorem. The conclusion is that the right side of (\ref{R11}) is equal to the second term in (\ref{R1}). Thus we have reduced the proof of the theorem to establishing (\ref{R1}).\\

{\bf \raggedleft Step 2.} For $L > \theta$ we let $v_+$ and $v_-$ be the points on $\mathcal{C}_{\vec{a}; \vec{\alpha}; \pi/6}$ such that $|v_+| = |v_-| = L$ and $v_{\pm}$ lie in the upper and lower complex half-planes respectively. By taking $L$ sufficiently large we can ensure that $v_{\pm}$ are well defined and have negative real part. We similarly denote by $w_{\pm}$ the points on $C_{a, 3\pi/4}$ such that $|w_{\pm}| = L$ and for large enough $L$ we have $w_{\pm}$ are well-defined with negative real part and that the angle the line joining $w_+$ with $v_+$ makes with the $x$-axis is at least $\pi/4$, see the left side of Figure \ref{S2_3}.

We also introduce the following contours. We let $\mathcal{C}_0^L$ denote the part of $\mathcal{C}_{\vec{a}; \vec{\alpha}; \pi/6}$ contained in the set $B_L = \{z \in \mathbb{C}: |z| \leq L\}$, $C^L$ denotes the part of $C_{a, 3\pi/4}$ contained in $B_L$, $\mathcal{C}_1^L$ denotes the two straight segments connecting $w_-$ to $v_-$ and $w_+$ to $v_+$, $\mathcal{C}^L = \mathcal{C}_0^L \cup \mathcal{C}_1^L$ and $C^L$ denotes the part of $C_{a, 3\pi/4}$ contained in $B_L$. All of these contours are oriented to have increasing imaginary part. See the right side of Figure \ref{S2_3}.

\begin{figure}[h]
\centering
\scalebox{0.6}{\includegraphics{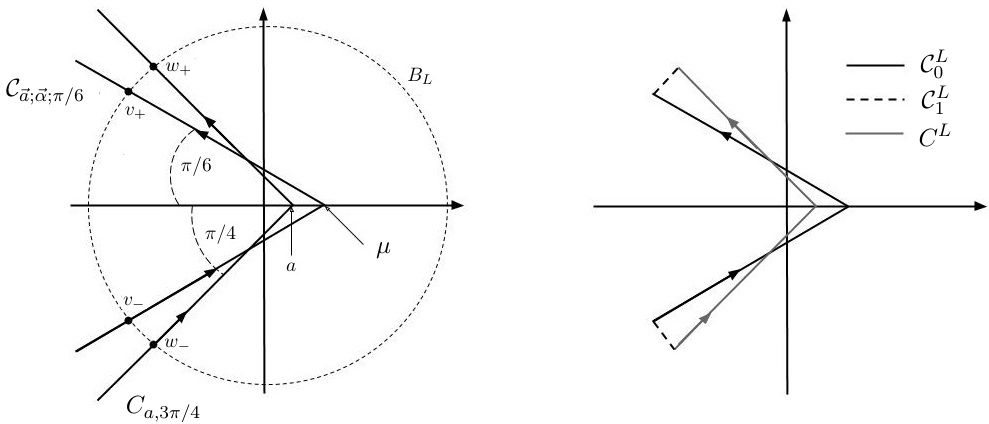}}
 \captionsetup{width=.9\linewidth}
\caption{The left part depicts the contours $\mathcal{C}_{\vec{a}; \vec{\alpha}; \pi/6}$ and $C_{a, 3\pi/4}$ as well as the points $w_{\pm}$ and $v_{\pm}$ given by the intersection of these contours with the circle of radius $L$ centered at the origin. The right part depicts the contours $\mathcal{C}^L_0$, $\mathcal{C}^L_1$ and $C^L$. }
\label{S2_3}
\end{figure}

Let $\epsilon > 0$ be given. We claim that we can find $L_1$ sufficiently large so that if $L \geq L_1$ we have the following statements
\begin{equation}\label{R2}
\left| \int_{\mathcal{C}_{\vec{a}; \vec{\alpha}; \pi/6}} \hspace{-2mm}\cdots \int_{\mathcal{C}_{\vec{a}; \vec{\alpha}; \pi/6}} \hspace{-2mm} \det \left[ \tilde{K}(v_i, v_j)\right]_{i,j = 1}^n  \prod_{i = 1}^n dv_i  -  \int_{\mathcal{C}^L_{0}} \hspace{-2mm}\cdots \int_{\mathcal{C}^L_{0}} \det \left[ \tilde{K}(v_i, v_j)\right]_{i,j = 1}^n  \prod_{i = 1}^n dv_i  \right| < \epsilon;
\end{equation}
\begin{equation}\label{R3}
\left| \int_{C_{a,3\pi/4}} \hspace{-4mm}\cdots \int_{C_{a,3\pi/4}} \hspace{-4mm} \det \left[ \tilde{K}(v_i, v_j)\right]_{i,j = 1}^n  \prod_{i = 1}^n dv_i - \int_{C^L} \cdots \int_{C^L}  \det \left[ \tilde{K}(v_i, v_j)\right]_{i,j = 1}^n  \prod_{i = 1}^n dv_i \right| < \epsilon;
\end{equation}
\begin{equation}\label{R4}
\left|\int_{\mathcal{C}^L_{0}} \hspace{-2mm}\cdots \int_{\mathcal{C}^L_{0}} \det \left[ \tilde{K}(v_i, v_j)\right]_{i,j = 1}^n  \prod_{i = 1}^n dv_i   - \int_{\mathcal{C}^L} \cdots \int_{\mathcal{C}^L} \det \left[ \tilde{K}(v_i, v_j)\right]_{i,j = 1}^n  \prod_{i = 1}^n dv_i  \right| < \epsilon;
\end{equation}
\begin{equation}\label{R5}
\int_{\mathcal{C}^L} \cdots \int_{\mathcal{C}^L} \hspace{-2mm} \det \left[ \tilde{K}(v_i, v_j)\right]_{i,j = 1}^n  \prod_{i = 1}^n dv_i = \int_{C^L} \cdots \int_{C^L} \hspace{-2mm} \det \left[\tilde{K}(v_i, v_j)\right]_{i,j = 1}^n  \prod_{i = 1}^n dv_i .
\end{equation}
If the above are true, we conclude that the left side of (\ref{R1}) is bounded in absolute value by $3\epsilon$, which would imply (\ref{R1}) as $\epsilon > 0$ is arbitrary. We establish the above statements in the steps below.\\

{\bf \raggedleft Step 3.} In this step we prove (\ref{R2}), (\ref{R3}) and (\ref{R4}). We prove (\ref{R5}) in the next step.

If we apply Proposition \ref{PE2v2} to $T= \tau$, $\delta_0 = \tilde{d}$, $b = \tilde{b}$, $x = 0$, $\phi = 3\pi/4$ and $a,u, \vec{a}, \vec{\alpha}, M,N$ as in the statement of the theorem we see that we can find a constant $C_0 > 0$ such that
\begin{equation}\label{EstKtilde}
\left| \tilde{K}(v,v') \right| \leq \frac{C_0}{1 + |v|^2},
\end{equation}
as long as $v,v' \in C_{a, 3\pi/4}$. Analogously, we can apply  Proposition \ref{PE2v2} to $T= \tau$, $\delta_0 = \tilde{d}$, $b = \tilde{b}$, $x = 0$, $\phi = 5\pi/6$, $a = \mu$ and $M, N, \vec{a}, \vec{\alpha}, u$ as in the statement of the theorem to show that (\ref{EstKtilde}) also holds (with a possibly bigger constant $C_0$) for $v,v' \in \mathcal{C}_{\vec{a}; \vec{\alpha}; \pi/6}$. In particular, by Lemma \ref{FDCT} we know that $\det \left[ \tilde{K}(v_i, v_j)\right]_{i,j = 1}^n $ is integrable on $\Gamma^n$ if $\Gamma = C_{a,3\pi/4}$ or $\Gamma = \mathcal{C}_{\vec{a}; \vec{\alpha}; \pi/6}$, which implies that we can make $L$ sufficiently large so that (\ref{R2}) and (\ref{R3}) both hold.

Notice that we can rewrite the expression in the absolute value of (\ref{R4}) as
\begin{equation}\label{FMT6}
 \sum_{ \epsilon_1, \dots, \epsilon_n \in \{0, 1\}: \sum_i \epsilon_i > 0}  \int_{\mathcal{C}_{\epsilon_1}^L} \cdots \int_{\mathcal{C}_{\epsilon_1}^L} \det \left[ \tilde{K}(v_i, v_j) \right]_{i,j = 1}^n  \prod_{i = 1}^n dv_i.
\end{equation}
By Hadamard's inequality and (\ref{EstKtilde}) we have
\begin{equation}\label{FMT7}
\left|  \det \left[ \tilde{K}(v_i, v_j) \right]_{i,j = 1}^n \right| \leq n^{n/2} \prod_{i =1}^n \frac{C_0}{1 + |v_i|^2}.
\end{equation}
Also we observe that there is a constant $A$ that depends on $C_0$ such that for all $L$ sufficiently large
\begin{equation}\label{FMT8}
\int_{\mathcal{C}_{0}^L }\frac{C_0 |d\mu(v)|}{1 + |v|^2} \leq A \mbox{ and } \int_{\mathcal{C}^L_{1} }\frac{C_0|d\mu(v)|}{1 + |v|^2} \leq \frac{A}{1 + L}.
\end{equation}
The first inequality in (\ref{FMT8}) follows from the integrability of $\frac{1}{1 + |v|^2}$ and the second from the fact that for large $L$ the contour $\mathcal{C}_1^L$ is typical distance $L$ away from the origin, and has length less than $2L$.
Combining (\ref{FMT6}), (\ref{FMT7}) and (\ref{FMT8}) we conclude that the right side of (\ref{R4}) is bounded by
\begin{equation*}
  \sum_{ \epsilon_1, \dots, \epsilon_n \in \{0, 1\}: \sum_i \epsilon_i > 0 } n^{n/2} \prod_{i = 1}^n \frac{A}{(1+L)^{\epsilon_i}} \leq \frac{n^{n/2} (2A)^n}{L+1} ,
\end{equation*}
where we used that the sum is over $\epsilon_i$ not all equal to $0$. The above implies that (\ref{R4}) holds as long as $L$ is taken large enough.\\

{\bf \raggedleft Step 4.} In this step we prove (\ref{R5}). In principle, one would like to argue that the integrands are analytic in the $v$ variables and that in deforming $C^L$ to $\mathcal{C}^L$ we do not pass any poles. Formalizing this idea is intricate, since the $\tilde{D}_v$ contours depend on the $v$'s. This is why we introduce a sequence of contours $\mathcal{C}_k$, which interpolates between the contours $C^L$ and $\mathcal{C}^L$, and for each $\mathcal{C}_k$ we define a corresponding contour $D_k$ for the $w$ variables. For these new contours we show that  the deformation of $\mathcal{C}_k$ to $\mathcal{C}_{k+1}$ is possible without affecting the value of the integral, which ultimately yields (\ref{R5}).

For each $ s\in [\Im(w_-), \Im(w_+)]$ we let $(x_1(s), s)$ and $(x_2(s), s)$ denote the unique points on $\mathcal{C}^L$ and $C^L$ that have $y$-coordinate equal to $s$. Let $|\Delta| = \max |x_1(s) - x_2(s)|$ and suppose that $K$ is a sufficiently large positive integer such that $\Delta/K < \tilde{d}/40$, where we call that $\tilde{d}$ was defined in Step 1. For each $k = 0,1, \dots, K$ we let $\mathcal{C}^L_k$ denote the contour
$$\mathcal{C}_k := \left\{ (x(s), s): x(s) = \frac{K-k}{K} x_1(s) +   \frac{k}{K} x_2(s) , s\in [\Im(w_-), \Im(w_+)]\right\},$$
 whose orientation is of increasing imaginary part. See the left part of Figure \ref{S2_4}.
\begin{figure}[h]
\centering
\scalebox{0.5}{\includegraphics{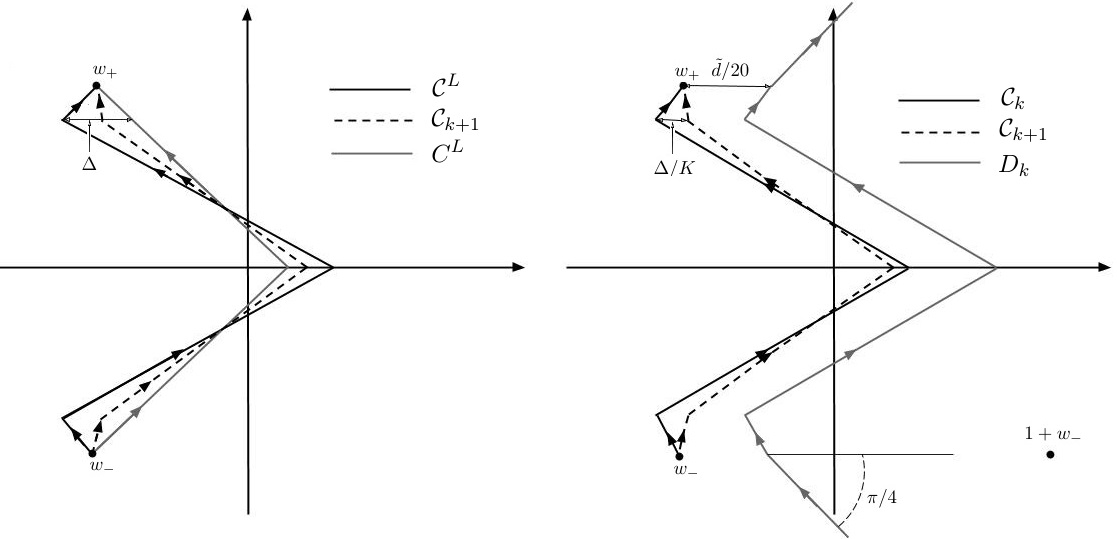}}
 \captionsetup{width=.9\linewidth}
\caption{The left part illustrates the contours $\mathcal{C}^L$ and $C^L$ as well as one of the interpolating contours $\mathcal{C}_{k+1}$. As $k$ varies from $0$ to $K$ the contours $\mathcal{C}_k$ traverse a collection of contours so that $\mathcal{C}_0 = \mathcal{C}^L$ and $\mathcal{C}_K = C^L$ and the contours $\mathcal{C}_k$ and $\mathcal{C}_{k+1}$ have maximal horizontal separation $\Delta/K$. The right part depicts the contours $\mathcal{C}_k$ and $\mathcal{C}_{k+1}$ schematically (some distances have been exaggerated to make the picture comprehensible) as well as the contour $D_k$. The contour $D_k$ is a translate by $\tilde{d}/20$ of $\mathcal{C}_k$ in the strip $\Im(z) \in [ \Im(w_-), \Im(w_+)]$ and departs at angles $\pm\pi/4$ at $w_{\pm} + \tilde{d}/20$ to infinity. The point $1 + w_-$ is drawn to indicate that $1 + \mathcal{C}_k$ and $1 + \mathcal{C}_{k+1}$ are both to the right of $D_k$. }
\label{S2_4}
\end{figure}

We claim that for each $k = 0, \dots, K-1$ we have
\begin{equation}\label{FMT12}
 \int_{\mathcal{C}_k}\cdots \int_{\mathcal{C}_k} \det \left[ \tilde{K}(v_i, v_j)\right]_{i,j = 1}^n  \prod_{i = 1}^n dv_i =  \int_{\mathcal{C}_{k+1}} \cdots \int_{\mathcal{C}_{k+1}} \det \left[ \tilde{K}(v_i, v_j)\right]_{i,j = 1}^n  \prod_{i = 1}^n dv_i.
\end{equation}
When $k = 0$ the left side of (\ref{FMT12}) equals the left side of (\ref{R5}), while at $k = K-1$ the right side of (\ref{FMT12}) equals the right side of (\ref{R5}). Consequently, (\ref{FMT12}) implies (\ref{R5}) and we focus on establishing (\ref{FMT12}).

By expanding the determinants and substituting the definition of $\tilde{K}(v, v')$ we see that it suffices to show that for each $k \in \{1, \dots, K-1\}$ and $\sigma \in S_n$ we have
 \begin{equation}\label{FMT13}
\begin{split}
 \int_{\mathcal{C}_k} \cdots \int_{\mathcal{C}_k} \int_{\tilde{D}_{v_1}}\hspace{-3mm}\cdots \int_{\tilde{D}_{v_n}}  \prod_{j = 1}^n  \frac{F(v_j, w_j)dw_j }{w_j - v_{\sigma(j)} }  \prod_{i = 1}^n  dv_i =  \\
 \int_{\mathcal{C}_{k+1}} \hspace{-3mm}\cdots \int_{\mathcal{C}_{k+1}} \int_{\tilde{D}_{v_1}}\hspace{-3mm}\cdots \int_{\tilde{D}_{v_n}}  \prod_{j = 1}^n  \frac{F(v_j, w_j)dw_j }{w_j - v_{\sigma(j)} }  \prod_{i = 1}^n  dv_i  , \mbox{ where }
\end{split}
\end{equation}
$$F(v, w) =\frac{\pi}{\sin(\pi (v-w))}   \prod_{n = 1}^N \frac{\Gamma(v - a_n)}{\Gamma(w - a_n)} \prod_{m = 1}^M \frac{\Gamma(\alpha_m - w)}{\Gamma(\alpha_m - v)} u^{w-v} e^{\tau v(w-v)+ \tau (w-v)^2/2}. $$
We deform all $\tilde{D}_{v_i}$ contours to the contour $D_k$, which consists of the following three pieces:
\begin{enumerate}
\item  it equals $\mathcal{C}_k + \tilde{d}/20$ for $z: \Im(z) \in [\Im(w_-), \Im(w_+)]$;
\item  it is given by $z(s) = w_+ + \tilde{d}/20+ (s-\Im(w_+)) + \i (s - \Im(w_+))$ for $s > \Im(w_+)$;
\item  it is given by $z(s) = w_- +\tilde{d}/20 + |s-\Im(w_-)| + \i (s - \Im(w_-))$ for $s < \Im(w_-)$.
\end{enumerate}
 In words, $D_k$ closely follows $\mathcal{C}_k$ a small distance away to the right and departs at angle $\pm \pi/4$ once it reaches the end of $\mathcal{C}_k + \tilde{d}/20$, see the right part of Figure \ref{S2_4}.

 Notice that by construction $\mathcal{C}_{k+1}$ is at most a distance $\tilde{d}/40$ away from $\mathcal{C}_k$ (horizontally) and so when deforming $\tilde{D}_{v_i}$ to $D_k$ for $v_i \in \mathcal{C}_k$ or $\mathcal{C}_{k+1}$ we do not cross any poles of the integrand. Indeed, the poles are located at $\alpha_m, 1+ \alpha_m, 2+ \alpha_m, \dots$, as well as $v_i + \mathbb{Z}$. We do not cross any of the poles $\alpha_m, 1+ \alpha_m, 2+ \alpha_m, \dots$ since by construction $\tilde{b}, \mu, a \in (\max(a),\min(\alpha))$. On the other hand, we do not cross any of the poles $v_i + \mathbb{Z}$ since $D_k$ is to the right by at least $\tilde{d}/40$ and at most $\tilde{d}/20$ from both $\mathcal{C}_k$ and $\mathcal{C}_{k+1}$ -- this ensures that $v_i$ are to the left of $D_k$ and $v_i + 1$ to the right.  By Cauchy's theorem the value of the integrals does not change and the estimates necessary to deform the contours near infinity come from Proposition \ref{PE1v2} applied to $K = \{0 \} \cup [\max(\vec{a}), \min(\vec{\alpha})]$, $T = \tau$ and $M,N, \vec{a}, \vec{\alpha}, u$ as in the statement of the theorem. We have thus reduced (\ref{FMT13}) to
 \begin{equation*}
\begin{split}
  \int_{\mathcal{C}_k} \cdots \int_{\mathcal{C}_k} \int_{D_k}\hspace{-3mm}\cdots \int_{D_k}  \prod_{j = 1}^n  \frac{F(v_j, w_j)dw_j }{w_j - v_{\sigma(j)} }  \prod_{i = 1}^n  dv_i =  \int_{\mathcal{C}_{k+1}} \hspace{-3mm}\cdots \int_{\mathcal{C}_{k+1}} \int_{D_k}\hspace{-3mm}\cdots \int_{D_k}  \prod_{j = 1}^n  \frac{F(v_j, w_j)dw_j }{w_j - v_{\sigma(j)} }  \prod_{i = 1}^n  dv_i,
\end{split}
\end{equation*}
and in particular,  our use of Proposition \ref{PE1v2} also implies that the integrals above are absolutely convergent. The latter, by Fubini's theorem, allows us to exchange the order of the integrals, reducing the proof of (\ref{FMT13}) to showing that
 \begin{equation*}
\begin{split}
\int_{D_{k}}\hspace{-3mm}\cdots \int_{D_{k}} \Bigg[  \int_{\mathcal{C}_k} \hspace{-3mm}\cdots \int_{\mathcal{C}_k}  \prod_{j = 1}^n \frac{F(v_j, w_j) }{w_j - v_{\sigma(j)} }  \prod_{i = 1}^n  d\mu(v_i)  -  \int_{\mathcal{C}_{k+1}} \hspace{-3mm}\cdots \int_{\mathcal{C}_{k+1}}  \prod_{j = 1}^n \frac{F(v_j, w_j) }{w_j - v_{\sigma(j)} }  \prod_{i = 1}^n  d\mu(v_i) \Bigg]  \prod_{j = 1}^n dw_j = 0.
\end{split}
\end{equation*}
The above is now immediate since by Cauchy's theorem we may deform the contour $\mathcal{C}_k$ to $\mathcal{C}_{k+1}$ without crossing any poles. Thus, we have established (\ref{R2}), (\ref{R3}), (\ref{R4}) and (\ref{R5}), which concludes the proof of the first part of the theorem.\\

{\bf \raggedleft Part II.} We let as before $\theta_0 = \min(\vec{\alpha}) - \max(\vec{a}) > 0$, and note that in Part I we proved the statement of the theorem when $\theta_0 > 1$. In this part we prove the theorem in the case when $\theta_0 \in (0, 1]$. For $\theta > - \theta_0$ we let $\mathbb{P}_\theta$ denote the probability measure of the mixed polymer model where the $X_{i,j}$ (as in Definition \ref{Dmixed}) have inverse-gamma distribution with parameters $\alpha_i - a_j + \theta$. We also write $\mathbb{E}_\theta$ for the expectation with respect to this measure. With the notation as in the statement of the theorem we have from the first part of the proof the following statement for $\theta > 1- \theta_0$
\begin{equation}\label{S2P2}
 \mathbb{E}_\theta \left[ e^{-u \ZP(\tau)} \right] = \det \left( I + K^{\tau}_{u, \theta} \right)_{L^2(C_{a,3\pi/4})},
\end{equation}
where
\begin{equation}\label{P2kernel2}
K^{\tau}_{u, \theta} (v,v') = \frac{1}{2\pi \i} \int_{D_v} \frac{\pi }{\sin(\pi (v-w))}  \prod_{n = 1}^N \frac{\Gamma(v - a_n)}{\Gamma(w - a_n)} \prod_{m = 1}^M \frac{\Gamma(\alpha_m + \theta - w)}{\Gamma(\alpha_m + \theta - v)} \frac{u^{w-v} e^{\tau(w^2 - v^2)/2}dw}{w - v'}.
\end{equation}
The goal is then to show that (\ref{S2P2}) holds when $\theta = 0$ and we accomplish this in three steps. In the first step we show that the left side of (\ref{S2P2}) has an analytic continuation in $\theta$ in the right half-plane $\{ z \in \mathbb{C}: \Re(z) > -\theta_0\}$. In the second step we construct an auxiliary function $\tilde{G}_u(z)$ and show that the latter is analytic in an open neighborhood of $[0, 1]$. Finally, in Step 3 we prove that both the left and right sides of (\ref{S2P2}) are equal to $\tilde{G}_u(\theta)$, which concludes the proof of the theorem.\\

{\bf \raggedleft Step 1.} We define
$$F_u(\theta) = \mathbb{E}_\theta \left[ e^{-u \ZP} \right],$$
and note that the function $F_u(\theta)$ makes sense for any $\theta > -\theta_0$. In this step we show that $F_u(\theta)$ has an analytic continuation to $V = \{ z \in \mathbb{C}: \Re(z) > -\theta_0\}$. To see this, notice that from (\ref{PathEnergy}) and (\ref{PathEnergy2}) we have
\begin{equation}\label{PartFunTheta}
\begin{split}
&F_u(\theta) = \frac{\int_{\mathbb{R}_+^{NM}} H_u( \vec{x}, \vec{B} ) \prod_{i = 1}^N \prod_{j = 1}^M x^{-\alpha_j + a_i - \theta - 1}_{ij}  \exp \left(- x_{ij}^{-1}\right)  dx_{ij}  \prod_{i = 1}^N \mathbb{P}(dB_i)  }{ \prod_{i = 1}^N \prod_{j = 1}^M \Gamma(\alpha_j - a_i + \theta )}, \\
&\mbox{ where } H_u(\vec{x}, \vec{B})  =  \exp \left( - u \sum_{n = 1}^N\, \sum_{\phi^d:(-M,1) \nearrow (-1,n)} \int_{\phi^{sd}:(0,n) \nearrow (\tau, N)} e^{E(\phi)}d \phi^{sd} \right), \mbox{ and }\\
&E(\phi) := \sum_{(i,j) \in \phi^d} x_{i,j} + B_n(s_n) + \sum_{k = n+1}^N \big(B_{k}(s_{k}) - B_{k}(s_{k-1})\big).
\end{split}
\end{equation}
In the top line (\ref{PartFunTheta}) we recall that $B_1, \dots, B_N$ are independent Brownian motions with drifts $\{a_n\}_{n = 1}^N$. Next if $z_0 \in V$ and $Re(z_0) > -\theta_0 + 2\delta$ for some $\delta > 0$ we have
\begin{equation}
\begin{split}
&\sup_{z: |z - z_0| \leq \delta} \int_{\mathbb{R}_+^{NM}}   \left| H_u( \vec{x}, \vec{B}) \right| \prod_{i = 1}^N \prod_{j = 1}^M\left| x^{-\alpha_j + a_i - z - 1}_{ij} \right| \exp \left(- x_{ij}^{-1}\right) dx_{ij}  \prod_{i = 1}^N \mathbb{P}(dB_i) \leq \hspace{-3mm} \sup_{z:|z - z_0| \leq \delta} \int_{\mathbb{R}_+^{NM}}  \\
&  \prod_{i = 1}^N \prod_{j = 1}^M  x^{-Re(z) -\alpha_j + a_i  - 1}_{ij}  \exp \left(- x_{ij}^{-1}\right) dx_{ij} =\prod_{i = 1}^N \prod_{j = 1}^M \sup_{z: |z - z_0| \leq \delta} \Gamma (Re(z) + \alpha_j - a_i)^{NM} < \infty.
\end{split}
\end{equation}
The last statement and Theorem \ref{Mattner} imply that the numerator in the first line of (\ref{PartFunTheta}) is analytic in $V$ and so is $\prod_{i = 1}^N \prod_{j = 1}^M  \Gamma^{-1}(\alpha_j - a_i + z)$. This proves the claim.\\

{\bf \raggedleft Step 2.} In this step we construct an auxiliary function $\tilde{G}_u(z)$ for $z$ in a neighborhood of $[0, 1]$ as follows. Let $\delta_0 = \min (d/4, a - \max(\vec{a}), (\min(\vec{\alpha}) - b)/4)$, where $a,b$ are as in the statement of the theorem. For $z \in U := \{z \in \mathbb{C}: d(z, [0, 1]) < \delta_0 \}$ we define
\begin{equation}\label{GtildeDef}
\tilde{G}_u(z) = \det \left( I + \tilde{K}_{u,z}\right)_{L^2(C_{a, 3\pi/4})},
\end{equation}
where the operator $\tilde{K}_{u,z}$ is defined in terms of its integral kernel
\begin{equation}
\tilde{K}_{u,z}(v,v') =  \frac{1}{2\pi i} \int_{\tilde{D}_v} \frac{\pi }{\sin(\pi (v-w))}  \prod_{n = 1}^N \frac{\Gamma(v - a_n)}{\Gamma(w - a_n)} \prod_{m = 1}^M \frac{\Gamma(\alpha_m + z - w)}{\Gamma(\alpha_m + z - v)} \frac{u^{w-v} e^{\tau(w^2 - v^2)/2}dw}{w - v'},
\end{equation}
where $\tilde{D}_v = D_v(b, \pi/4, \delta_0)$ is as in Definition \ref{ContE}. To show that $\tilde{G}_u(z)$ is well-defined we start by applying Proposition \ref{PE1v2} with  $T = \tau$, $K = \overline{U}$, and $a, u, M,N, \vec{a}, \vec{\alpha}$ as above to conclude that for $v, v' \in C_{a, 3\pi/4}$ we have
\begin{equation}\label{UBA1}
\sup_{z \in U}  \int_{\tilde{D}_v} \left|  \frac{\pi }{\sin(\pi (v-w))}  \prod_{n = 1}^N \frac{\Gamma(v - a_n)}{\Gamma(w - a_n)} \prod_{m = 1}^M \frac{\Gamma(\alpha_m + z - w)}{\Gamma(\alpha_m + z - v)} \frac{u^{w-v} e^{\tau(w^2 - v^2)/2}dw}{w - v'} \right| |dw|  < \infty,
\end{equation}
where $|dw|$ denotes integration with respect to arc-length and we also used that $|\Gamma(z)|$ is uniformly bounded over compact sets that are disjoint from $\mathbb{Z}$. Equation (\ref{UBA1}) implies that $\tilde{K}_{u,z}(v,v')$ is well-defined for each $v,v' \in C_{a, 3\pi/4}$ and moreover, by Theorem \ref{Mattner}, we know that $\tilde{K}_{u,z}(v,v')$ is analytic in $U$ for each fixed $v,v' \in C_{a, 3\pi/4}$.

We next apply Proposition \ref{PE2v2} for $T= \tau$, $a,b,\delta_0$ as above, $\phi = 3\pi/4$, $M, N, \vec{a}, \vec{\alpha}, u$ as in the statement of the theorem. The conclusion is that we can find a constant $C_0 > 0$ such that for all $v, v' \in  C_{a, 3\pi/4}$ we have
\begin{equation}\label{UBA2}
\sup_{z \in U} |\tilde{K}_{u,z}(v,v')|  \leq \frac{C_0}{1 + |v|^2}.
\end{equation}
Equation (\ref{UBA2}) and Lemma \ref{AT} imply that $\tilde{G}_u(z) $ is well-defined and analytic in $z \in U$.\\

{\bf \raggedleft Step 3.} In this step we show that both the left and right sides of (\ref{S2P2}) are equal to $\tilde{G}_u(\theta)$ for all $\theta \in [0,1]$. In particular, setting $\theta = 0$ concludes the second part of the proof of the theorem.

We observe that if $x \in (1 - \theta_0, 1]$ we know that $\tilde{G}_u(x)$ equals the right side of (\ref{S2P2}) with $\theta$ replaced with $x$ and $D_v = D_v(a,b,\delta_0)$. In view of (\ref{S2P2}) we conclude that $\tilde{G}_u(x) = F_u(x)$ for $x \in (1 - \theta_0, 1]$ and since both functions are analytic in a neighborhood of $[0, 1]$ we conclude that the left side of (\ref{S2P2}) equals $\tilde{G}_u(\theta)$ for $\theta \in [0,1]$, cf. \cite[Chapter 2, Corollary 4.9]{Stein}.

By Cauchy's theorem we can deform the $D_v(b, \pi/4,d)$ contour in (\ref{P2kernel2}) to $D_v(b, \pi/4, \delta_0)$ (here $\delta_0$ was defined in Step 2) without affecting the value of the integral, at which point we recognize that $ \det \left( I + K^{\tau}_{u,\theta} \right)_{L^2(C_{a,3\pi/4})} = \tilde{G}_u(\theta)$. In other words, the right side of (\ref{S2P2}) equals $\tilde{G}_u(\theta)$ for all $\theta \in [0,1]$. This suffices for the proof.

%
\section{Asymptotic analysis of Fredholm determinants}\label{Section3} In this section we investigate the limit of the Fredholm determinant $ \det (I +K_u)_{L^2(C_a, 3\pi/4)}$ from Theorem \ref{LGPT1} when the parameters are appropriately scaled (the precise scaling is given in Definition \ref{DefScaleS3}). The exact convergence statement, which is the main technical result of the paper, is given in Theorem \ref{mainThmGUE} and is proved in Section \ref{Section3.3} after some preliminary results are presented in Section \ref{Section3.2}. Theorem \ref{mainThmGUE} is used later in Section \ref{Section4} to establish Theorems \ref{LGPCT} and \ref{BBP}.

%
\subsection{Limit of $\det (I +K_u)_{L^2(C_a, 3\pi/4)}$ }\label{Section3.1} In this section we show that when the parameters in the definition of $\det (I +K_u)_{L^2(C_a, 3\pi/4)}$ from Theorem \ref{LGPT1} are scaled appropriately the Fredholm determinant converges to a limiting one. We first turn to specifying how we scale parameters.

\begin{definition}\label{DefScaleS3}
Let $\Psi(x)$ denote the digamma function, i.e.
\begin{equation}\label{digammaS3}
\Psi(x) = \frac{\Gamma'(x)}{\Gamma(x)} = - \gamma_{E} + \sum_{n = 0}^\infty \left(\frac{1}{n + 1} - \frac{1}{n+z} \right),
\end{equation}
where $\gamma_{E}$ is the Euler constant.\\

Suppose that $M, N \geq 1$ and $\theta > 0$ are given. We let $z_{c}(M,N)$ denote the maximizer of 
$$W_{M,N}(x) := N \Psi(x) + M \Psi(\theta - x)$$
on the interval $(0, \theta)$. Notice that the above expression converges to $- \infty$ as $x \rightarrow 0+$ or $x \rightarrow \theta-$ and also the function $W_{M,N}(x)$ is strictly concave, hence the maximum exists and is unique. We let $W_{M,N}$ denote $W_{M,N}(z_c)$.

For any $\alpha > 0$ we define
$$\sigma_\alpha := \left( \sum_{n = 0}^\infty \frac{\alpha}{(n+z_c)^3} +  \sum_{n = 0}^\infty \frac{1}{(n+\theta - z_c)^3} \right)^{1/3}.$$

The way we scale the parameters is as follows. We will let $M,N \rightarrow \infty$ while $M\geq N$ and $\alpha = N/M \geq \delta$, where $\delta > 0$ is fixed.
We set $\alpha = \alpha(M,N) = N/M$ and for a given $x \in \mathbb{R}$ we set 
$$u = u(x,M,N):= e^{W_{M,N} - M^{1/3} \sigma_{\alpha} x}.$$
Notice that if $\alpha \in [\delta, 1]$ and $\theta > 0$ is fixed then $\sigma_\alpha$ is positive and bounded away from $0$ and $\infty$, see e.g. (\ref{ULBalpha}).   Finally, we fix $\rr, \rl \in \mathbb{Z}_{\geq 0}$ and vectors $\vec{x} = ( x_1, \dots, x_{\rr}) \in \mathbb{R}^{\rr}$, $\vec{y} = (y_1, \dots, y_{\rl}) \in \mathbb{R}^{\rl}$ such that $\min(\vec{y}) > \max(\vec{x})$. If $\rr = 0$ we set $  \max(\vec{x}) = -\infty$ and if $\rl = 0$ we set $\min(\vec{y}) = \infty$. We also let $\vec{a} \in \mathbb{R}^{N + \rr}$ and $\vec{\alpha} \in \mathbb{R}^{M + \rl}$ be such that 
$$a_i(M,N)  \hspace{-0.5mm} = \hspace{-0.5mm} z_c + x_i \sigma_\alpha^{-1} M^{-1/3} + o(M^{-1/3}) \mbox{ for $i = 1, \dots, \rr$}, \hspace{0.5mm} a_i(M,N) = 0 \mbox{ for $i = \rr+1, \dots, \rr+ N$}, $$
$$\alpha_i(M,N) \hspace{-0.5mm}=\hspace{-0.5mm} z_c + y_i \sigma_\alpha^{-1}  M^{-1/3} + o(M^{-1/3}) \mbox{ for $i = 1, \dots, \rl$}, \hspace{0.5mm} \alpha_i(M,N) = \theta \mbox{ for $i = \rl+1, \dots, \rl+ M$}. $$
We also let $\mu = (\max(\vec{x}) + \min(\vec{y}))/2$ with the convention that $\mu = 0$ if $\rr = \rl = 0$, $\mu = \max(\vec{x}) + 1$ if $\rl = 0$ and $\rr > 0$, $\mu = \min(\vec{y}) - 1$ if $\rr = 0$ and $\rl > 0$. We always assume that $M$ is sufficiently large so that so that $\min(\vec{\alpha}) > z_c +  \mu \sigma_\alpha^{-1}  M^{-1/3} > \max( \vec{a})$ and $\min(\vec{\alpha})  > 0$.

We further let $K_u$ be as in Theorem \ref{LGPT1} for the vectors $\vec{\alpha}$, $\vec{a}$ and parameter $u$ as above, $a = z_c + \mu \sigma_\alpha^{-1}  M^{-1/3}$ and $b,d$ arbitrarily chosen so as to satisfy the conditions of that theorem. We also define the contour $\tilde{D}$ to be the contour that consists of the vertical segment connecting $\mu + \rho- \i \rho$ and $\mu + \rho + \i \rho$ and for $|\Im(z)| \geq \rho$ it agrees with $C_{\mu, \pi/4}$  from Definition \ref{ContV}, where $\rho > 0$ is chosen sufficiently small so that $\min(\vec{y}) > \mu + \rho$, see Figure \ref{S3_1}. The contour $\tilde{D}$ is oriented to have increasing imaginary part.  If $\rl = 0$ we let $\rho = 1$.
\end{definition}

The main result of the section is as follows.

\begin{theorem} \label{mainThmGUE}
Fix $x \in \mathbb{R}$ and assume the same notation and scaling as in Definition \ref{DefScaleS3}. Then
\begin{equation}\label{mainLimitGUEeq}
\lim_{M \rightarrow \infty} \det \left( I + K_u \right)_{L^2(C_{a ,3\pi/4})} =  \det \left( I + K_x^\infty \right)_{L^2(C_{\mu ,3\pi/4})},
\end{equation}
where the operator $ K_x^\infty$ is defined in terms of its integral kernel
\begin{equation}\label{kernelMain}
\begin{split}
 &K_x^\infty(v,v') = \frac{1}{2\pi \i} \int_{\tilde{D}}g^{\infty,x}_{v,v'}(w) dw, \mbox{ with }  \\
&g^{\infty,x}_{v,v'}(w)  = \prod_{n = 1}^{\rr} \frac{( {w} - x_n )}{({v}  - x_n  )} \prod_{m = 1}^{\rl} \frac{( y_m -  {v} )}{(y_m  -{w} )}\cdot  \frac{\exp(-v^3/3 + w^3/3 - x w + x v)}{(v-w)(w - v')},
\end{split}
\end{equation}
where $C_{\mu ,3\pi/4}$ is as in Definition \ref{ContV} and $\tilde{D}$ is as in Definition \ref{DefScaleS3}.
\end{theorem}
\begin{remark}
The Fredholm determinant on the right side of (\ref{mainLimitGUEeq}) is as in Section \ref{Section2.1}. In particular, part of the statement of the theorem is that the kernel $ K_x^\infty(v,v')$ is well-defined, each term in the series (\ref{fredholmDefS2}) is finite and the series converges for the given choice of parameters.
\end{remark}

%
\subsection{Relevant estimates}\label{Section3.2} In this section we introduce some notation and establish a few results that will be useful for the proof of Theorem \ref{mainThmGUE}, which is given in the next section. 
\begin{definition}\label{DefDM}
Define 
\begin{equation}\label{GfunS3}
G_{M,N}(z) := N \log \Gamma (z) - M \log \Gamma (\theta - z) - W_{M,N}   z - C_{M,N},
\end{equation}
where $C_{M,N} = N \log \Gamma (z_c) - M \log \Gamma (\theta - z_c) - W_{M,N}  z_c $, and $z_c, W_{M,N}$ are as in Definition \ref{DefScaleS3}.
We also define 
$$G_\alpha (z) = M^{-1} G_{M,N}(z) =  \alpha \log \Gamma (z) -  \log \Gamma (\theta - z) - M^{-1} W_{M,N}   z - M^{-1}C_{M,N},$$
where we recall that $\alpha = N/M \in [\delta, 1]$ as in Definition \ref{DefScaleS3}. Continuing with the notation from Definition \ref{DefScaleS3} we define the contours $D_M$ through $D_M = z_c + \sigma_\alpha^{-1} M^{-1/3} \cdot \tilde{D}$, see Figure \ref{S3_1}. The contours $D_M$ are oriented to have increasing imaginary part and we write $D_M^{\pm}$ for the slope $\pm 1$ pieces of $D_M$ and $D_M^0$ for the vertical part of the contour.
\end{definition}
\begin{figure}[h]
\scalebox{0.6}{\includegraphics{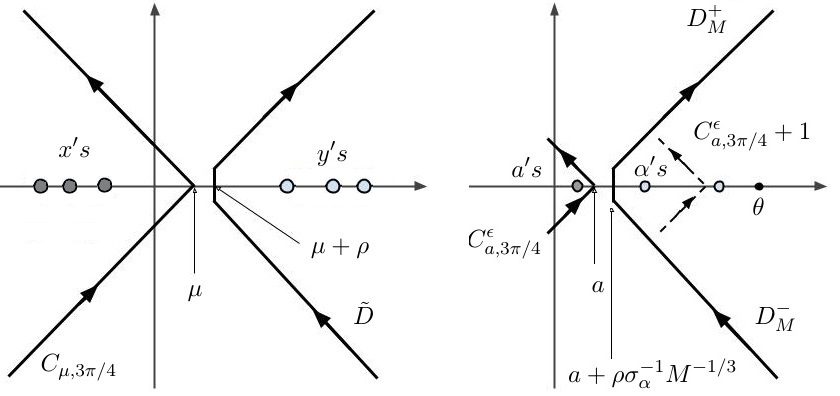}}
\captionsetup{width=\linewidth}
 \caption{The left part depicts the contour $\tilde{D}$ from Definition \ref{DefScaleS3} and the contour $C_{\mu, 3\pi/4}$. The dark gray dots are $x_1 \dots, x_{\rr}$ and the light gray ones are $y_1, \dots, y_{\rl}$. The right part depicts the contour $C^{\epsilon}_{a, 3\pi/4}$ as defined in Step 1 of the proof of Theorem \ref{mainThmGUE} in Section \ref{Section3.3} and the contour $D_M = D_M^- \cup D_M^0 \cup D_M^+$. The dark gray dot represents $a_i$ for $i = 1, \dots, \rr$ and the light gray ones the points $m + \alpha_i$ for $m \geq 0$ and $i = 1, \dots, \rl$. Some of the distances are exaggerated. }
\label{S3_1}
\end{figure}

We next state several results for future use. The first three are proved in Section \ref{Section9.4} and the third one is proved in Section \ref{Section10.3}
\begin{lemma}\label{analGS3}Fix $\theta > 0$, $\delta \in (0,1)$ and assume that $M \geq N \geq 1$, $N/M \in [\delta, 1]$. 
There exist constants $C > 0$ and $r > 0$ depending on $\theta$ and $\delta$ such that $G_{\alpha}$ is analytic in the disc $|z - z_c| < r$ and the following hold for $|z - z_c| \leq r$:
\begin{equation}\label{powerGS3}
\begin{split}
&\left|G_\alpha(z)+ (z-z_c)^3\sigma^3_\alpha/3 \right| \leq C |z-z_c|^4;\\
&\Re[G_\alpha(z) ] \geq (\sqrt{2}/2)^3 |z-z_c|^3\sigma^3_\alpha/6  \mbox{ when $z \in C_{z_c, \phi}$ with $\phi =\pi/4$; }\\
&\Re[G_\alpha(z) ] \leq -  (\sqrt{2}/2)^3 |z-z_c|^3 \sigma^3_\alpha/6  \mbox{ when $z \in C_{z_c, \phi}$ with $\phi = 3\pi/4 $ }.
\end{split}
\end{equation}
In the above equations $z_c, \sigma_\alpha$ are as in Definition \ref{DefScaleS3} and $C_{a, \phi}$ is as in Definition \ref{ContV}.
\end{lemma}

\begin{lemma}\label{ReleL1S3} Fix $\theta > 0$, $\delta \in (0,1)$ and assume that $M \geq N \geq 1$, $N/M \in [\delta, 1]$. Let $z(r) = z_c  + r e^{i\phi}$ with $z_c$ as in Definition \ref{DefScaleS3}. Then we have
\begin{equation}\label{eqG1S3}
\begin{split}
&\Re \left[\frac{d}{dr} G_{M,N}(z(r)) \right] \leq 0, \mbox{ provided that $\phi = 3\pi/4$ or $\phi = 5\pi/4$,}\\
&\Re \left[\frac{d}{dr} G_{M,N}(z(r)) \right] \geq 0, \mbox{ provided that $\phi = \pi/4$ or $\phi = -\pi/4$.}
\end{split}
\end{equation}
In addition, if $\delta_1  > 0$ is given then there is a constant $c > 0$ depending on $\delta_1, \delta, \theta$ such that
\begin{equation}\label{eqG2S3}
\begin{split}
\Re \left[G_{M,N}(z(r)) \right] \leq -c (M+N) r \log(1 + r) \mbox{ for any $r \geq \delta_1$ provided $\phi = 3\pi/4$ or $\phi = 5\pi/4$,}\\
\Re \left[G_{M,N}(z(r)) \right] \geq c (M+N) r \log(1 + r) \mbox{ for any $r \geq \delta_1$ provided $\phi = \pi/4$ or $\phi = -\pi/4$,}\\
\end{split}
\end{equation}
\end{lemma}

\begin{lemma}\label{horGS3}Fix $\theta, A > 0$, $\delta \in (0,1)$ and assume that $M \geq N \geq 1$, $N/M \in [\delta, 1]$. 
There exist constants $M_0, C_0 > 0$ depending on $\delta, \theta, A$ such that if $M \geq M_0$, $x \in [-A, A]$ and $z(r) =  z_c  + r e^{\i\phi}$ with $z_c$ as in Definition \ref{DefScaleS3}, $r\geq 0$ and $\phi \in \{\pi/4, 3\pi/4, 5\pi/4, 7\pi/4\}$ we have
\begin{equation}\label{horGS3Eq}
\left| \Re [G_{M,N}(z(r) + xM^{-1/3})  - G_{M,N}(z(r))] \right| \leq C_0 M^{2/3} \cdot ( 1 + r).
\end{equation} 
\end{lemma}

\begin{proposition}\label{PKBS3}
Let $A \geq 0$ and $x\geq -A$ be given and let $M, N, a, u(x, M, N)$ be as in Definition \ref{DefScaleS3}. Then for any $\epsilon > 0$ there exist positive constants $M_0, c ,C > 0$ depending on $A,\epsilon$ and the parameters in  Definition \ref{DefScaleS3} such that if $M \geq M_0$, $v, v' \in C_{a,3\pi/4}$  we have
\begin{equation}\label{KernelBoundV0S3}
\left|K_u (v,v')\right| \leq CM^{1/3} \mbox{ and if $|v - a| \geq \epsilon$ we further have }\left|K_u (v,v')\right| \leq C e^{-c_0 M \log (1 + |v|)}.
\end{equation}
If $\rr = \rl = 0$ then $M_0, c ,C > 0$ depend only on $\theta$, $\delta$, $A$ and $\epsilon$. 
\end{proposition}

\begin{lemma}\label{SineBound} There exists a universal constant $\tilde{C} > 0$ such that if $s \in \mathbb{C}$ with $d(s, \mathbb{Z}) = \delta > 0$ then
\begin{equation}\label{IneqSineScale}
\left| \frac{\pi}{\sin(\pi s)} \right| \leq \frac{\tilde{C}}{\delta}.
\end{equation}
\end{lemma}
\begin{proof} Let $s = x+ \i y$ and observe that by the periodicity of the sine function it suffices to prove the lemma for $x \in [-1/2, 1/2]$. Notice that by the triangle inequality we have
$$| \sin(\pi s)| = \left| \frac{e^{\i s} - e^{-\i s}}{2\i} \right| \geq \frac{e^{|y|} - e^{-|y|}}{2} \geq \frac{e^{|y|} - 1}{2} \geq |y|/2.$$
If $|y| \geq 1$ and $x \in [-1/2, 1/2]$ the latter implies that 
$$|\sin(\pi s)| \delta^{-1} \geq  (|y|/2) (\sqrt{x^2 + y^2})^{-1} \geq (|y|/2) (2|y|)^{-1} = 1/4,$$
which implies that (\ref{IneqSineScale}) holds for all such $s$ with $\tilde{C} = 4\pi$. On the other hand, if $-1 \leq y \leq 1$ and $x \in [-1/2, 1/2]$ we have that 
$$|\sin(\pi s) |\delta^{-1} = \left|\frac{\sin(\pi s)}{s} \right|,$$
and since $s^{-1} \sin(\pi s)$ is a non-vanishing analytic function in the region $(x,y) \in [-1/2,1/2] \times [-1,1]$ we know that there exists $C_1 > 0$ such that $|s^{-1} \sin(\pi s)| \geq C_1$ on this region. We may thus take $\tilde{C} = \pi \cdot \max(C_1^{-1}, 4)$ and then (\ref{IneqSineScale}) would hold for all $s \in \mathbb{C}$.
\end{proof}
\begin{lemma}\label{SumBound} For any $\delta > 0$ and $\Delta > 0$ there exists $M_0 > 0$ such that if $M \geq M_0$ we have
\begin{equation}\label{SumBoundE}
\sum_{n = 1}^\infty   \frac{n^{n/2} \Delta^nM^{n/3}}{n!} \leq e^{\delta M}.
\end{equation}
\end{lemma}
\begin{proof} From \cite[(1) and (2)]{Rob} we have for each $n \geq 1$ that
\begin{equation}\label{Stirling}
 n! = \sqrt{2\pi} n^{n+1/2} e^{-n} e^{r_n} \mbox{, where } (12n+1)^{-1} < r_n < (12n)^{-1}
\end{equation}
and so
$$ \frac{n^{n/2} \Delta^nM^{n/3}}{n!}  \leq \frac{e^{n/2}\Delta^nM^{n/3} (2\pi n)^{-1/4} }{\sqrt{n!}}\leq \frac{(2\Delta)^nM^{n/3} }{\sqrt{n!}} =: A_n.$$

We also see that 
$$\frac{A_{n+1}}{A_n} = \frac{2 \Delta M^{1/3}}{\sqrt{n+1}},$$
which implies that 
$$A_{n+1} \geq A_n \mbox{ if $\sqrt{n+1} \leq 2 \Delta M^{1/3}$ and } A_{n+1} \leq A_n \mbox{ if $\sqrt{n+1} \geq 2 \Delta M^{1/3}$}.$$
The above inequalities imply that for $m = \lfloor 4 \Delta^2 M^{2/3} \rfloor$ we have  $A_m \geq A_n$ for all $n \in \mathbb{N}$. On the other hand, (\ref{Stirling}) implies
$$A_m \leq \frac{(4\Delta^2 M^{2/3})^{m/2} e^{m/2}}{e^{(m/2) \log m}} \leq \left( 1 + m^{-1}\right)^{m/2} \cdot e^{m/2} \leq 2 e^{m/2},$$
where the latter inequality holds provided $M_0$ is sufficiently large and $M \geq M_0$. Combining all of the above estimates we conclude that for all large $M$
$$\sum_{n = 1}^\infty   \frac{n^{n/2} \Delta^nM^{n/3}}{n!} \leq \sum_{n = 1}^\infty A_n \leq M \cdot A_m + \sum_{n = M+1}^\infty A_n \leq 2M e^{m/2} + 1.$$
The latter clearly implies (\ref{SumBoundE}) for all large $M$, as the right side is $O(\exp(M^{2/3} \log M)).$
\end{proof}

%
\subsection{Proof of Theorem \ref{mainThmGUE}}\label{Section3.3}We assume the same notation as in Theorem \ref{mainThmGUE} as well as Section \ref{Section3.2} above. For the sake of clarity we split the proof into four steps. In the first step we truncate the contour $C_{a, 3\pi/4}$ and show that it suffices to prove the theorem for the truncated contour using  Proposition \ref{PKBS3} and Lemma \ref{SumBound}. In the second step we rewrite the Fredholm determinant for the truncated contour in a way that is suitable for the application of Lemma \ref{FDKCT}. In the third step we state several estimates, which show that the conditions of Lemma \ref{FDKCT} are satisfied, and the application of the latter and Lemma \ref{FDCT} concludes the proof of theorem. The estimates from step three are then established in the fourth and final step by appealing to Lemmas \ref{analGS3}, \ref{ReleL1S3},\ref{horGS3} and  \ref{SineBound}.\\

{\raggedleft {\bf Step 1.}} For $\epsilon > 0$ we let $C^{\epsilon}_{a, 3\pi/4}$ denote the portion of $C_{a, 3\pi/4}$ inside $B_{\epsilon}(a)$ -- the disk of radius $\epsilon$, centered at $a$. We claim that there exists $\epsilon\in (0,1/2)$ such that
\begin{equation}\label{TWRed1}
\lim_{M \rightarrow \infty} \det (I + K_{u})_{L^2(C^{\epsilon}_{a, 3\pi/4})} = \det \left( I + K_x^\infty \right)_{L^2(C_{\mu ,3\pi/4})}.
\end{equation}
We prove (\ref{TWRed1}) in the next steps. Here we assume its validity and conclude the proof of the theorem.\\

Let us write $C_{a, 3\pi/4} = C_0 \cup C_1$ where $C_0 = C^{\epsilon}_{a, 3\pi/4}$ is the portion contained inside $B_\epsilon(a)$ and $C_1$ is the portion outside $B_\epsilon(a)$.  From the definition of Fredholm determinants, see Section \ref{Section2.1}, we know that
\begin{equation}\label{FMT8v2}
\begin{split}
 & \det (I + K_u)_{L^2(C_{a, 3\pi/4})} - \det (I + K_{u})_{L^2(C^{\epsilon}_{a, 3\pi/4})} = \sum_{n = 1}^\infty \frac{B_n}{n!} , \mbox{ where }
\end{split}
\end{equation} 
\begin{equation}
B_n = \sum_{ \epsilon_1, \dots, \epsilon_n \in \{0, 1\}: \sum_i \epsilon_i > 0} \frac{1}{(2\pi \i )^n} \int_{C_{\epsilon_1}} \cdots \int_{C_{\epsilon_n}} \det \left[ K_u(v_i, v_j) \right]_{i,j = 1}^n  \prod_{i = 1}^n dv_i.
\end{equation}
By Hadamard's inequality and Proposition \ref{PKBS3} we have that if $v_i \in C_{\epsilon_i}$ for $i = 1, \dots, n$ and $M \geq M_0$
\begin{equation}
\left|  \det \left[ K_u(v_i, v_j) \right]_{i,j = 1}^n \right| \leq n^{n/2} M^{n/3} \prod_{i =1}^n C e^{-c\epsilon_i M \log( 1 + |v_i|)},
\end{equation}
where the constants $c,C$ and $M_0$ are as in Proposition \ref{PKBS3}.
Consequently, for all large $M$ we have
\begin{equation}\label{FMT9v2}
|B_n| \leq  \sum_{ \epsilon_1, \dots, \epsilon_n \in \{0, 1\}: \sum_i \epsilon_i > 0 } n^{n/2} C^n M^{n/3}e^{-c'M \sum_{i = 1}^n \epsilon_i}\leq  e^{-c'M} \cdot n^{n/2} (2C)^nM^{n/3} ,
\end{equation}
where we used that the sum is over $\epsilon_i$ not all equal to $0$ and $c' = (c/2) \log(1 + \epsilon)$. From (\ref{FMT9v2}) and Lemma \ref{SumBound} applied to $\delta = c'/2$ and $\Delta = 2C$ we deduce
\begin{equation}\label{FMT10v2}
 \sum_{n = 1}^\infty \frac{|B_n| }{n!} \leq   e^{-c'M}  \sum_{n = 1}^\infty   \frac{n^{n/2} (2C)^nM^{n/3}}{n! } \leq e^{-c'M/2}.
\end{equation}
 Combining (\ref{TWRed1}) with (\ref{FMT8v2}) and (\ref{FMT10v2}) we conclude the statement of the theorem.\\

{\raggedleft {\bf Step 2.}} In this step we rewrite $ \det (I + K_{u})_{L^2(C^{\epsilon}_{a, 3\pi/4})}$ in a way that is more suitable for the application of Lemma \ref{FDKCT}. Let $r$ be as in Lemma \ref{analGS3} and fix $\epsilon > 0$ such that $\epsilon < \min (r/2, 1/2)$. We will prove (\ref{TWRed1}) for this choice of $\epsilon$, which is fixed in the rest of the proof.

By Cauchy's theorem we may deform the $D_v(b,d, \pi/4)$ contour in the definition of $K_u(v,v')$ to $D_M$ from Definition \ref{DefDM} without affecting the value of the kernel as long as $M$ is sufficiently large. Indeed, notice that by our choice of $\epsilon \leq 1/2$ we have that $C^{\epsilon}_{a, 3\pi/4} + 1$ lies to the right of $D_M$ and so we do not cross any poles while deforming $D_v(b,d, \pi/4)$ to $D_M$, see Figure \ref{S3_1}. The decay estimates necessary to deform the contour near infinity come from Proposition \ref{PE1v2} applied to $K = [0,\theta]$, $T =0$ and $ \vec{a}, \vec{\alpha}, u$ as in the statement of this theorem. 

Utilizing the above contour, we may perform a change of variables $v_i = \sigma_\alpha^{-1} M^{-1/3} \tilde{v}_i + z_c$ and $w = \sigma_{\alpha}^{-1} M^{-1/3} \tilde{w} + z_c$ to rewrite 
\begin{equation}\label{TWRed2v1}
\det (I + K_u)_{L^2(C^{\epsilon}_{a, 3\pi/4})} = \det (I + K^M_x)_{L^2(C_{\mu, 3\pi/4})},
\end{equation}
where 
$$K_x^M(\tilde{v},\tilde{v}') = {\bf 1}_{|\tilde{v} - \mu| \leq \epsilon \sigma_\alpha  M^{1/3}} \cdot {\bf 1}_{|\tilde{v}' - \mu| \leq \epsilon  \sigma_\alpha  M^{1/3}}\frac{1}{2\pi \i } \int_{\tilde{D}} g^{M,x}_{\tilde{v}, \tilde{v}'}(\tilde{w}) d\tilde{w},$$
where $\tilde{D} = \sigma_\alpha M^{1/3} (D_M - z_c)$ is as in Definition \ref{DefScaleS3} and $g^{M,x}_{\tilde{v}, \tilde{v}'}(\tilde{w}) = A_1^M \cdot A_2^M \cdot A_3^M$, where 
\begin{equation*}
\begin{split}
& A_1^M = \prod_{n = 1}^{\rr} \frac{\Gamma(\sigma_\alpha^{-1} M^{-1/3} \tilde{v} + z_c - a_n)}{\Gamma(\sigma_{\alpha}^{-1} M^{-1/3} \tilde{w} + z_c - a_n)} \prod_{m = 1}^{\rl} \frac{\Gamma(\alpha_m - \sigma_{\alpha}^{-1} M^{-1/3} \tilde{w} -z_c)}{\Gamma(\alpha_m - \sigma_\alpha^{-1} M^{-1/3} \tilde{v} - z_c)}, \\
&A_2^M =  \frac{\pi  M^{-1/3}\sigma_\alpha^{-1}  }{\sin(\pi M^{-1/3}\sigma_\alpha^{-1}(\tilde{v}-\tilde{w})) (\tilde{w}- \tilde{v}')}, \mbox{ and } A_3^M =  e^{M[G_\alpha(\tilde{v}\sigma_\alpha^{-1}M^{-1/3} + z_c) - G_\alpha(\tilde{w}\sigma_\alpha^{-1}M^{-1/3} +z_c)]} e^{ x  (\tilde{v}-\tilde{w})}.
\end{split}
\end{equation*}
In view of (\ref{TWRed2v1}), we see that to show (\ref{TWRed1}) it suffices to prove that
\begin{equation}\label{TWRed2}
\lim_{M \rightarrow \infty} \det (I + K^M_x)_{L^2(C_{\mu, 3\pi/4})}= \det \left( I + K_x^\infty \right)_{L^2(C_{\mu ,3\pi/4})}.
\end{equation}

{\raggedleft {\bf Step 3.}} In this step we prove (\ref{TWRed2}). Notice that for fixed $\tilde{v},\tilde{v}' \in C_{\mu, 3\pi/4}$ and $\tilde{w} \in \tilde{D}$ we have
\begin{equation}\label{blue0}
\lim_{M \rightarrow \infty} g^{M,x}_{\tilde{v},\tilde{v}'}(\tilde{w})  =  g^{\infty,x}_{\tilde{v},\tilde{v}'}(\tilde{w}) ,
\end{equation}
where $g^{\infty,x}_{\tilde{v},\tilde{v}'}$ is as in the statement of the theorem. Indeed, we have for fixed $\tilde{v},\tilde{v}' \in C_{\mu, 3\pi/4}$ and $\tilde{w} \in \tilde{D}$ that
\begin{equation}\label{blue1v1}
\lim_{M \rightarrow \infty }A_2^M = \frac{1}{(\tilde{v} - \tilde{w})(\tilde{w}- \tilde{v}')},
\end{equation}
and also using Lemma \ref{analGS3} we see that
\begin{equation}\label{blue2v1}
\lim_{M \rightarrow \infty}A_3^M= \exp( -\tilde{v}^3/3 + \tilde{w}^3/3 + x  (\tilde{v}-\tilde{w})) .
\end{equation}
Finally, using the functional equation $\Gamma(z+1) = z \Gamma(z)$ and the scaling of $a_n$ for $n =1, \dots, \rr$ and $\alpha_m$ for $m = 1, \dots, \rl$ from Definition \ref{DefScaleS3} we get 
\begin{equation}\label{A1Eq}
\begin{split}
 &A^M_1 = A^M_{1,1} \cdot A^M_{1,2}\cdot A^M_{1,3}, \mbox{ where }A^M_{1,1} =\prod_{n = 1}^{\rr} \frac{( \tilde{w} - x_n+ o(1))}{(\tilde{v}  - x_n + o(1) )} \prod_{m = 1}^{\rl} \frac{( y_m -  \tilde{v}+ o(1)  )}{(y_m  -\tilde{w} + o(1))},\\
&A_{1,2}^M =  \frac{\prod_{m = 1}^{\rl}  \Gamma(\sigma_{\alpha}^{-1} M^{-1/3} y_m - \sigma_{\alpha}^{-1} M^{-1/3} \tilde{w} + 1 + o(1)  )}{\prod_{n = 1}^{\rr}\Gamma(\sigma_{\alpha}^{-1} M^{-1/3} \tilde{w} - \sigma_\alpha^{-1} M^{-1/3} x_n + 1+ o(1) )}, \\
&A_{1,3}^M =  \frac{\prod_{n = 1}^{\rr} \Gamma(\sigma_\alpha^{-1} M^{-1/3} \tilde{v} - \sigma_\alpha^{-1} M^{-1/3} x_n + 1 + o(1)  )}{\prod_{m = 1}^{\rl}  \Gamma(\sigma_{\alpha}^{-1} M^{-1/3} y_m  - \sigma_\alpha^{-1} M^{-1/3} \tilde{v} + 1+ o(1) )}.
\end{split}
\end{equation}
From the last equation we conclude
\begin{equation}\label{blue3v1}
\begin{split}
\lim_{M \rightarrow \infty} A^M_1 =  \prod_{n = 1}^{\rr} \frac{( \tilde{w} - x_n )}{(\tilde{v}  - x_n  )} \prod_{m = 1}^{\rl} \frac{( y_m -  \tilde{v} )}{(y_m  -\tilde{w} )}.
\end{split}
\end{equation}
Equations (\ref{blue1v1}), (\ref{blue2v1})  and (\ref{blue3v1}) imply (\ref{blue0}).\\
 
We claim that we can find positive constants $\Lambda_i, \lambda_i > 0$ for $i = 1, 2,3$ such that for all large enough $M$ we have
\begin{equation}\label{S3UpperBounds}
\begin{split}
&|A^M_1| \leq \Lambda_1 \cdot \exp \left( \lambda_1 ( |\tilde{w} - \mu | \log (1 + |\tilde{w} - \mu|) + |\tilde{v} - \mu | \log (1 + |\tilde{v} - \mu|)) \right), \hspace{2mm} |A^M_2| \leq \Lambda_2 \\
& |A^M_3| \leq \Lambda_3 \cdot \exp \left( - (\lambda_1 + \lambda_3) ( |\tilde{w} - \mu | \log (1 + |\tilde{w} - \mu|) +  |\tilde{v} - \mu | \log (1 + |\tilde{v} - \mu|) ) \right),
\end{split}
\end{equation}
provided that $\tilde{v}, \tilde{v}' \in C_{\mu, 3\pi/4}$, $|\tilde{v} - \mu| \leq \epsilon \sigma_\alpha M^{1/3}$, $|\tilde{v}' - \mu| \leq \epsilon \sigma_\alpha M^{1/3}$ and $\tilde{w} \in \tilde{D}$. We will prove (\ref{S3UpperBounds}) in the next step. Here we assume its validity and conclude the proof of (\ref{TWRed2}).

We may now apply Lemma \ref{FDKCT} to the functions $g^{M,x}_{\tilde{v},\tilde{v}'}(\tilde{w})$ with $F_1(\tilde{v}) = \sqrt{\Lambda_1 \Lambda_2 \Lambda_3} \exp(- \lambda_3 |\tilde{v} - \mu| \log( 1 + |\tilde{v}- \mu|) + |x| |\tilde{w}|) =  F_2(\tilde{v})$ and $\Gamma_1 = C_{\mu, 3\pi/4}$, $\Gamma_2 = \tilde{D}$. Notice that the functions $F_i$ are integrable on $\Gamma_i$. As a consequence we see that $ K^M_x$ and $ K^\infty_x$ satisfy the conditions of Lemma \ref{FDCT}, from which we conclude (\ref{TWRed2}) and in particular that the right side is well-defined.\\

{\bf \raggedleft Step 4.} In this step we establish (\ref{S3UpperBounds}). We first deal with the bound on $A_1^M$. Since the denominators in $A_{1,1}^M$ are uniformly bounded away from $0$ we conclude that 
\begin{equation}\label{Pink1}
|A_{1,1}^M| = O \left( \exp ( \rr \log(1 + |\tilde{w}|)  + \rl \log (1 + |\tilde{v}|) ) \right).
\end{equation}
In addition, since by assumption $|\tilde{v} - \mu| \leq \epsilon \sigma_\alpha M^{1/3}$ and $\epsilon \leq 1/2$ we see that both the numeror and denominator of $A_{1,3}^M$ are uniformly bounded away from $0$ and $\infty$ so that 
\begin{equation}\label{Pink2}
|A_{1,3}^M| = O \left( 1\right).
\end{equation}
It follows from \cite[(2), pp. 32]{Luke} that if $z \in \mathbb{C}$ and $|arg(z)| \leq \pi - \epsilon$ for some $\epsilon \in (0, \pi)$ then 
\begin{equation}\label{S3AssGamma}
\Gamma(z) = e^{-z} z^{z - 1/2} (2\pi )^{1/2} \cdot \left( 1 + O(z^{-1}) \right),
\end{equation}
where the constant in the big $O$ notation depends on $\epsilon$ and we take the principal branch of the logarithm. Also by \cite[Theorem 1.6, Chapter 6]{Stein} there are positive constants $c_1, c_2$ such that 
\begin{equation}\label{S3growthGamma}
\left| \frac{1}{\Gamma(z)} \right|\leq c_1 e^{c_2 |z| \log |z|}.
\end{equation}
Equations (\ref{S3AssGamma}) and (\ref{S3growthGamma}) imply that we can find constants $c_3, c_4 > 0$ such that 
\begin{equation}\label{Pink3}
|A_{1,2}^M|  \leq c_3 \cdot \exp \left( c_4 \sigma_{\alpha}^{-1} M^{-1/3}  |\tilde{w}| \log (1 +\sigma_{\alpha}^{-1} M^{-1/3} |\tilde{w}|)   \right).
\end{equation}
Equations (\ref{Pink1}), (\ref{Pink2}) and (\ref{Pink3}) together imply the first inequality in (\ref{S3UpperBounds}).\\

We next deal with the bound on $A_2^M$. We observe that for $\tilde{v} \in C_{\mu, 3\pi/4}$, $|\tilde{v} - \mu | \leq \sigma_{\alpha} M^{1/3}$ and, $\tilde{w} \in \tilde{D}$ we have
\begin{equation}\label{Pink4}
\left| \frac{\pi M^{-1/3}\sigma_\alpha^{-1} }{\sin(\pi M^{-1/3} \sigma_\alpha^{-1}(\tilde{v} - \tilde{w})) } \right| = O(1).
\end{equation}
Indeed, if $|\pi M^{-1/3} \sigma_\alpha^{-1}(\tilde{v} - \tilde{w})) | \geq 1/2$ the above follows from Lemma \ref{SineBound} and if $ |\pi M^{-1/3} \sigma_\alpha^{-1}(\tilde{v} - \tilde{w})) | \leq  1/2$ we have 
$$\left| \frac{\pi M^{-1/3}\sigma_\alpha^{-1} }{\sin(\pi M^{-1/3} \sigma_\alpha^{-1}(\tilde{v} - \tilde{w})) } \right|  = \left| \frac{\pi M^{-1/3}\sigma_\alpha^{-1} (\tilde{v} - \tilde{w})  }{\sin(\pi M^{-1/3} \sigma_\alpha^{-1}(\tilde{v} - \tilde{w})) } \right|   \cdot \frac{1}{|\tilde{v} - \tilde{w}|},$$
which is also bounded since $\frac{z}{|\sin(z)|}$ is bounded in the disc of radius $1/2$ centered at the origin and $|\tilde{v} - \tilde{w}| \geq \rho$ (recall that $\rho$ was given in the definition of $\tilde{D}$ in Definition \ref{DefScaleS3} and is the smallest distance between $\tilde{D}$ and $C_{\mu, 3\pi/4}$). Since we also have $|\tilde{v}' - \tilde{w}| \geq \rho$ we see that (\ref{Pink4}) implies the second inequality in (\ref{S3UpperBounds}). \\

Finally, we establish the bound on $A_3^M$. The third inequality in (\ref{S3UpperBounds}) would follow if we can show that for $\tilde{v}\in C_{\mu, 3\pi/4}$, $|\tilde{v} - \mu| \leq \epsilon \sigma_\alpha M^{1/3}$ and $\tilde{w} \in \tilde{D}$ we have
\begin{equation}\label{cutoff1}
\begin{split}
&\left|\exp(-MG_\alpha(\tilde{w}\sigma_\alpha^{-1}M^{-1/3} + z_c) \right|  = O\left(\exp(-(\lambda_1 + 1) |\tilde{w} - \mu| \log( 1 + |\tilde{w} - \mu|) \right) \mbox{ and }\\
&\left|\exp(MG_\alpha(\tilde{v}\sigma_\alpha^{-1} M^{-1/3} + z_c) \right|  = O\left( \exp(-(\lambda_1 + 1) |\tilde{v} - \mu | \log( 1 + |\tilde{v} - \mu|) \right).
\end{split}
\end{equation}
We first consider the case when $|\tilde{w} - \mu| \leq M^{1/3 - 1/10}$. In this case we have from Lemma \ref{analGS3} that there exists $c_6 > 0$ such that
$$|M G_\alpha(\tilde{w}\sigma_\alpha^{-1}M^{-1/3} + z_c)  +   \tilde{w}^3/3| \leq c_6 M^{-1/3}|\tilde{w}|^4 \leq c_6 M^{-1/10}|\tilde{w}|^3 .$$
In addition, since $\tilde{w} \in \tilde{D}$ we know that there are constants $c_7, c_8 > 0$ such that 
$$|\exp ( \tilde{w}^3/3)| = \exp ( \Re [w^3/3]) \leq c_7 \exp( - c_8 |\tilde{w}|^3),$$
which in turn implies that 
$$\left|\exp(-MG_\alpha(\tilde{w}\sigma_\alpha^{-1}M^{-1/3} + z_c) \right|  = \exp(-M   \Re G_\alpha(\tilde{w}\sigma_\alpha^{-1}M^{-1/3} + z_c)) \leq c_7 \exp (  (c_6 M^{-1/10} - c_8) |\tilde{w}|^3 ).$$
The latter equation implies the first line of (\ref{cutoff1}) when $|\tilde{w} - \mu| \leq M^{1/3 - 1/10}$ and an analogous argument shows the second line when $|\tilde{v} - \mu| \leq M^{1/3 - 1/10}$. 

We next consider the case when $M^{1/3 - 1/10}\leq |\tilde{w} - \mu| \leq \epsilon \sigma_\alpha M^{1/3}$. Using that $\epsilon \leq r/2$ as in Lemma \ref{analGS3} by assumption we see that the lemma is applicable and so we get 
$$M \Re G_\alpha((\tilde{w} - \mu)\sigma_\alpha^{-1}M^{-1/3} + z_c) \geq (\sqrt{2}/2)^3 |\tilde{w} - \mu|^3/6.$$
Furthermore, by Lemma \ref{horGS3} we see that there is a constant $c_9 > 0$ such that 
$$M | \Re G_\alpha((\tilde{w} - \mu)\sigma_\alpha^{-1}M^{-1/3} + z_c)  - \Re G_\alpha(\tilde{w} \sigma_\alpha^{-1}M^{-1/3} + z_c) | \leq c_9 M^{2/3}.$$
Combining the last two statements we see that 
$$\left|\exp(-MG_\alpha(\tilde{w}\sigma_\alpha^{-1}M^{-1/3} + z_c) \right| \leq \exp( c_9 M^{2/3} - (\sqrt{2}/2)^3 |\tilde{w} - \mu|^3/6) \leq \exp(  - (\sqrt{2}/2)^3 |\tilde{w} - \mu|^3/12),$$
where the last inequality holds for all large enough $M$ and we used that $ |w - \mu|^3 \geq M^{7/10}$ by assumption. The last inequality clearly implies the first line in (\ref{cutoff1}) when $M^{1/3 - 1/10}\leq |\tilde{w} - \mu| \leq \epsilon \sigma_\alpha M^{1/3}$ and an analogous argument shows the second line when $M^{1/3 - 1/10}\leq |\tilde{v} - \mu| \leq \epsilon \sigma_\alpha M^{1/3}$.

Finally, we consider the case $|\tilde{w} - \mu|  \geq \epsilon \sigma_\alpha M^{1/3}$. In view of Lemma \ref{ReleL1S3} we know that there is a $c_{10} > 0$ such that 
$$M \Re G_\alpha((\tilde{w} - \mu)\sigma_\alpha^{-1}M^{-1/3} + z_c) \geq c_{10} M^{2/3} | \tilde{w} - \mu | \log (1 + \sigma_\alpha^{-1}M^{-1/3} | \tilde{w} - \mu |).$$
Furthermore, by Lemma \ref{horGS3} we see that there is a constant $c_{11} > 0$ such that 
$$M | \Re G_\alpha((\tilde{w} - \mu)\sigma_\alpha^{-1}M^{-1/3} + z_c)  - \Re G_\alpha(\tilde{w} \sigma_\alpha^{-1}M^{-1/3} + z_c) | \leq c_{11} M^{2/3} ( 1 + \sigma_\alpha^{-1}M^{-1/3}  | \tilde{w} - \mu|).$$
Combining the last two statements we see that  
$$\left|\exp(-MG_\alpha(\tilde{w}\sigma_\alpha^{-1}M^{-1/3} + z_c) \right| \leq \exp( - c_{10} | \tilde{w} - \mu |   M^{2/3}  \log (1 + \sigma_\alpha^{-1}M^{-1/3}  | \tilde{w} - \mu |) ) \cdot  $$
$$\exp(c_{11} M^{2/3} + c_{11} \sigma_\alpha^{-1}M^{1/3}  | \tilde{w} - \mu| )  .$$
which implies the first line of in (\ref{cutoff1}) when $|\tilde{w} - \mu|  \geq \epsilon \sigma_\alpha M^{1/3}$. This suffices for the proof.

\end{proof}

%
\section{Proof of Theorems \ref{LGPCT}, \ref{S1LDE1} and \ref{BBP}}\label{Section4}
In this section we prove the three main results of the paper from Section \ref{Section1}.

%
\subsection{Proof of Theorem \ref{LGPCT}}\label{Section4.1} In this section we prove Theorem \ref{LGPCT} by utilizing Theorem \ref{mainThmGUE}. We work under the same assumptions as in the statement of Theorem \ref{LGPCT}. As discussed in Remark \ref{S1RemTW} it suffices to prove the theorem when $N/M \in [\delta,1]$, which we assume in the sequel.

We will use the following elementary probability lemma.
\begin{lemma}\label{prob} \cite[Lemma 4.1.39]{BorCor}
Suppose that $f_n$ is a sequence of functions $f_n: \mathbb{R} \rightarrow [0,1]$, such that for each $n$, $f_n(y)$ is strictly decreasing in $y$ with a limit of $1$ at $y = -\infty$ and $0$ at $y = \infty$. Assume that for each $\delta > 0$ one has on $\mathbb{R}\backslash[-\delta,\delta]$, $f_n \rightarrow {\bf 1}_{\{y \leq 0\}}$ uniformly. Let $X_n$ be a sequence of random variables such that for each $x \in \mathbb{R}$ 
$$\mathbb{E}[f_n(X_n - x)] \rightarrow p(x),$$
and assume that $p(x)$ is a continuous probability distribution function. Then $X_n$ converges in distribution to a random variable $X$, such that $\mathbb{P}(X \leq x) = p(x)$.
\end{lemma}

Let $y \in \mathbb{R}$, $\alpha = N/M$ and $\sigma_\alpha, u(y,M,N)$ be as in Definition \ref{DefScaleS3}. Define 
$$f_M(z) = e^{-e^{\sigma_{\alpha}M^{1/3} z}},$$
and observe that since by assumption we have $M \geq N \geq 9$ for all $M$ sufficiently large we can apply Theorem \ref{LGPT1} to conclude
$$\mathbb{E} \left[f_M(\mathcal{F}(M,N) - y)\right] = \mathbb{E} \left[ e^{-u \ZP} \right] = \det \left( I + K_u \right)_{L^2(C_{z_c,3\pi/4})}.$$

On the other hand, by Theorem \ref{mainThmGUE} we know that as $M \rightarrow \infty$ the right-hand side converges to $\det (I + K^{\infty}_{y})_{L^2(C_{0, 3\pi/4})}$,
where $ K_y^\infty$ is defined in terms of its integral kernel
\begin{equation*}
\begin{split}
 &K_y^\infty(v,v') = \frac{1}{2\pi \i} \int_{\tilde{D}}   \frac{\exp(-v^3/3 + w^3/3 - y w + y v)}{(v-w)(w - v')} dw.
\end{split}
\end{equation*}
The proof of  \cite[Lemma C.1]{BCF} can now be repeated verbatim to show that 
$$\det (I + K^{\infty}_{y})_{L^2(C_{0, 3\pi/4})} =  F_{GUE}(y).$$

Combining the above statements we see that 
$$\lim_{M \rightarrow \infty} \mathbb{E} \left[f_M(\mathcal{F}(M,N) - y)\right] =F_{GUE}(y).$$
Since $f_M$ and $F_{GUE}$ satisfy the conditions of Lemma \ref{prob}, we see that the last equation and the lemma imply the statement of the theorem.

%
\subsection{Proof of Theorem \ref{S1LDE1}}\label{Section4.2}
 In this section we prove Theorem \ref{S1LDE1} and we assume the same notation as in the statement of the theorem. For clarity we split the proof into several steps. In the first step we reduce the statement of the theorem to establishing a certain bound on $\left| \det \left( I + K_u \right)_{L^2(C_{z_c,3\pi/4})} - 1 \right|$ for all large enough $M$. In Step 2 we truncate the contour $C_{z_c, 3\pi/4}$ and show that it suffices to establish the same bound for the truncated contour using Proposition \ref{PKBS3} and Lemma \ref{SumBound}. In the third step we truncate the contours in the definition of the kernel $K_u$ and show that it suffices to obtain certain estimates on each summand in the resulting Fredholm determinant series. In Steps 4 and 5 we show the required estimates in Step 3 are satisfied.\\

{\bf \raggedleft Step 1.} Assume the same notation as in Theorem \ref{LGPT1}, where $u$ is as in Definition \ref{DefScaleS3}, $a_ i = 0$ for $i = 1, \dots, N$, $\alpha_j = \theta$ for $j = 1, \dots, M$, $a = z_c$ and $b,d$ are arbitrarily chosen so as to satisfy the conditions of that theorem. Suppose further that $N,M$ are scaled as in  Definition \ref{DefScaleS3}. We claim that there exist $C_0, c_0, \epsilon_0, x_0 > 0$ and $M_0 \in \mathbb{N}$ all depending on $\theta$ and $\delta$ such that if $x \in [x_0, \epsilon_0 M^{2/3}]$ and $M \geq M_0$ we have
\begin{equation}\label{mainLimitLDEeq}
\left| \det \left( I + K_u \right)_{L^2(C_{z_c,3\pi/4})} - 1 \right| \leq  C_0 e^{-c_0 x^{3/2}}.
\end{equation}
We prove (\ref{mainLimitLDEeq}) in the steps below. Here we assume its validity and conclude the proof of the theorem. 

We want to show that there exist $C_1,C_2, c_1, c_2 > 0$ such that for any $x \geq 0$
\begin{equation}\label{LDEeq}
\mathbb{P} \left( \mathcal{F}(M,N) \geq x  \right) = \mathbb{P} \left( \frac{\log \ZP + W_{M,N}}{\sigma_\alpha M^{1/3}} \geq x  \right) \leq C_1 e^{-c_1 M} + C_2 e^{-c_2 x^{3/2}},
\end{equation}
where $\alpha = N/M$, $\sigma_\alpha$ and $W_{M,N}$ are as in Definition \ref{DefScaleS3}. Notice that if $M$ is bounded then (\ref{LDEeq}) can be made to trivially hold by making $C_1$ sufficiently large for any $x \geq 0$. Thus we only need to prove (\ref{LDEeq}) for large enough $M$. 

Let $M_0, C_0,c_0,\epsilon_0, x_0$ be such that (\ref{mainLimitLDEeq}) holds and let $u(x - 1,M,N)$ be as in Definition \ref{DefScaleS3}. If we set $f_M(z) = \exp(-e^{\sigma_{\alpha}  M^{1/3}z}),$ then by Theorem \ref{LGPT1} we know that for all large enough $M$ 
$$\mathbb{E} \left[f_M(\mathcal{F}(M,N) - x + 1)\right] = \mathbb{E} \left[ e^{-u \ZP} \right] = \det \left( I + K_u \right)_{L^2(C_{z_c,3\pi/4})}.$$
The latter and (\ref{mainLimitLDEeq}) imply that for all large enough $M$ and $x \in [x_0 + 1, \epsilon_0 M^{2/3}]$
$$\mathbb{E} \left[1 - f_M(\mathcal{F}(M,N) - x+1)\right]  \leq  C_0 e^{-c_0 (x-1)^{3/2}}.$$
Notice that $1 - f_M(\mathcal{F}(M,N)- x+1) \geq 0$ and on the event $\{\mathcal{F}(M,N) \geq x\}$ we have
$$1- f_M(\mathcal{F}(M,N) - x+1) = 1 - e^{-e^{\sigma_{\alpha}  M^{1/3}(\mathcal{F}(M,N) - x +1)}} \geq 1 - e^{-e^{\sigma_{\alpha}M^{1/3}}} \geq 1/2,$$
where the last inequality holds for all $M$ sufficiently large depending on $\delta$ and $\theta$ alone. The latter two inequalities imply that 
$$ 2C_0 e^{-c_0 (x-1)^{3/2}}  \geq \mathbb{P}(X_M \geq x),$$
which proves (\ref{LDEeq}) when $x \in [x_0 + 1, \epsilon_0 M^{2/3}]$ and $M$ is sufficiently large. If $x \geq \epsilon_ 0M^{2/3}$ then we have from our earlier work that
$$\mathbb{P} \left( \mathcal{F}(M,N) \geq x  \right) \leq \mathbb{P} \left( \mathcal{F}(M,N) \geq \epsilon_0M^{2/3}  \right) \leq 2C_0 e^{-c_0 (\epsilon_0 M^{2/3} -1)^{3/2}},$$
which proves (\ref{LDEeq}) in this case as well. This proves (\ref{LDEeq}) for all $x \geq x_0 +1$ and by possibly making $C_2$ larger we get the result for $x \in [0, x_0 + 1]$ as well, concluding the proof of the theorem. \\

{\bf \raggedleft Step 2.} In the remaining steps we prove (\ref{mainLimitLDEeq}). Let $\theta > 0$ and $\delta \in (0,1)$ be given. Let $r$ be as in Lemma \ref{powerGS3} and fix $\epsilon > 0$ such that $\epsilon < \min (r/2, 1/4)$. We let $C^{\epsilon}_{z_c, 3\pi/4}$ denote the portion of $C_{z_c, 3\pi/4}$ inside $B_{\epsilon}(z_c)$ -- the disc of radius $\epsilon$, centered at $z_c$. We claim that we can find $C_0, c_0, \epsilon_0,x_0 > 0$ and $M_0 \in \mathbb{N}$ such that if $M \geq M_0$ and $x \in [x_0, \epsilon_0 M^{2/3}]$ then
\begin{equation}\label{LDERed1}
\left| \det (I + K_{u})_{L^2(C^{\epsilon}_{z_c, 3\pi/4})} - 1  \right| \leq C_0 e^{-c_0 x^{3/2}}.
\end{equation}
We prove (\ref{LDERed1}) in the next steps. Here we assume its validity and conclude the proof of the theorem.\\

Repeating verbatim the argument in Step 1 of the proof of Theorem \ref{mainThmGUE} (when we apply Proposition \ref{PKBS3} we do it with $r_1 = r_2 = A = 0$) we get that for all large $M$
\begin{equation}\label{ST1}
 \left| \det (I + K_u)_{L^2(C_{z_c, 3\pi/4})} - \det (I + K_{u})_{L^2(C^{\epsilon}_{z_c, 3\pi/4})} \right|\leq e^{-c'M/2},
\end{equation} 
where $c' = (c/2) \log(1 + \epsilon)$. Combining (\ref{ST1}) with (\ref{LDERed1}) we conclude (\ref{mainLimitLDEeq}).\\

{\raggedleft {\bf Step 3.}} From the definition of Fredholm determinants, see Section \ref{Section2.1}, we know that
\begin{equation}\label{ST4}
\begin{split}
 & \det (I + K_{u})_{L^2(C^{\epsilon}_{z_c, 3\pi/4})} -1 = \sum_{n = 1}^\infty \frac{H_n}{n!} , \mbox{ where }
\end{split}
\end{equation} 
$$H_n =   \frac{1}{(2\pi \i )^n} \int_{C^{\epsilon}_{z_c, 3\pi/4}} \cdots \int_{C^{\epsilon}_{z_c, 3\pi/4}} \det \left[ K_u(v_i, v_j) \right]_{i,j = 1}^n  \prod_{i = 1}^n dv_i.$$
We thus can reduce the proof of (\ref{LDERed1}) to showing that there exist $C_0, c_0, \epsilon_0,x_0 > 0$ and $M_0 \in \mathbb{N}$ such that if $M \geq M_0$ and $x \in [x_0, \epsilon_0 M^{2/3}]$
\begin{equation}\label{LDERed2}
\begin{split}
 \sum_{n = 1}^\infty \frac{|H_n|}{n!} \leq  C_0 e^{-c_0 x^{3/2}}.
\end{split}
\end{equation} 

Let $D_M$ be as Definition \ref{DefDM}. We write $D_M = D_M^{\epsilon, 0} \cup D_M^{\epsilon,1}$, where $D_M^{\epsilon, 0}$ is the portion of $D_M$ inside $B_\epsilon(z_c)$  and $D_M^{\epsilon, 1}$ is the portion outside. Define for $\beta \in \{0,1\}$ and $v, v' \in C^{\epsilon}_{z_c, 3\pi/4} $ the kernel
$$K^{\beta}_{u}(v,v') = \frac{1}{2\pi \i}\int_{D_M^{\epsilon, \beta}}\frac{F(v, w)}{w - v' }dw  \mbox{, where } F(v,w) = \frac{\pi}{\sin(\pi (v-w))} \cdot e^{G_{M,N}(v) - G_{M,N}(w)} \cdot e^{M^{1/3} (w- v) \sigma_{\alpha} x}.$$

 We claim that we can find $C_2, c_2, \epsilon_0, x_0 > 0$ and $M_0 \in \mathbb{N}$ such that if $M \geq M_0$ and $x \in [x_0, \epsilon_0 M^{2/3}]$ 
\begin{equation}\label{LDERed3}
\begin{split}
\sum_{n = 1}^\infty \frac{1}{n!}  \left| \frac{1}{(2\pi \i)^n}  \int_{C^{\epsilon}_{z_c, 3\pi/4}} \cdots \int_{C^{\epsilon}_{z_c, 3\pi/4}} \det \left[ K^0_u(v_i, v_j) \right]_{i,j = 1}^n  \prod_{i = 1}^n dv_i \right| \leq  C_2 e^{-c_2 x^{3/2}}.
\end{split}
\end{equation}
We prove (\ref{LDERed3}) in the next steps. Here we assume its validity and conclude the proof of (\ref{LDERed2}).\\

Let us fix $n \in \mathbb{N}$, and note that by the linearity of the determinant function we have
\begin{equation}\label{STV1}
\frac{H_n}{n!}  = \frac{1}{n!} \sum_{ (\beta_1, \dots, \beta_n) \in \{0, 1\}^n} \frac{1}{(2\pi \i)^n} \int_{C^{\epsilon}_{z_c, 3\pi/4}}  \cdots  \int_{C^{\epsilon}_{z_c, 3\pi/4}}  \det \left[ K^{\beta_i}_u(v_i, v_j) \right]_{i,j = 1}^n  \prod_{i = 1}^n dv_i.
\end{equation}
We also note that by Lemma \ref{ReleL1S3}, with $\delta_1 = \epsilon$, we have that there exists a constant $c_\epsilon > 0$ such that for any $v, v' \in C^{\epsilon}_{z_c, 3\pi/4}$
\begin{equation}\label{STV2}
\left|K^{1}_u(v, v')\right| \leq 2M^{1/3}\int_{ \epsilon}^{\infty} \exp \left( -c (M+N)y \log (1 +y) \right) dy \leq e^{-c_\epsilon M}.
\end{equation}
In deriving the above inequality we also used that $\left| e^{M^{1/3} (w- v) \sigma_{\alpha} x} \right|  \leq 1$ as the real part of the exponent is negative and we used Lemma \ref{SineBound}  to bound the sine function. 

In addition, by Lemmas \ref{analGS3}, \ref{ReleL1S3} and  \ref{SineBound} we know that for any $v, v' \in C^{\epsilon}_{z_c, 3\pi/4}$
\begin{equation}\label{STV3}
\left|K^{0}_u(v, v')\right| \leq CM^{1/3}.
\end{equation}
Equations (\ref{STV2}), (\ref{STV3}) and Hadamard's inequality imply that 
\begin{equation}\label{STV4}
\left| \int_{C^{\epsilon}_{z_c, 3\pi/4}}  \cdots  \int_{C^{\epsilon}_{z_c, 3\pi/4}}  \det \left[ K^{\beta_i}_u(v_i, v_j) \right]_{i,j = 1}^n  \prod_{i = 1}^n dv_i\right| \leq C^{n} n^{n/2} M^{n/3} e^{-c_{\epsilon}M\sum_{i = 1}^n \beta_i}.
\end{equation}

Finally, combining (\ref{LDERed3}), (\ref{STV1}) and (\ref{STV4}) we conclude that 
$$\sum_{n = 1}^\infty \frac{|H_n|}{n!}  \leq e^{-c_\epsilon M} \cdot \sum_{n = 1}^\infty \frac{C^n n^{n/2}M^{n/3}}{n!}  + C_2 e^{-c_2 x^{3/2}},$$
which together with Lemma \ref{SumBound} implies (\ref{LDERed2}).\\

{\raggedleft {\bf Step 4.}} The first three steps above reduce the proof of the theorem to establishing the existence of $\epsilon_0, x_0 > 0$ such that if $x \in [x_0, \epsilon_0 M^{2/3}]$ and $M$ is sufficiently large we have that (\ref{LDERed3}) holds for some $ c_2, C_2 > 0$. We will establish this statement in the next step. In this step we specify our choice of $\epsilon_0$ and $x_0$ and introduce some useful notation for later. 

 Let $\epsilon > 0$ be as in Step 2. Then from Lemma \ref{analGS3}  we know that for some $C_4 > 0$ we have
$$|G_\alpha(z) + (z-z_c)^3 \sigma_{\alpha}^3/3| \leq C_4 |z-z_c|^4$$
as long as $|z-z_c| \leq \epsilon$. We fix this choice of $C_4.$ Let $\tilde{D} = \sigma_{\alpha} M^{1/3}(D_M - z_c)$ be as in Definition \ref{DefScaleS3}. We also write $\tilde{D}^0 = \sigma_{\alpha} M^{1/3}(D_M^0 - z_c)$ and $\tilde{D}^{\pm} = \sigma_\alpha M^{1/3} (D_M^{\pm} - z_c)$, where $D_M^{\pm}$ and $D_M^0$ are as in Definition \ref{DefDM}. We define the constants 
\begin{equation}\label{DefConst}
D_1 := \int_{C_{0, 3\pi/4}} e^{-|z|^3}|dz| < \infty \mbox{ and } D_2 := \int_{\tilde{D}} e^{-|z|^3}|dz| < \infty,
\end{equation}
where $|dz|$ denotes integration with respect to arc-length. 

We now let $\epsilon_0 > 0$ be sufficiently small so that $\epsilon_0 \leq \epsilon/2$ and also 
$$ 4C_4  \epsilon_0^{2} + \sigma_{\alpha}^3 \epsilon_0(2\sqrt{2})/3 - \sigma_{\alpha} \leq - 3\sigma_{\alpha}/4,$$ 
where we observe that this can be done uniformly in $\delta$ when $\alpha \in [\delta,1]$, see Definition \ref{DefScaleS3}. This fixes our choice of $\epsilon_0$. We next let $x_0$ be sufficiently large so that the following inequalities hold:
$$\sqrt{\epsilon_0} x_0^{3/2} \sigma_{\alpha}/4 \geq 2 \sqrt{2} \mbox{, } \hspace{5mm} \sigma_{\alpha}^3 (\sqrt{2}/2)^3 \epsilon_0^{3/2}x_0^{3/2} /6 \geq 1 $$
$$\mbox{ and }\sum_{n = 1}^\infty (\tilde{C}D_1D_2)^n \exp \left( - (n-1) (\sqrt{\epsilon_0}/2) x^{3/2}_0 \sigma_\alpha \right) \leq 2\tilde{C} D_1 D_2,$$
where $\tilde{C}$ is the universal constant afforded by Lemma \ref{SineBound}. This fixes our choices of $\epsilon_0$ and $x_0$. \\

{\raggedleft {\bf Step 5.}} In this step we prove (\ref{LDERed3}). Let $\epsilon_0$ and $x_0$ be as in Step 4 and let $x \in [x_0, \epsilon_0 M^{2/3}]$. We also let $\rho_0 = \sqrt{\epsilon_0}M^{-1/3} x^{1/2}$ and note that $\rho_0 \in (0, \epsilon_0]$ by definition. We let $C_{\rho_0}$ denote the portion of the contour $\rho_0 \cdot \tilde{D} +z_c$ that is contained in the disc of radius $\epsilon$ centered at $z_c$. Notice that since $\epsilon_0 < \epsilon/2$ and $\epsilon < 1/4$ by definition we can deform the contour $D_M^{\epsilon, 0}$ in the definition of $K^0_u$ to $C_{\rho_0}$ without affecting the value of the integral by Cauchy's theorem, since we do not cross any poles in the deformation.

Fix $n \in \mathbb{N}$ and notice that by expanding the determinant $\det \left[ K^0_u(v_i, v_j) \right]_{i,j = 1}^n$ and performing the change of variables $v = \rho_0 \tilde{v} + z_c$ and $w = \rho_0 \tilde{w} + z_c$ we get
\begin{equation}\label{LDERed4}
\begin{split}
&A_n:=   \frac{1}{(2\pi \i)^n} \int_{C^{\epsilon}_{z_c, 3\pi/4}} \cdots \int_{C^{\epsilon}_{z_c, 3\pi/4}} \det \left[ K^0_u(v_i, v_j) \right]_{i,j = 1}^n  \prod_{i = 1}^n dv_i = \frac{1}{(2\pi \i)^n} \sum_{\sigma \in S_n}  \\
&\int_{C_{0, 3\pi/4}} \cdots \int_{C_{0, 3\pi/4}} \int_{\tilde{D}} \cdots \int_{\tilde{D}} \prod_{i = 1}^n{\bf 1}_{|\tilde{v}_i| \leq \epsilon \rho_0^{-1}}{\bf 1}_{|\tilde{w}_i| \leq \epsilon \rho_0^{-1}} \prod_{i =1}^ng^{M,x}_{\tilde{v}_i, \tilde{v}_{\sigma(i)}}(\tilde{w}_i)\prod_{i = 1}^n d\tilde{w}_i \prod_{i = 1}^n d\tilde{v}_i  ,
\end{split}
\end{equation}
where
$$g^{M,x}_{\tilde{v}, \tilde{v}'}(\tilde{w}) = \frac{\pi  \rho_0  e^{M[G_\alpha(\rho_0 \tilde{v} + z_c) - G_\alpha(\rho_0 \tilde{w} +z_c)]}e^{  \sqrt{\epsilon_0} x^{3/2} \sigma_{\alpha}  (\tilde{v}-\tilde{w})} }{\sin(\pi \rho_0 (\tilde{v}-\tilde{w})) (\tilde{w}- \tilde{v}')}.$$ 

If $\tilde{C} > 0$ denotes the constant from Lemma \ref{SineBound} we see that if $\tilde{v}, \tilde{v}' \in C_{0, 3\pi/4}$ and $\tilde{w} \in \tilde{D}$ satisfy $|\tilde{v}| \leq \epsilon \rho_0^{-1} \leq 4^{-1} \rho_0^{-1}$ and $|\tilde{w}| \leq \epsilon \rho_0^{-1} \leq 4^{-1} \rho_0^{-1}$ then
\begin{equation}\label{TTV1}
\left|  \frac{\pi \rho_0 }{\sin(\pi \rho_0 (\tilde{v}-\tilde{w})) (\tilde{w}- \tilde{v}')} \right| \leq \tilde{C}.
\end{equation}
We next have from Lemma \ref{analGS3} that if $\tilde{w} \in \tilde{D}^0$ (this contour was defined in Step 4) and $\tilde{v} \in C_{0,3\pi/4}$ with $|\tilde{v}| \leq \epsilon \rho_0^{-1}$ then 
\begin{equation}\label{TTV2}\begin{split}
&\left| g^{M,x}_{\tilde{v}, \tilde{v}'}(\tilde{w}) \right| \leq \tilde{C} \exp( - M\sigma_{\alpha}^3 (\sqrt{2}/2)^3 \rho_0^3 |\tilde{v}|^3 /6+ M C_4 \rho_0^4 |\tilde{w}|^4 +M\sigma_{\alpha}^3  \rho_0^3 |\tilde{w}|^3/3 - \sqrt{\epsilon_0} x^{3/2} \sigma_{\alpha} ) \leq \\
&\tilde{C} \exp( - \sigma_{\alpha}^3 (\sqrt{2}/2)^3 \epsilon_0^{3/2}x^{3/2} |\tilde{v}|^3/6 + 4 C_4 \rho_0 \epsilon_0^{3/2}x^{3/2}  +\sigma_{\alpha}^3 \epsilon_0^{3/2}(2\sqrt{2}) x^{3/2}/3 -   \sqrt{\epsilon_0} x^{3/2} \sigma_{\alpha} ) \leq \\
& \tilde{C} \exp( -   |\tilde{v}|^3  - |\tilde{w}|^3 -  ( \sqrt{\epsilon_0}/2) x^{3/2} \sigma_{\alpha} ) ,
\end{split}
\end{equation}
where we used that $|\tilde{w}| \leq \sqrt{2}$ and in the last inequality we also used that $x \in [x_0, \epsilon_0 M^{2/3}]$ and our definition of $x_0$ and $\epsilon_0$ from Step 4. 

On the other hand, from Lemma \ref{analGS3} we have that if $\tilde{w}\in \tilde{D}^{+} \cup \tilde{D}^-$ and $\tilde{v} \in C_{0,3\pi/4}$ are chosen to satisfy $|\tilde{v}| \leq \epsilon \rho_0^{-1}$ and $|\tilde{w}| \leq \epsilon \rho_0^{-1}$ then
\begin{equation}\label{TTV3}\begin{split}
&\left| g^{M,x}_{\tilde{v}, \tilde{v}'}(\tilde{w}) \right| \leq\tilde{C} \exp( - M\sigma_{\alpha}^3 (\sqrt{2}/2)^3 \rho_0^3 |\tilde{w}|^3 /6- M\sigma_{\alpha}^3 (\sqrt{2}/2)^3 \rho_0^3 |\tilde{v}|^3 /6) \cdot \\
&\exp( - \sqrt{\epsilon_0} x^{3/2} \sigma_{\alpha}(\sqrt{2}/2) (|\tilde{v}| + |\tilde{w}|) ) \leq \tilde{C} \exp( -  \epsilon_0^{3/2}\sigma_{\alpha}^3 (\sqrt{2}/2)^3 x^{3/2} |\tilde{v}|^3/6 ) \cdot \\
&  \exp( - \epsilon_0^{3/2} \sigma_{\alpha}^3 (\sqrt{2}/2)^3 x^{3/2} |\tilde{w}|^3/6 - \sqrt{\epsilon_0} x^{3/2} \sigma_{\alpha}(\sqrt{2}/2) (|\tilde{v}| + |\tilde{w}|)  ) \leq \\
& \exp( -  |\tilde{v}|^3 -    |\tilde{w}|^3 - \sqrt{\epsilon_0} x^{3/2} \sigma_{\alpha} ).
\end{split}
\end{equation}
where we used that $|\tilde{w}| \geq \sqrt{2}$ and in the last inequality used that $x \geq x_0$ and our definition of $x_0$. 

Combining (\ref{TTV2}) and (\ref{TTV3}) with (\ref{LDERed4}) we see that
\begin{equation*}
\begin{split}
&|A_n| \leq n! \exp( -n (\sqrt{\epsilon_0}/2) x^{3/2} \sigma_{\alpha} ) \tilde{C}^n  \int_{C_{0, 3\pi/4}} \hspace{-5mm} \cdots \int_{C_{0, 3\pi/4}} \int_{\tilde{D}} \cdots \int_{\tilde{D}}  \prod_{i = 1}^n \exp( -  |\tilde{v}_i|^3 -    |\tilde{w}_i|^3)  |d\tilde{w}_i| |d\tilde{v}_i| = \\
&n! \exp( -n \sqrt{\epsilon_0} x^{3/2} \sigma_{\alpha} ) (\tilde{C} D_1 D_2)^n, 
\end{split}
\end{equation*}
where $D_i$ are as (\ref{DefConst}). We conclude that 
$$\sum_{n = 1}^\infty \frac{|A_n|}{n!} \leq \sum_{n = 1}^\infty  \exp( -n (\sqrt{\epsilon_0}/2) x^{3/2} \sigma_{\alpha} ) (\tilde{C} D_1 D_2)^n \leq 2 \tilde{C} D_1 D_2 \cdot  \exp(- (\sqrt{\epsilon_0}/2) x^{3/2} \sigma_{\alpha}), $$
where in the last inequality  we used our definition of $x_0$ and $\epsilon_0$ and the fact that $x \in [x_0, \epsilon_0 M^{2/3}]$. This proves (\ref{LDERed3}) with $C_2 = 2 \tilde{C}D_1D_2$ and $c_2 >0$ sufficiently small so that $c_2 < (\sqrt{\epsilon_0}/2) \sigma_{\alpha}$, which can be done uniformly in $\delta$ if $\alpha \in [\delta, 1]$, see Definition \ref{DefScaleS3}.

%
\subsection{Proof of Theorem \ref{BBP} }\label{Section4.3} We first turn to the proof of Lemma \ref{S1Lemma}. 

\begin{proof}[Proof of Lemma \ref{S1Lemma}] We continue with the same notation as in the statement of the lemma. For clarity we split the proof of the lemma into several steps. In the first step, we show that $F_{{ \rm BBP}; \vec{x}, \vec{y}}$ is well-defined and does not depend on the choice of $a,b \in \mathbb{R}$ as long as $\min(\vec{y}) > b > a > \max(\vec{x})$. In the second step we prove that $F_{{ \rm BBP}; \vec{x}, \vec{y}}$ is entire. In the third step we prove that the restriction of $F_{{ \rm BBP}; \vec{x}, \vec{y}}$ to $\mathbb{R}$ is real-valued, increasing and that the limit near $\infty$ is $1$. In the fourth step we prove that the limit of $F_{{ \rm BBP}; \vec{x}, \vec{y}}$ near $-\infty$ is $0$.\\

{\bf \raggedleft Step 1.} In this step we prove that $F_{{ \rm BBP}; \vec{x}, \vec{y}}$ is well-defined and does not depend on the choice of $a,b \in \mathbb{R}$ as long as $\min(\vec{y}) > b > a > \max(\vec{x})$. Let 
$$g_{v,v'}(w) =  \prod_{n = 1}^{\rr} \frac{( {w} - x_n )}{({v}  - x_n  )} \prod_{m = 1}^{\rl} \frac{( y_m -  {v} )}{(y_m  -{w} )}\cdot  \frac{\exp(-v^3/3 + w^3/3 - r w + r v)}{(v-w)(w - v')}.$$
Observe that 
$$\sup_{v' \in C_{a, 3\pi/4}} | g_{v,v'}(w) | \leq \frac{1}{(b-a)^2} \cdot \left| \prod_{n = 1}^{\rr} \frac{( {w} - x_n )}{({v}  - x_n  )} \prod_{m = 1}^{\rl} \frac{( y_m -  {v} )}{(y_m  -{w} )} \right| e^{|rw| + |rv|}\exp(-\Re[v^3]/3 + \Re[w^3]/3) .$$
It follows from Lemma \ref{FDKCT} applied to $\Gamma_1 = C_{a, 3\pi/4}$, $\Gamma_2 = C_{b, \pi/4}$, 
$$F_1(v) = \frac{1}{b-a} \cdot  \frac{\prod_{n = 1}^{\rl} | y_m -  {v} |}{\prod_{m = 1}^{\rr} |{v}  - x_n |}  e^{|rv| - \Re[v^3]/3}, \hspace{2mm}F_2(w) = \frac{1}{b-a} \cdot  \frac{\prod_{m = 1}^{\rr} |{w}  - x_n |}{\prod_{n = 1}^{\rl} | y_m -  {w} |}  e^{|rw| + \Re[w^3]/3},$$
that the integral
$$K_r^{\rm BBP}(v,v')= \frac{1}{2\pi \i}\int_{\Gamma_2}g_{v,v'}(z) dz.$$
is well-defined and moreover, that Lemma \ref{FDCT} is applicable from which we conclude that the series $\det \left( I + K_r^{\rm BBP} \right)_{L^2(C_{a ,3\pi/4})} $ in (\ref{S1FDet2}) is absolutely convergent. Here we used that $F_i$ is integrable on $\Gamma_i$ as follows from the cube in the exponential and the definition of the contours. This proves that $F_{{ \rm BBP}; \vec{x}, \vec{y}}(r) $ is well-defined. In addition, by expanding the determinants we have
\begin{equation}\label{S4ExpandDet}
\begin{split}
& F_{{ \rm BBP}; \vec{x}, \vec{y}}(r)  = 1 + \sum_{N = 1}^\infty \frac{1}{N!} \sum_{\sigma \in S_N} \int_{C^N_{a ,3\pi/4}}  \int_{C^N_{b ,\pi/4}} \prod_{i = 1}^N\prod_{n = 1}^{\rr} \frac{( {w}_i - x_n )}{({v}_i  - x_n  )} \prod_{m = 1}^{\rl} \frac{( y_m -  {v}_i )}{(y_m  -{w}_i)}\times\\
&\prod_{i = 1}^N  \frac{\exp(-v_i^3/3 + w_i^3/3 - r w_i + r v_i)}{(v_i-w_i)(w_i - v_{\sigma(i)})}\prod_{i = 1}^N \frac{dw_i}{2\pi \i} \frac{dv_i}{2\pi \i},
\end{split}
\end{equation}
from which we immediately conclude that $F_{{ \rm BBP}; \vec{x}, \vec{y}}(r) $ does not depend on $a,b$ as long as $\min(\vec{y}) > b > a > \max(\vec{x})$. Indeed, if $a', b'$ also satisfy $\min(\vec{y}) > b' > a' > \max(\vec{x})$ then we can deform $C_{a, 3\pi/4}$ to $C_{a', 3\pi/4}$ and $C_{b, \pi/4}$ to $C_{b', \pi/4}$ in each summand in (\ref{S4ExpandDet}) without crossing any poles and thus without affecting the value of the integral by Cauchy's theorem. The decay necessary to deform the contours near infinity come from the cubes in the exponential.\\

{\bf \raggedleft Step 2.} In this step we show that $ F_{{ \rm BBP}; \vec{x}, \vec{y}} $ is entire. Fix $v, v' \in C_{a, 3\pi/4}$. It follows from \cite[Chapter 2, Theorem 5.4]{Stein} that for each $N \in \mathbb{N}$ the function 
$$K_r^N(v,v') =  \int_{C_{b,\pi/4}} \prod_{n = 1}^{\rr} \frac{( {w} - x_n )}{({v}  - x_n  )} \prod_{m = 1}^{\rl} \frac{( y_m -  {v} )}{(y_m  -{w} )}\cdot  \frac{\exp(-v^3/3 + w^3/3 - r w + r v)}{(v-w)(w - v')} \cdot{\bf 1} \{ |w| \leq N\} dw $$
is entire in $r$. Consequently, as the uniform over compacts limit of entire functions we have that
$$\int_{C_{b,\pi/4}} \prod_{n = 1}^{\rr} \frac{( {w} - x_n )}{({v}  - x_n  )} \prod_{m = 1}^{\rl} \frac{( y_m -  {v} )}{(y_m  -{w} )}\cdot  \frac{\exp(-v^3/3 + w^3/3 - r w + r v)}{(v-w)(w - v')} dw$$
is also entire in $r$, cf. \cite[Chapter 2, Theorem 5.2]{Stein}. Let us fix $r_0 \in \mathbb{C}$ and observe that 
$$ \sup_{r: |r - r_0| \leq 1}  \int_{C_{b,\pi/4}}\left| \prod_{n = 1}^{\rr} \frac{( {w} - x_n )}{({v}  - x_n  )} \prod_{m = 1}^{\rl} \frac{( y_m -  {v} )}{(y_m  -{w} )}\cdot  \frac{\exp(-v^3/3 + w^3/3 - r w + r v)}{(v-w)(w - v')} \right| |dw| \leq $$
$$ C\cdot \frac{\prod_{n = 1}^{\rl} | y_m -  {v} |}{\prod_{m = 1}^{\rr} |{v}  - x_n |}   \exp(-\Re[v^3]/3 +(|r_0| + 1)|v|) $$
where $|dw|$ is integration with respect to arc-length and $C$ is a constant that depends on $r_0$. 

Since the the function on the second line is integrable over $C_{a ,3\pi/4}$ (by the cube in the exponential) we conclude by Lemma \ref{AT} that $F_{{ \rm BBP}; \vec{x}, \vec{y}}(r)$ is entire.\\

{\bf \raggedleft Step 3.} In this step we show that when restricted to $\mathbb{R}$ the function $F_{{ \rm BBP}; \vec{x}, \vec{y}} $ is real-valued, non-negative, increasing and converging to $1$ at $\infty$. Arguing as in Step 1, we have that we can deform $C_{a, 3\pi/4}$ to $C_{\mu, 3\pi/4}$ and $C_{b, \pi/4}$ to $\tilde{D}$ as in Definition \ref{DefScaleS3} in (\ref{S4ExpandDet}) without affecting the values of the integrals. Once we perform this deformation we see that for $r \in \mathbb{R}$ we have $F_{{ \rm BBP}; \vec{x}, \vec{y}}(r)   =  \det \left( I + K_r^\infty \right)_{L^2(C_{\mu ,3\pi/4})}$, where $K_r^\infty$ is as in Theorem \ref{mainThmGUE}. If we perform the same scaling as in Definition \ref{DefScaleS3} we see from Theorem \ref{mainThmGUE} and Theorem \ref{LGPT1} that 
\begin{equation}\label{S4RE1}
F_{{ \rm BBP}; \vec{x}, \vec{y}}(r)  = \lim_{M \rightarrow \infty} \mathbb{E} \left[ e^{-u Z^{M + \rl, N + \rr}} \right],
\end{equation}
where we recall from Definition \ref{DefScaleS3} that $u(r,M,N) = e^{W_{M,N} - M^{1/3} \sigma_{\alpha} r}$. In particular, since we have $Z^{M + \rl, N + \rr} > 0$ we see that $F_{{ \rm BBP}; \vec{x}, \vec{y}}(r)$  is the pointwise limit of a sequence of real-valued, non-negative, increasing functions and is thus itself real-valued, non-negative and increasing.

We next show that 
\begin{equation}\label{S4RE2}
\lim_{k \rightarrow \infty} F_{{ \rm BBP}; \vec{x}, \vec{y}}(k) = 1.
\end{equation}
Notice that if $g^k_{v,v'}(w) $ is as in Step 1 for $r = k$ we have
$$|g^k_{v,v'}(w)| \leq \frac{1}{(b-a)^2} \cdot \left| \prod_{n = 1}^{\rr} \frac{( {w} - x_n )}{({v}  - x_n  )} \prod_{m = 1}^{\rl} \frac{( y_m -  {v} )}{(y_m  -{w} )} \right| e^{- k(b -a)}\exp(-\Re[v^3]/3 + \Re[w^3]/3),$$
where we used the fact that $|e^{kz}| \leq 1$ if $\Re[z] \leq 0$. We see that the conditions of Lemma \ref{FDKCT} are satisfied with $\Gamma_1 = C_{a, 3\pi/4}$, $\Gamma_2 = C_{b, \pi/4}$, 
$$F_1(v) = \frac{1}{b-a} \cdot  \frac{\prod_{n = 1}^{\rl} | y_m -  {v} |}{\prod_{m = 1}^{\rr} |{v}  - x_n |}  e^{ - \Re[v^3]/3}, \hspace{2mm}F_2(w) = \frac{1}{b-a} \cdot  \frac{\prod_{m = 1}^{\rr} |{w}  - x_n |}{\prod_{n = 1}^{\rl} | y_m -  {w} |}  e^{ \Re[w^3]/3},$$
and $g^\infty_{v,v'}(w) = 0$. We thus conclude (\ref{S4RE2}) from Lemmas \ref{FDKCT} and \ref{FDCT}.\\

{\bf \raggedleft Step 4.} In this final step we prove that 
\begin{equation}\label{S4RE3}
\lim_{k \rightarrow \infty} F_{{ \rm BBP}; \vec{x}, \vec{y}}(-k) = 0.
\end{equation}
Let us first briefly explain the main idea of the proof. We start by assuming the same scaling as in Definition \ref{DefScaleS3}, and following the notation from Section \ref{Section1.2} define 
$$Z(p,q;P,Q) = \sum_{ \pi \in \Pi_{p,q}^{P,Q}} w(\pi), \mbox{ where } w(\pi) = \prod_{(i,j) \in \pi} w_{i,j},$$
and $\Pi_{p,q}^{P,Q}$ is the collection of all up-right paths from $(p,q)$ to $(P,Q)$. Here we assumed that $1 \leq p \leq P \leq M + \rl$ and $1 \leq q \leq Q \leq N + \rr$. 
In view of (\ref{S4RE1}) applied to $\rr = \rl = 0$ we know that for any $r \in \mathbb{R}$
\begin{equation}\label{S4RE4}
\lim_{M \rightarrow \infty} \mathbb{E} \left[ e^{-u  Z( \rl + 1,  \rr + 1 ; M + \rl, N + \rr) } \right] = F_{GUE}(r).
\end{equation}
Furthermore, by the same equation for general fixed $\rr, \rl$, we know that 
\begin{equation}\label{S4RE5}
\lim_{M \rightarrow \infty}  \mathbb{E} \left[ e^{-u  Z(  1,  1 ;  M + \rl, N + \rr) } \right]= F_{{ \rm BBP}; \vec{x}, \vec{y}}(r).
\end{equation}
The idea now is to show that with high probability 
$$\log Z( 1, 1 ; M + \rl,  N + \rr)  \geq \log Z(\rl+1, \rr + 1 ;  M + \rl, N + \rr) + R_M,$$ 
where $R_M < 0$ is a sequence such that $R_M = o(M^{1/3})$. If true the latter would imply 
\begin{equation}\label{S4RE6}
F_{{ \rm BBP}; \vec{x}, \vec{y}}(r) \leq F_{GUE}(r).
\end{equation}
Since $\lim_{ r \rightarrow -\infty} F_{GUE}(r) = 0$, we would conclude (\ref{S4RE3}) by domination. In the remainder we focus on establishing (\ref{S4RE6}) by performing the comparison of $\log Z( 1, 1 ; M + \rl,  N + \rr)$ and $\log  Z(\rl+1, \rr + 1 ;  M + \rl, N + \rr)  $ that we described above.\\

Let $E_M$ denote the event $\{ \min_{1 \leq i \leq \rl, 1 \leq j \leq \rr} w_{i,j} \leq M^{-1} \}$. Note that by Definition \ref{DefScaleS3} we have that the weights $w_{i,j}$ have distribution
$$f_{\theta_{i,j}}(x) = \frac{{\bf 1} \{ x > 0 \} }{\Gamma(\theta_{i,j})} \cdot x^{-\theta_{i,j} - 1} \cdot \exp( - x^{-1}),$$
where $cM^{-1/3} \leq \theta_{i,j} \leq C$ for some constants $C,c > 0$. In particular, we see that 
\begin{equation}\label{S4RE7}
\mathbb{P}( w_{i,j} \leq M^{-1}) = \frac{1}{\Gamma(\theta_{i,j})} \int_{M}^\infty x^{\theta_{i,j} - 1} e^{-x} dx \leq e^{-M/2},
\end{equation}
where in the last inequality we used that $\Gamma(\theta_{i,j})^{-1}x^{\theta_{i,j} - 1}e^{-x/2} \leq 1$ for all large enough $M$ as long as $ x \geq M$. In particular, we see from (\ref{S4RE7}) that
\begin{equation}\label{S4RE8}
\mathbb{P}( E_M) \leq 1 - [ 1- e^{-M/2}]^{\rr\rl} = O(e^{-M/2}).
\end{equation}
On the event $E^c_M$ we have that 
\begin{equation}\label{S4RE9}
\begin{split}
&Z( 1, 1 ;  M + \rl, N + \rr)  \geq \prod_{i = 1}^{\rr + 1} w_{i, 1} \cdot \prod_{j = 2}^{\rl + 1} w_{\rr + 1, j} \cdot Z( \rl+1, \rr + 1 ; M + \rl, N + \rr) \geq \\
&M^{-\rr - \rl - 1 } \cdot Z( \rl+1, \rr + 1 ; M + \rl, N + \rr) .
\end{split}
\end{equation}
We now observe that if $x \in \mathbb{R}$ and $\epsilon > 0$ are fixed and $u = u(x,M,N)$ we have
\begin{equation}\label{S4qwe}
\begin{split}
&F_{{ \rm BBP}; \vec{x}, \vec{y}}(x -\epsilon ) \leq \liminf_{M \rightarrow \infty} \mathbb{E} \left[ e^{-u  Z(  1,  1 ; M + \rl,  N + \rr) \cdot M^{\rr + \rl +1 } } \right]= \\
& \liminf_{M \rightarrow \infty} \mathbb{E} \left[ e^{-u  Z(  1,  1 ; M + \rl, N + \rr) \cdot M^{\rr + \rl +1 }  } \cdot {\bf 1}\{ E^c_M\} \right] \leq \\
&  \liminf_{M \rightarrow \infty} \mathbb{E} \left[e^{-u  Z( \rl + 1,  \rr + 1 ;  M + \rl, N + \rr) } \cdot {\bf 1}\{ E^c_M\} \right] = \\
&\liminf_{M \rightarrow \infty} \mathbb{E} \left[e^{-u  Z( \rl + 1,  \rr + 1 ; , M + \rl, N + \rr)   }  \right ]= F_{GUE}(x).
\end{split}
\end{equation}
In the first equality we used (\ref{S4RE5}) for $r = x -\epsilon$ and the fact that
$$u(x -\epsilon,M,N) = e^{W_{M,N} - M^{1/3} \sigma_{\alpha} (x-\epsilon)} \geq e^{W_{M,N} - M^{1/3} \sigma_{\alpha} x}M^{\rr + \rl +1 } = u(x,M,N)M^{\rr + \rl +1 } $$
for all large enough $M$. In the second and next to last equality we used (\ref{S4RE8}), and in the last equality we used (\ref{S4RE4}). In the middle inequality we used (\ref{S4RE9}). Letting $\epsilon \rightarrow 0+$ in (\ref{S4qwe}) and using the continuity of $F_{{ \rm BBP}; \vec{x}, \vec{y}}$ from Step 2 we conclude (\ref{S4RE6}). As explained earlier (\ref{S4RE6}) implies (\ref{S4RE3}) and this concludes the proof of the lemma.
\end{proof}

\begin{proof}[Proof of Theorem \ref{BBP}]  As discussed in Remark \ref{S1RemBBP} it suffices to prove the theorem when $N/M \in [\delta,1]$, which we assume in the sequel.

Let $r \in \mathbb{R}$, $\alpha = N/M$ and $\sigma_\alpha, u(r,M,N)$ be as in Definition \ref{DefScaleS3}. Define 
$$f_M(z) = e^{-e^{\sigma_{\alpha}M^{1/3} z}},$$
and observe that since by assumption we have $M \geq N \geq 9$ for all $M$ sufficiently large we can apply Theorem \ref{LGPT1} to conclude
$$\mathbb{E} \left[f_M(\mathcal{F}(M + \rl, N + \rr) - r)\right] = \mathbb{E} \left[ e^{-u \ZP} \right] = \det \left( I + K_u \right)_{L^2(C_{a,3\pi/4})}.$$

On the other hand, by Theorem \ref{mainThmGUE} we know that as $M \rightarrow \infty$ the right-hand side converges to $\det (I + K^{\infty}_{y})_{L^2(C_{\mu, 3\pi/4})}$,
which as shown in Step 3 of the proof of Lemma \ref{S1Lemma} equals $F_{{ \rm BBP}; \vec{x}, \vec{y}}(r) .$

Combining the above statements we see that 
$$\lim_{M \rightarrow \infty} \mathbb{E} \left[f_M(\mathcal{F}(M + \rl,N + \rr) - r) \right] =F_{{ \rm BBP}; \vec{x}, \vec{y}}(r).$$
Since $f_M$ and $F_{{ \rm BBP}; \vec{x}, \vec{y}}$ satisfy the conditions of Lemma \ref{prob} (here we used the continuity of $F_{{ \rm BBP}; \vec{x}, \vec{y}}$) we see that the last equation and Lemma \ref{S1Lemma} imply the statement of the theorem.
\end{proof}

%
\section{Appendix A}\label{Section9}
In Sections \ref{Section2} and \ref{Section3} we needed several results, whose proofs require a careful analysis of the properties of the function $G_{M,N}$, recalled in this section in equation (\ref{S5Gfun}). In this section we study some of these properties, focusing on the way the real part of $G_{M,N}$ varies along different contours.

%
\subsection{Definition and basic properties}\label{Section9.1}
Let $\Psi(x)$ denote the digamma function, i.e.
\begin{equation}\label{digamma}
\Psi(x) = \frac{\Gamma'(x)}{\Gamma(x)} = - \gamma_{E} + \sum_{n = 0}^\infty \left(\frac{1}{n + 1} - \frac{1}{n+z} \right),
\end{equation}
where $\gamma_{E}$ is the Euler constant. Suppose that $M, N \geq 1$ and $\theta > 0$ are given. We let $z_{c}(M,N,\theta)$ denote the maximizer of 
$$W_{M,N}(x) := N \Psi(x) + M \Psi(\theta - x)$$
on the interval $(0, \theta)$. As mentioned in Section \ref{Section3.1} the above expression converges to $- \infty$ as $x \rightarrow 0+$ or $x \rightarrow \theta-$ and also the function $W_{M,N}(x)$ is strictly concave, hence the maximum exists and is unique. Since $z_c$ is the maximizer we have
\begin{equation}\label{zcdef}
0 = W'_{M,N}(z_c) = N \sum_{n = 0}^{\infty} \frac{1}{(n+z_c)^2} - M \sum_{n = 0}^\infty \frac{1}{(n+\theta - z_c)^2}.
\end{equation}

The main object of interest in this section is the following function
\begin{equation}\label{S5Gfun}
G_{M,N}(z) = N \log \Gamma (z) - M \log \Gamma (\theta - z) - W_{M,N}(z_{c}) z - C_{M,N},
\end{equation}
where the constant $C_{M,N}$ is such that $G_{M,N}(z_{c}) = 0$. We also define $\alpha = N/M$ and $G_\alpha(z)$ through the equation
$$G_\alpha(z) = M^{-1} \cdot G_{M,N}(z).$$

In the remainder of this section we derive a few basic properties for the function $G_{M,N}$ or equivalently $G_\alpha$. From (\ref{zcdef}) we know that $z_c(\alpha)$ is the unique number in $(0, \theta)$ that satisfies
$$0 = \alpha \sum_{n = 0}^\infty \frac{1}{(n+z_c)^2} - \sum_{n =0}^\infty \frac{1}{(n+\theta - z_c)^2}.$$ 
In particular, for $z \in (0, \theta)$ the function
$$g(z) = \frac{\sum_{n =0}^\infty \frac{1}{(n+\theta - z)^2}}{ \sum_{n = 0}^\infty \frac{1}{(n+z)^2}} $$
satisfies $g(z_c) = \alpha$. The numerator of $g(z)$ is clearly increasing, while the denominator is decreasing, so that function $g(z)$ is a continuous strictly increasing bijection between $(0, \theta)$ and $(0, \infty)$. Consequently, $g^{-1}$ is also a strictly increasing continuous bijection between $(0,\infty)$ and $(0,\theta)$. We conclude that $g^{-1}([\delta, 1]) = [a,b]$ where $a = g^{-1}(\delta)$ and $b = g^{-1}(1)$. 

 One directly computes $G_\alpha(z_c) = \partial_zG_\alpha(z_c) = \partial_z^2 G_\alpha(z_c) = 0$, while $\partial^3_z G_\alpha(z_c) = 2 \sigma_\alpha^3$ with 
$$\sigma_\alpha := \left( \sum_{n = 0}^\infty \frac{\alpha}{(n+z_c)^3} +  \sum_{n = 0}^\infty \frac{1}{(n+\theta - z_c)^3} \right)^{1/3}.$$
From the discussion above we see that 
\begin{equation}\label{ULBalpha}
\left( \sum_{n = 0}^\infty \frac{\delta}{(n+b)^3} +  \sum_{n = 0}^\infty \frac{1}{(n+\theta - a)^3} \right)^{1/3} \leq \sigma_{\alpha} \leq  \left( \sum_{n = 0}^\infty \frac{1}{(n+a)^3} +  \sum_{n = 0}^\infty \frac{1}{(n+\theta - b)^3} \right)^{1/3}
\end{equation}

We end this section with a formula for the derivative of $G_{M,N}$ along a generic contour. Denote $z(r) =  z_c + x(r) + \i y(r)$. Then we have
\begin{equation*}\label{GenG}
\begin{split}
&\frac{d}{dr} G_{M,N}(z(r)) = (x'(r) + \i y'(r)) \cdot \left( N \Psi(z(r)) + M \Psi(\theta - z(r)) - W_{M,N}(z_{c})  \right) = \\
& (x'(r) + \i y'(r))\cdot \left( N \sum_{n = 0 }^\infty \left(\frac{1}{n + z_{c}} - \frac{1}{n+z(r)} \right) + M \sum_{n = 0 }^\infty \left(\frac{1}{n + \theta - z_{c} }- \frac{1}{n+\theta - z(r)} \right) \right)=\\
&  (x'(r) + \i y'(r))   N \sum_{n = 0 }^\infty\frac{|n + z(r)|^2 - (n+z_{c})(n + \overline{z}(r))}{(n + z_{c})|n+z(r)|^2}  + \\
&  (x'(r) + \i y'(r)) M \sum_{n = 0 }^\infty \frac{|n+\theta - z(r)|^2 - (n+ \theta - z_{c})(n+\theta - \overline{z}(r))}{(n + \theta -  z_{c})|n+\theta - z(r)|^2 }.
\end{split}
\end{equation*}
Taking the real part of the above equation we obtain
\begin{equation}\label{GenForm}
\begin{split}
&\Re \left[\frac{d}{dr} G_{M,N}(z(r)) \right] =  N \sum_{n = 0 }^\infty\frac{ x'(r)(x^2(r) + y^2(r))}{(n + z_{c})|n+z(r)|^2} + M\sum_{n = 0}^{\infty} \frac{ x'(r)(x^2(r) + y^2(r))}{(n + \theta -  z_{c})|n+\theta - z(r)|^2} \\
& +  \left( N \sum_{n = 0}^\infty \frac{x(r)x'(r) - y(r)y'(r)}{|n+z(r)|^2} - M \sum_{n = 0}^\infty \frac{x(r)x'(r) - y(r)y'(r)}{|n + \theta - z(r)|^2} \right).
\end{split}
\end{equation}

%
\subsection{Estimates along $C_{z_c, \phi}$}\label{Section9.2}
In this section we obtain estimates on $Re \left[ G_{M,N}(z)\right]$ when $\arg(z)$ is bounded away from $0, \pi/2, \pi, 3\pi/2$. The main result we prove is the following.
\begin{lemma}\label{HL1} Suppose that $M, N \geq 1$, $\theta > 0$ and $\delta_1 \in (0, \pi/4)$. Then we can find constants $C,R > 0$ that depend on $M,N$, $\theta$ and $\delta_1$ such that if $z = re^{\i\phi}$ with $r \geq R$ we have
\begin{equation}\label{GUBC}
\begin{split}
&\Re \left[G_{M,N}(re^{\i \phi}) \right]  \leq - C  r \log (r) \mbox{ if $\phi \in [\pi/2 + \delta_1, \pi-\delta_1] \cup [\pi + \delta_1, 3\pi/2 - \delta_1]$;} \\
&\Re \left[G_{M,N}( re^{\i \phi}) \right]  \geq C  r \log (r) \mbox{ if $\phi \in [\delta_1, \pi/2 - \delta_1] \cup [-\pi/2 + \delta_1, -\delta_1]$.} 
\end{split}
\end{equation}
\end{lemma}
\begin{proof}
Our approach is to study $\Re \left[G_{M,N}( z_{c} + re^{\i \phi}) \right]$. Let us fix $\theta > 0$ and $\delta_1 > 0$ as in the statement of the lemma. Setting $z(r) = z_c + re^{\i\phi}$ and applying (\ref{GenForm}) we get
\begin{equation}\label{GenAngle}
\begin{split}
\Re \left[\frac{d}{dr} G_{M,N}(z(r)) \right] =  &N \sum_{n = 0 }^\infty \frac{r^2 \cos(\phi) }{(n + z_c)|n+z(r)|^2} -  N \sum_{n = 0 }^\infty\frac{r (\sin^2(\phi) - \cos^2(\phi))}{|n+z(r)|^2} + \\
&M\sum_{n = 0 }^\infty \frac{r^2 \cos(\phi) }{(n + \theta - z_c)|n+\theta - z(r)|^2 }  +  M \sum_{n = 0 }^\infty\frac{r( \sin^2(\phi) - \cos^2(\phi))}{|n + \theta - z(r)|^2}.
\end{split}
\end{equation} 
We study the large positive $r$ behavior of the right side above for different ranges of $\phi$.\\

Suppose first that $\phi \in  [\pi/2 + \delta_1/2, \pi-\delta_1/2] \cup [\pi + \delta_1/2, 3\pi/2 - \delta_1/2]$. Below $C$ will stand for a positive constant, which depends on $\theta$ and $\delta_1$, and whose value may change from line to line.
\begin{equation}\label{NAW1}
\begin{split}
& \sum_{n = 0 }^\infty\frac{r}{|n+ z(r)|^2} =   \sum_{n = 0 }^\infty\frac{ r }{(n+ z_c + r\cos(\phi))^2 + r^2 \sin^2(\phi)} =  \\
&  \sum_{n = 0 }^\infty\frac{r }{(n+ z_c)^2  +2r(n+ z_c) \cos(\phi)  + r^2 }\leq \sum_{n = 0 }^\infty\frac{r }{(n+ z_c)^2  + 2r(n+ z_c) \cos(\pi + \delta_1/2)  + r^2 }  \leq \\
&   \sum_{n = 0 }^\infty\frac{ C r }{(n+z_c)^2   + r^2 } \leq  \frac{ C r}{ z_c^2 + r^2} + \int_0^{\infty} \frac{ C rdx}{x^2 + r^2} = \frac{ C r}{ z_c^2 + r^2} + \frac{ C\pi}{2}.
\end{split}
\end{equation}
 Analogously, one shows that
\begin{equation}\label{NAW2}
\begin{split}
& \sum_{n = 0 }^\infty\frac{r}{|n+\theta - z(r)|^2} \leq \frac{C r}{(\theta - z_c)^2 + r^2} + \frac{C\pi}{2}.
\end{split}
\end{equation}
Furthermore we have
\begin{equation}\label{NAW3}
\begin{split}
& \sum_{n = 0 }^\infty \frac{r^2 \cos(\phi)  }{(n + z_c)|n+z(r)|^2} \leq \sum_{n = 0 }^\infty \frac{ \cos(\pi/2 + \delta_1/2) r^2 }{(n +z_c) ((n+z_c)^2 + 2(n+z_c)r\cos(\phi) + r^2)} \leq \\
&  \sum_{n = 0 }^\infty \frac{ \cos(\pi/2 + \delta_1/2) r^2 }{(n + z_c) ((n+z_c)^2  + r^2)} \leq -C r^2\int_{z_c}^{\infty} \frac{dx}{(x + z_{c})(x^2 + r^2)} = \\
& \frac{ -Cr^2}{r^2 + z_c^2} \left( \frac{z_c }{r} \left(\frac{\pi}{2} - \tan^{-1}(z_c/r) \right)+ \log \left(\frac{\sqrt{z_c^2 + r^2}}{2z_c}\right)   \right) \leq \\
&  \frac{ -Cr^2}{r^2 + z_c^2}\log \left(\frac{\sqrt{z_c^2 + r^2}}{2z_c}\right) \leq -C\log (r),
\end{split}
\end{equation}
where the last inequality holds for all $r \geq r_0$, where $r_0$ depends on $\theta$ alone. Analogous considerations show that
 \begin{equation}\label{NAW4}
\begin{split}
&\sum_{n = 0 }^\infty \frac{r^2 \cos(\phi) }{(n + \theta - z_{c})|n+\theta - z(r)|^2 }\leq -C \log (r),
\end{split}
\end{equation}
where the last inequality holds for all $r \geq r_0$, where $r_0$ depends on $\theta$ alone.
Combining (\ref{NAW1}), (\ref{NAW2}), (\ref{NAW3}) and (\ref{NAW4}) with (\ref{GenAngle}) we conclude that there exist $r_0$ and $c_0$ depending on $\theta$ and $\delta_1$ such if $\phi \in  [\pi/2 + \delta_1/2, \pi-\delta_1/2] \cup [\pi + \delta_1/2, 3\pi/2 - \delta_1/2]$ and $r \geq r_0$ we have
\begin{equation}\label{derGdec}
\Re \left[ \frac{d}{dr} G_{M,N}(z_c + re^{\i \phi}) \right]  \leq - c_0 (M+N)\log (r).
\end{equation}
The first equation in (\ref{GUBC}) is now an immediate consequence of (\ref{derGdec}) since $G_{M,N}(z_c) = 0$.

A similar argument for $\phi \in [ \delta_1/2,  \pi/2 - \delta_1/2] \cup [-\pi/2 + \delta_1/2, - \delta_1/2]$ reveals that there exists $r_0$ and $c_0$ depending on $\theta$ and $\delta_1$ such that if $r \geq r_0$ then
\begin{equation}\label{derGinc}
\Re \left[ \frac{d}{dr} G_{M,N}(z_c + re^{\i \phi}) \right]  \geq c_0 (M+N)\log (r).
\end{equation}
The second equation in (\ref{GUBC}) is now an immediate consequence of (\ref{derGinc}) since $G_{M,N}(z_c) = 0$.
\end{proof}

%
\subsection{Estimates along vertical contours}\label{Section9.3} In this section we study the properties of $\Re \left[ G_{M,N}(z(r))\right]$ along contours of the form $C_{z_c, \pi/2}$, i.e. when we move vertically up or down from the point $z_c$. The main result we prove is as follows.

\begin{lemma}\label{LemmaVert}
Suppose that $M \geq N \geq 1$. If we set $z(r) = z_c + \i r$ we claim
\begin{equation}\label{GVer}
\begin{split}
&\Re \left[\frac{d}{dr} G_{M,N}(z_c + \i r) \right]\geq  0 \mbox{ when $r \geq 0$ and } \Re \left[\frac{d}{dr} G_{M,N}(z_c + \i r) \right]\leq  0 \mbox{ when $r \leq 0$}.
\end{split}
\end{equation}
\end{lemma}
\begin{proof}
Notice that since $M\geq N$ we have $\theta \geq 2z_c$. In view of (\ref{GenForm}) we can reduce (\ref{GVer}) to establishing 
\begin{equation}\label{GenFormVer}
 N \sum_{n = 0}^\infty \frac{1}{|n+z(r)|^2} - M \sum_{n = 0}^\infty \frac{1}{|n + \theta - z(r)|^2} \leq 0
\end{equation}
for $r \geq 0$. Recall from (\ref{zcdef}) that
$$N \sum_{n = 0}^\infty \frac{1}{(n+z_c)^2} = M \sum_{n = 0}^\infty \frac{1}{(n+\theta -z_c)^2},$$
and so it suffices to show that
\begin{equation}\label{RedQ1}
 \sum_{m \geq 0} \sum_{n \geq 0}  \frac{1}{(m+\theta - z_c)^2} \frac{1}{|n+z(r)|^2} -  \sum_{m \geq 0} \sum_{n \geq 0}\frac{1}{(m+z_c)^2} \frac{1}{|n + \theta - z(r)|^2} \leq 0.
\end{equation}

If $m = n$ the above reads
$$\frac{1}{(n + \theta - z_c)^2} \frac{1}{(n+z_c )^2 +r^2} - \frac{1}{(n+z_c)^2} \frac{1}{(n+\theta  - z_c)^2 + r^2} = $$
$$ \frac{-r^2(\theta-2z_c)(2n + \theta)}{(n+\theta -z_c)^2((n+z_c)^2 + r^2)(n+z_c)^2((n+\theta - z_c)^2 + r^2)},$$
which is non-positive as $\theta \geq 2z_c$. 

 We split the remaining sum as follows
$$ \sum_{m,n \geq 0, m\neq n}  \frac{1}{(m+\theta - z_c)^2} \frac{1}{|n+z(r)|^2} -  \sum_{m,n \geq 0, m\neq n}\frac{1}{(m+z_c)^2} \frac{1}{|n + \theta - z(r)|^2} = $$
$$\sum_{n = 0}^\infty \sum_{m = n+ 1}^\infty \frac{1}{(m+\theta -z_c)^2} \frac{1}{(n+z_c)^2 + r^2} + \frac{1}{(n+\theta -z_c)^2} \frac{1}{(m+z_c)^2 + r^2} $$
$$- \frac{1}{(m+z_c)^2}\frac{1}{(n + \theta - z_c)^2 + r^2} - \frac{1}{(n+z_c)^2}\frac{1}{(m + \theta - z_c)^2 + r^2} = $$
$$\sum_{n = 0}^\infty \sum_{m = n+ 1}^\infty \frac{-r^2(n+m+\theta)(m-n +\theta -2z_c)}{(m+\theta -z_c)^2(n+z_c)^2((n+z_c)^2+r^2)((m+\theta -z_c)^2 + r^2)} + $$
$$ \frac{r^2(n+m+\theta)(m -n + 2z_c  - \theta)}{(m+z_c)^2(n+\theta-z_c)^2((m+z_c)^2 + r^2)((n+\theta -z_c)^2 + r^2)}.$$
We want to show that each summand above is negative. We see that the first term is always negative and so if we increase its denominator we will make the summand larger. Notice that
\begin{equation}
\begin{split}
&(m+\theta -z_c)^2(n+z_c)^2((n+z_c)^2+r^2)((m+\theta -z_c)^2 + r^2) - \\
&(m+z_c)^2(n+\theta-z_c)^2((m+z_c)^2 + r^2)((n+\theta -z_c)^2 + r^2) = \\
& -(\theta - 2z_c)(m-n) (a r^4 + b r^2 +c),
\end{split}
\end{equation}
where 
$$a = 2mn + \theta (m+n) + 2\theta z_c - 2z_c^2 > 0, b= (2n +\theta)(m-n)^3 + 4(n+z_c)(n+\theta -z_c) (m-n)^2 + $$
$$ (2n + \theta)(2n^2 + 2n \theta + \theta^2 - 2\theta z_c + 2z_c^2)(m-n) + 2(n+z_c)(n + \theta - z_c)(n^2 + n\theta  +\theta^2 - 3\theta z_c + 2z_c^2) > 0,$$
$$c = (2mn + m\theta + n \theta + 2\theta z_c - 2z_c^2)( (2n^2 + 2n\theta + \theta^2 - 2\theta z_c + 2z_c^2) (m-n)^2 + $$
$$ + 2(n + z_c)(2n + \theta)(n+\theta - z_c) + 2 (n+z_c)^2 (n + \theta -z_c)^2 > 0.$$
We conclude that the second denominator is larger than the first and so by replacing it we make the summands larger. But then the summand becomes
$$ \frac{r^2(n+m+\theta)( 4z_c  - 2\theta)}{(m+z_c)^2(n+\theta-z_c)^2((m+z_c)^2 + r^2)((n+\theta -z_c)^2 + r^2)},$$
which is clearly non-positive. We conclude that the sum is negative (unless $M = N$ in which case it is $0$) and this proves (\ref{RedQ1}).
\end{proof}

As an immediate corollary we have the following result.
\begin{lemma}\label{PosLemma} Suppose that $M \geq N \geq 1$ and $\theta > 0$. Then for any $x > 0$ and $y \in \mathbb{R}$ with $|y| \geq x$
\begin{equation} \label{eqsum1}
\Re \left[ G_{M,N}(z_c + x +\i y) \right] \geq 0.
\end{equation}
\end{lemma}
\begin{proof}
Let us briefly explain the idea of the proof. We will construct $y_0 \in \mathbb{R}$ and a contour connecting the points $z_c + \i y_0$ to $z_c + x + \i y$ such that $\Re[G_{M,N}(z)]$ increases along this contour. This would imply through (\ref{GVer}) that
\begin{equation}\label{EstHyp}
\Re \left[ G_{M,N}(z_c + x +\i y) \right] \geq \Re \left[ G_{M,N}(z_c  +\i y_0) \right]\geq  \Re \left[ G_{M,N}(z_c) \right] = 0.
\end{equation}

Let $C =y^2 - x^2 \geq 0$ and set $y_0 = \sqrt{C}$ if $y \geq 0$ and $y_0 = -\sqrt{C}$ if $y < 0$. Consider the curve $z(r) =z_c+  r \pm \i \sqrt{C + r^2}$, where the positive sign is taken when $y \geq 0$  and the negative is taken otherwise; here $r$ varies in $[0,x]$. The curve $z(r)$ connects $y_0$ to $(x,y)$ as $r$ ranges from $0$ to $x$. In addition, writing $z(r) = z_c + x(r) + \i y(r)$ we see that $x'(r) x(r) - y'(r)y(r) = 0$, which in view of (\ref{GenForm}) shows that $\Re \left[ G_{M,N}(z(r)) \right]$ increases on $[0,x]$. From this we conclude (\ref{EstHyp}) where the second inequality follows from (\ref{GVer}) and last equality is true by the definition of $z_c$.
\end{proof}

%
\subsection{Proof of Lemmas \ref{analGS3} and  \ref{ReleL1S3} }\label{Section9.4}

In this section we give the proof of  Lemmas \ref{analGS3} and \ref{ReleL1S3}, whose statements are recalled here for the reader's convenience. We follow the same notation as in Section \ref{Section3.2}.

\begin{lemma}\label{S6analGS3}
Fix $\theta > 0$, $\delta \in (0,1)$ and assume that $M \geq N \geq 1$, $N/M \in [\delta, 1]$. 
There exist constants $C > 0$ and $r > 0$ depending on $\theta$ and $\delta$ such that $G_{\alpha}$ is analytic in the disc $|z - z_c| < r$ and the following hold for $|z - z_c| \leq r$:
\begin{equation}\label{S6powerGS3}
\begin{split}
&\left|G_\alpha(z)+ (z-z_c)^3\sigma^3_\alpha/3 \right| \leq C |z-z_c|^4 \mbox{ whenever $|z - z_c| < r$};\\
& \Re[G_\alpha(z) ] \geq  (\sqrt{2}/2)^3|z-z_c|^3\sigma^3_\alpha/6  \mbox{ when $z \in C_{z_c, \phi}$ with $\phi =  \pi/4$; }\\
&\Re[G_\alpha(z) ] \leq - (\sqrt{2}/2)^3|z-z_c|^3 \sigma^3_\alpha/6  \mbox{ when $z \in C_{z_c, \phi}$ with $\phi =  3\pi/4$ }.
\end{split}
\end{equation}
In the above equations $z_c, \sigma_\alpha$ are as in Definition \ref{DefScaleS3} and $C_{a, \phi}$ is as in Definition \ref{ContV}.
\end{lemma}
\begin{proof}
As explained in the beginning of this section $z_c(\alpha)$ defines an increasing continuous between $(0,\infty)$ and $(0, \theta)$. We let $[a,b]$ denote the preimage of $[\delta, 1]$ under this bijection.

Let $r_0 > 0$ be sufficiently small, depending on $\theta$ and $\delta$, such that $\theta -b \geq 2r_0$ and $a \geq 2r_0$. The singularities of $G_{\alpha}$ come from the gamma functions in its definition, and are at $z = 0, -1, -2, \dots$ and $z = \theta , \theta + 1, \theta + 2, \dots$. In particular, by our choice of $r_0$ we know that the closure of $B_{r_0}(z_c) := \{z \in \mathbb{C}: |z - z_c| < r_0\}$ is disjoint from those sets. By Hadamard theorem, see \cite[Theorem 4.4, Chapter 2]{Stein}, we can expand $G_{\alpha}$ in absolutely convergent power series inside $B_{r_0}(z_c)$. From the definition of $G_\alpha$ we have that $G_\alpha(z_c) = \partial_zG_\alpha(z_c) = \partial_z^2 G_\alpha(z_c) = 0$, while $\partial^3_z G_\alpha(z_c) = 2 \sigma_\alpha^3.$ The latter implies that 
$$G_\alpha(z) = \frac{1}{3} \cdot \sigma_\alpha^3 \cdot (z- z_c)^3 + \sum_{ n = 4}^\infty a_n(\alpha) (z-z_c)^n.$$

Define $K = \{z \in \mathbb{C}: d(z, [a,b]) \leq r_0\}$. Since $G_\alpha(z)$ is jointly continuous on $(z,\alpha) \in K \times [\delta, 1]$, which is compact, we know that we can find a constant $A > 0$, depending on $\delta$ and $\theta$ such that 
$$\sup_{\alpha \in [\delta, 1], z \in K} |G_\alpha(z)| \leq A.$$
Using the Cauchy inequalities, see \cite[Corollary 4.3, Chapter 2]{Stein}, we know that 
$$|a_n(\alpha)| \leq Ar_0^{-n}.$$
The above inequalities show that if $|z - z_c| \leq r_0/2$, then 
$$\left|G_\alpha(z) -  \frac{1}{3} \cdot \sigma_\alpha^3 \cdot (z- z_c)^3\right| \leq \sum_{n = 4}^\infty |a_n(\alpha)| \cdot |z - z_c|^n \leq A |z-z_c|^4 r_0^{-4}  \cdot \sum_{n = 4}^\infty (1/2)^{n-4} = 2A |z-z_c|^4 r_0^{-4}  ,$$
which proves the first line in (\ref{S6powerGS3}) for any $r \in (0, r_0/2)$ and $C = 2A r_0^{-4}$.

Next we pick $r \leq r_0$ sufficiently small so that $Cr < W = (2\sqrt{2})\sigma^3_\alpha/48$ for all $\alpha \in [\delta, 1]$. Then 
$$\left| \Re[G_\alpha(z)   + (z-z_c)^3\sigma^3_\alpha/3 ] \right| \leq \left| G_\alpha(z)+ (z-z_c)^3\sigma^3_\alpha/3 \right| \leq  C |z-z_c|^4 \leq Cr |z-z_c|^3 \leq W |z-z_c|^3.$$

The above implies for $z \in C_{z_c, \phi}$ with $\phi = \pi/4$ that
$$\Re[G_\alpha(z) ]\geq -\Re[ (z-z_c)^3]\sigma^3_\alpha/3 - W |z-z_c|^3 =|z-z_c|^3 \cdot \left((2 \sqrt{2}) \sigma^3_\alpha/24 - W \right) = (\sqrt{2}/2)^3|z-z_c|^3\sigma^3_\alpha/6. $$
Similarly if $z \in C_{z_c, \phi}$ with $\phi = 3\pi/4$  we get
$$\Re[G_\alpha(z) ]\leq  -\Re[ (z-z_c)^3]\sigma^3_\alpha/3+ W |z-z_c|^3 = - (\sqrt{2}/2)^3|z-z_c|^3\sigma^3_\alpha/6. $$
The latter two inequalities imply the remainder of (\ref{S6powerGS3}).
\end{proof}

\begin{lemma}\label{ReleL1S5} Fix $\theta > 0$, $\delta \in (0,1)$ and assume that $M \geq N \geq 1$, $N/M \in [\delta, 1]$. Let $z(r) = z_c  + r e^{\i\phi}$. Then we have
\begin{equation}\label{eqG1S5}
\begin{split}
&\Re \left[\frac{d}{dr} G_{M,N}(z(r)) \right] \leq 0, \mbox{ provided that $\phi = 3\pi/4$ or $\phi = 5\pi/4$,}\\
&\Re \left[\frac{d}{dr} G_{M,N}(z(r)) \right] \geq 0, \mbox{ provided that $\phi = \pi/4$ or $\phi = -\pi/4$.}
\end{split}
\end{equation}
In addition, if $\delta_1  > 0$ is given then there is a constant $c > 0$ depending on $\delta_1, \delta, \theta$ such that
\begin{equation}\label{eqG2S5}
\begin{split}
\Re \left[G_{M,N}(z(r)) \right] \leq -c (M+N) r \log(1 + r) \mbox{ for any $r \geq \delta_1$ provided $\phi = 3\pi/4$ or $\phi = 5\pi/4$,}\\
\Re \left[G_{M,N}(z(r)) \right] \geq c (M+N) r \log(1 + r) \mbox{ for any $r \geq \delta_1$ provided $\phi = \pi/4$ or $\phi = -\pi/4$,}\\
\end{split}
\end{equation}
\end{lemma}
\begin{proof}
From (\ref{GenAngle}) we get whenever $\phi = \pi/4$ or $3\pi/4$ or $5\pi/4$ or $7\pi/4$ that
\begin{equation}\label{GenAnglev2}
\begin{split}
\Re \left[\frac{d}{dr} G_{M,N}(z(r)) \right] =  &\sum_{n = 0 }^\infty \frac{N r^2   \cos(\phi)}{(n + z_c)|n+z(r)|^2}+ \sum_{n = 0 }^\infty \frac{M r^2 \cos(\phi)}{(n + \theta - z_c)|n+\theta - z(r)|^2 },
\end{split}
\end{equation}
which automatically proves (\ref{eqG1S5}). 

Below $C$ stands for a positive constant, whose value depends on $\theta$ and $\delta$ alone, and whose value may change from line to line. By repeating the same arguments as in (\ref{NAW3}) we see that if $\phi = 3\pi/4$ or $5\pi/4$ we have
\begin{equation}\label{NAW3v2}
\begin{split}
&\Re \left[\frac{d}{dr} G_{M,N}(z(r)) \right] \leq -C (M+N) \log (r),\\
\end{split}
\end{equation}
where the last inequality holds for all $r \geq r_0 $, where $r_0$ depends on $\theta$ alone. Furthermore, from (\ref{GenAnglev2}) we have for any $r > 0$ that
\begin{equation}\label{NAW3v2}
\begin{split}
&\Re \left[\frac{d}{dr} G_{M,N}(z(r)) \right] \leq \frac{-N(\sqrt{2}/2) r^2}{z_c (z_c^2 + r^2 )} + \frac{-Mr^2(\sqrt{2}/2)}{(\theta - z_c)((\theta - z_c)^2 + r^2)} \leq -C(M+N)r^2.
\end{split}
\end{equation}

Let $R_0 \geq r_0$ be sufficiently large so that for any $r \geq R_0$ we have
$$\int_{r_0}^r \log x dx = (r\log r - r) - (r_0 \log r_0 - r_0) \geq (1/2) r\log r.$$
Then, using (\ref{eqG1S5}) we see that if $r \in [\delta_1, R_0]$ we have
$$\Re \left[G_{M,N}(z(r)) \right] = \int_0^{r} \Re \left[\frac{d}{dx} G_{M,N}(z(x)) \right]dx \leq \int_{0}^{r} -C(M+N)x^2dx \leq -(C/3)(M+N)\delta_1^3,$$
while if $r \geq R_0$ we have
$$\Re \left[G_{M,N}(z(r)) \right] = \int_0^{r} \Re \left[\frac{d}{dx} G_{M,N
}(z(x)) \right]dx \leq \int_{r_0}^{r} \Re \left[\frac{d}{dx} G_{M,N}(z(x)) \right]dx \leq $$
$$\leq \int_{r_0}^{r} -C(M+N) \log(x)dx = - (C/2)(M+N)r \log r .$$
The last two inequalities imply the first line in (\ref{eqG2S5}). The second one is derived analogously.
\end{proof}

%
\subsection{Proof of Lemma \ref{horGS3}}\label{Section9.5}

In this section we give the proof of  Lemma \ref{horGS3}, whose statement is recalled here for the reader's convenience. We follow the same notation as in Section \ref{Section3.2}.

\begin{lemma}\label{horG}Fix $\theta, A > 0$, $\delta \in (0,1)$ and assume that $M \geq N \geq 1$, $N/M \in [\delta, 1]$. 
There exist constants $M_0, C_0 > 0$ depending on $\delta, \theta, A$ such that if $M \geq M_0$, $x \in [-A, A]$ and $z(r) =  z_c  + r e^{i\phi}$ with $z_c$ as in Definition \ref{DefScaleS3}, $r \geq 0$ and $\phi \in \{\pi/4, 3\pi/4, 5\pi/4, 7\pi/4\}$ we have
\begin{equation}\label{horGEq}
\left| \Re [G_{M,N}(z(r) + xM^{-1/3})  - G_{M,N}(z(r))] \right| \leq C_0 M^{2/3} \cdot ( 1 + r).
\end{equation} 
\end{lemma}
\begin{proof} Let us denote $z(r,w) = z_c  + r e^{\i\phi} + w  = x(w) + \i y(w)$ for $w M^{1/3} \in [-A, A]$  . From equation (\ref{GenForm}) we have
\begin{equation}\label{GenForm2}
\begin{split}
&\Re \left[\frac{d}{dw} G_{M,N}(z(r,w)) \right] =  N \sum_{n = 0 }^\infty\frac{ (z_c + w)^2 + 2r \cos (\phi) (z_c+ w) +r^2 }{(n + z_{c})((n + z_c + w)^2 + 2r (n + z_c + w) \cos (\phi) +r^2 )}  +  \\
&M\sum_{n = 0}^{\infty} \frac{(z_c + w)^2 + 2r \cos (\phi) (z_c+ w) +r^2}{(n + \theta -  z_{c})((n +  \theta - z_c - w)^2 -2 r(n +  \theta - z_c - w) \cos (\phi) +r^2)} +  \\
& N \sum_{n = 0}^\infty \frac{z_c + w + r \cos (\phi) }{(n + z_c + w)^2 + 2r (n + z_c + w) \cos (\phi) +r^2} -  \\
& M \sum_{n = 0}^\infty \frac{z_c + w + r \cos (\phi)}{(n +  \theta - z_c - w)^2 -2 r(n +  \theta - z_c - w) \cos (\phi) +r^2}.
\end{split}
\end{equation}
Below $C$ will stand for a positive constant, which depends on $\theta , A, \delta$. We first note that for $M$ large enough
\begin{equation}\label{NAW1P1}
\begin{split}
&  \sum_{n = 0 }^\infty\frac{ 1}{(n + z_c + w)^2 + 2r (n + z_c + w) \cos (\phi) +r^2 }  \leq  \sum_{n = 0 }^\infty\frac{ C }{(n+z_c + w)^2   + r^2 } \leq  \\
&  \frac{ C }{ (z_c + w)^2 + r^2} +  \frac{ C }{ (z_c + w + 1)^2 + r^2} +   \int_1^{\infty} \frac{ C dx}{x^2 + r^2} \leq \frac{ C }{ 1 + r} .
\end{split}
\end{equation}
Analogously, one shows that
\begin{equation}\label{NAW2P2}
\begin{split}
& \sum_{n = 0}^\infty \frac{1}{(n +  \theta - z_c - w)^2 -2 r(n +  \theta - z_c - w) \cos (\phi) +r^2} \leq \frac{C}{1 + r}.
\end{split}
\end{equation}
Combining (\ref{GenForm2}) with (\ref{NAW1P1}) and (\ref{NAW2P2}) we conclude that for all large enough $M$
$$\left|\Re \left[\frac{d}{dw} G_{M,N}(z(r,w)) \right] \right| \leq C(1+r),$$
and then (\ref{horGEq}) follows by the mean value theorem.
\end{proof}

%
\section{Appendix B}\label{Section10} In this section we provide the proofs of various results we used in Sections \ref{Section2} and \ref{Section3}. 

We summarize several estimates about special functions that will be required in the proofs below. Note that if $\mbox{dist}(s, \mathbb{Z}) \geq c$ for some fixed constant $c> 0$ and $s = x+ \i y$ then there exists a constant $c' > 0$ such that
\begin{equation}\label{sinIneq}
\left| \Gamma(-s) \Gamma(1 + s) \right| = \left| \frac{\pi}{ \sin(\pi s)} \right| \leq c' e^{-\pi |y|}.
\end{equation}
It follows from \cite[(2), pp. 32]{Luke} that if $z \in \mathbb{C}$ and $|\arg(z)| \leq \pi - \epsilon$ for some $\epsilon \in (0, \pi)$ then 
\begin{equation}\label{AssGamma}
\Gamma(z) = e^{-z} z^{z - 1/2} (2\pi )^{1/2} \cdot \left( 1 + O(z^{-1}) \right),
\end{equation}
where the constant in the big $O$ notation depends on $\epsilon$ and we take the principal branch of the logarithm. Also by \cite[Theorem 1.6, Chapter 6]{Stein} there are positive constants $c_1, c_2$ such that 
\begin{equation}\label{growthGamma}
\left| \frac{1}{\Gamma(z)} \right|\leq c_1 e^{c_2 |z| \log |z|}.
\end{equation}

%
\subsection{Proof of Propositions \ref{PE1v2} and \ref{PE3v2}}\label{Section10.1}
In this section we give the proofs of Propositions \ref{PE1v2} and \ref{PE3v2}, whose statements are recalled here for the reader's convenience.

\begin{proposition}\label{PE1v3}
Fix $M, N \geq 1$, $\vec{a} = (a_1, \dots, a_N) \in \mathbb{R}^N$, $\vec{\alpha}= (\alpha_1, \dots, \alpha_M) \in \mathbb{R}^M$, $T \geq 0$, a compact set $K \subset \mathbb{C}$ and $v,u\in \mathbb{C}$ with $\Re(u) > 0$. Then there exist positive constants $L_0, C_0, c_0$ depending on all previous constants and $K$ such that if $\tau \in[0,T]$, $x,a \in K$, $w = a + z$, $\arg(z) \in [\pi/4, \pi/3] \cup [-\pi/3, -\pi/4]$, $|w| \geq L_0$
\begin{equation}\label{PE1v3e1}
\left|\frac{\pi }{\sin[(\pi (v-w))}  \frac{\prod_{m = 1}^M  \Gamma(x+ \alpha_m - w )}{\prod_{n = 1}^N\Gamma(w - a_n) }  u^{w-v} e^{\tau (w^2 - v^2)/2} \right| \leq C_0 e^{- c_0 |w| \log |w|}.
\end{equation}
\end{proposition}

\begin{proposition}\label{PE3v3}
Fix $M, N \geq 1$, $\vec{a} = (a_1, \dots, a_N) \in \mathbb{R}^N$, $\vec{\alpha}= (\alpha_1, \dots, \alpha_M) \in \mathbb{R}^M$, $\tau > 0$, a compact set $K \subset \mathbb{C}$ and $v,u\in \mathbb{C}$ with $\Re(u) > 0$. Then there exist positive constants $L_0, C_0, c_0$ depending on all of the previous constants such that if $w = a + z$ with $a \in K$, $\arg(z) \in [\pi/4, \pi/2] \cup [-\pi/2, -\pi/4]$, $|w| \geq L_0$ we have
\begin{equation}\label{PE3v3e1}
\left|\frac{\pi }{\sin(\pi (v-w))} \frac{\prod_{m = 1}^M  \Gamma( \alpha_m - w )}{\prod_{n = 1}^N\Gamma(w - a_n) } u^{w-v}e^{\tau (w^2 - v^2)/2}  \right| \leq C_0 e^{- c_0 |w| \log |w|}.
\end{equation}
\end{proposition}

As the proofs of the above propositions are quite similar, we will combine them.
\begin{proof}

We consider the two cases $|\arg(z)| \in [\pi/4, \pi/3]$ and $|\arg(z)| \in [\pi/3, \pi/2]$ separately, with the second one only being relevant for Proposition \ref{PE3v3}. Below we fix an arbitrary $\theta >0$, say $\theta = 1$.

Assume first that $|\arg(z)| \in [\pi/4, \pi/3]$. Below $c_i$ denote positive constants. Observe that
$$ \frac{\Gamma(\theta - w)^M}{\Gamma(w)^N} = \exp\left(- G_{M,N}(w) + W_{M,N}(z_{c})w + C_{M,N} \right).$$
In particular, by Lemma \ref{HL1} we have for all $|w|$ sufficiently large that
\begin{equation}\label{U11}
\left|  \frac{\Gamma(\theta - w)^M}{\Gamma(w)^N}  \right| \leq \exp \left( - c_1|w| \log |w| \right).
\end{equation}
Furthermore by (\ref{AssGamma}) we have for all $x \in K$ and $|w|$ sufficiently large that
\begin{equation}\label{U12}
\left|  \frac{\prod_{m = 1}^M\Gamma(x + \alpha_m - w)}{\Gamma(\theta - w)^M} \cdot \frac{\Gamma(w)^N}{\prod_{n = 1}^N\Gamma( w - a_n)}  \right| \leq |w|^{c_2} .
\end{equation}
In addition, by (\ref{sinIneq}) we have for all $|w|$ sufficiently large that
\begin{equation}\label{U13}
\left|  \frac{\pi }{\sin(\pi (v-w))} \right| \leq  c_3.
\end{equation}
Finally, we have for all $|w|$ sufficiently large and $\tau \in [0,T]$ that
\begin{equation}\label{U14}
\left|u^{w-v} e^{\tau (w^2-v^2)/2} \right| \leq c_4 e^{c_5 |w|},
\end{equation}
where we used the fact that 
$$|e^{\tau z^2/2}| = \exp \left( \tau \Re[z^2]/2 \right) \leq 1,$$
since $\Re[z^2] \leq 0$ by our assumption that $|\arg(z)| \in [\pi/4, \pi/3]$. Combining (\ref{U11}), (\ref{U12}), (\ref{U13}) and (\ref{U14}) we conclude (\ref{PE1v3e1}) and so Proposition \ref{PE1v3} is proved. Also (\ref{U11}), (\ref{U13}) and (\ref{U14}) prove Proposition \ref{PE3v3} in the case $|\arg(z)| \in [\pi/4, \pi/3]$.\\

We next suppose that $|\arg(z)| \in [\pi/3, \pi/2]$ and finish the proof of Proposition \ref{PE3v3} in this case. Combining (\ref{sinIneq}) and (\ref{growthGamma}) we see that for all $|w|$ sufficiently large 
\begin{equation}\label{U15}
\left| \frac{\prod_{m = 1}^M  \Gamma( \alpha_m - w )}{\prod_{n = 1}^N\Gamma(w - a_n) } \right| \leq \exp \left( c_6|w| \log |w| \right).
\end{equation}
Also, since $\tau > 0$, we have
\begin{equation}\label{U16}
\left|u^{w-v} e^{v\tau (w-v)+ \tau (w-v)^2/2}\right| \leq c_7 e^{c_8 |w| - \tau|w|^2/4},
\end{equation}
where we used 
$$|e^{\tau z^2/2}| = \exp \left( \tau \Re[z^2]/2 \right) \leq \exp \left( - \tau |z|^2/4 \right) ,$$
which in turn relies on our assumption that $|\arg(z)| \in [\pi/3, \pi/2]$. Combining (\ref{U13}), (\ref{U15}) and (\ref{U16}) we conclude that for all $|w|$ sufficiently large we have
$$\left|\frac{\pi }{\sin(\pi (v-w))} \frac{\prod_{m = 1}^M  \Gamma( \alpha_m - w )}{\prod_{n = 1}^N\Gamma(w - a_n) } u^{w-v}e^{\tau (w^2 - v^2)/2}  \right|  \leq c_9 \exp( -c_{10} |w|^2),$$
which certainly implies (\ref{PE3v3e1}). This concludes the proof of Proposition \ref{PE3v3}.
\end{proof}

%
\subsection{Proof of Proposition \ref{PE2v2}}\label{Section10.2} In this section we give the proof of Proposition \ref{PE2v2}, whose statement is recalled here for the reader's convenience.

\begin{proposition}\label{PE2v3}
Fix  $M, N \geq 1$,  $T \geq 0$, $\tau \in [0,T]$, $\vec{a} = (a_1, \dots, a_N) \in \mathbb{R}^N$, $\vec{\alpha}= (\alpha_1, \dots, \alpha_M) \in \mathbb{R}_{>0}^M$ and $u \in \mathbb{C}$ with $\Re(u) > 0$. Put $\theta_0 = \min(\vec{\alpha}) - \max(\vec{a})$ and assume that $\theta_0 > 0$ and $\delta_0 \in (0, \min(1/4, \theta_0 /16))$. Suppose that $v, v' \in C_{a, \phi}$ as in Definition \ref{ContV} with $a \in [ \max(\vec{a}) + \delta_0,  \min(\vec{\alpha}) - 5\delta_0]$ and $\phi \in [3\pi/4, 5\pi/6]$. Finally, fix $b \in [a+ 2\delta_0, \min(\vec{\alpha}) - 3\delta_0]$ and denote by $D_v$ the contour $D_v(b, \pi/4, \delta_0)$ as in Definition \ref{ContE}. Then there exists a positive constant $ C_0$ depending on $\vec{a}, \vec{\alpha}, \delta_0, u, N, M, T$ (and not $\tau$) such that if $x \in \mathbb{C}$ with $d(x, [ 0, 1]) \leq \delta_0$ then we have
\begin{equation}\label{PE2e1v3}
  \left| \int_{D_v} \frac{\pi d\mu(w)}{\sin(\pi (v-w))}   \prod_{n = 1}^N \frac{\Gamma(v - a_n)}{\Gamma(w - a_n)} \prod_{m = 1}^M \frac{\Gamma(\alpha_m + x - w)}{\Gamma(\alpha_m + x- v)} \frac{u^{w-v} e^{\tau (w^2 - v^2)/2}}{w - v'} \right| \leq \frac{C_0}{1 + |v|^2}.
\end{equation}
\end{proposition}
\begin{proof}
 In the arguments below, $c_i$ will denote positive constants that depend on the parameters in the statement of the proposition. Let us put $\theta = \min(\vec{\alpha})$. We notice that we can rewrite the integrand as
\begin{equation}
 F(w)= \frac{H(w)}{H(v)}\cdot \frac{\pi\exp\left[ G_{M,N}(v) - G_{M,N}(w) \right] }{\sin( \pi (v-w))} \frac{u^{w-v} e^{\tau (w^2-v^2)/2} }{w - v'},
\end{equation}
where $G_{M,N}$ is as in (\ref{S5Gfun}) and 
$$H(w) = \exp(- W_{M,N}(z_{c})w ) \cdot \prod_{n = 1}^N \frac{\Gamma(w)}{\Gamma(w - a_n)} \prod_{m = 1}^M \frac{\Gamma(\alpha_m + x - w)}{\Gamma(\theta - w)}.$$

We split the integral over the two pieces $D_v^1$ and $D_v^2$ (see Definition \ref{ContE}). If $w \in D_v^1$ and $|w|$ is sufficiently large we may apply Lemma \ref{HL1} to get
\begin{equation}\label{U21}
\left | \exp ( - G_{M,N}(w))\right| \leq \exp\left(- c_1 |w|\log|w| \right).
\end{equation}
 On the other hand, using (\ref{sinIneq}), we have for all $w \in D_v^1$ and $v,v' \in C_{a, \phi}$
\begin{equation}\label{U22}
\left | \frac{\pi}{\sin(\pi (v-w))} \right| \leq c_2 \mbox{ and } \left| \frac{1}{ w - v'}\right| \leq c_3.
\end{equation}

Setting $v = a + r_v \cos(\phi) + \i r_v \sin(\phi)$ and $w = b + r_w \cos(\pi/4) \pm \i r_w \sin(\pi/4)$, we get
$$\Re[w^2 - v^2] =(-\cos^2(\phi) + \sin^2(\phi))r_v^2 - 2a \cos(\phi) r_v  + b^2 - a^2 + 2b \cos(\pi/4) r_w,$$
where we used that $\cos(\pi/4) = \sin(\pi/4)$. In particular, by our assumption on $\phi$ we have $ (-\cos^2(\phi) + \sin^2(\phi)) \leq 0$ and we conclude that
\begin{equation}\label{U23}
\left|e^{\tau (w^2-v^2)/2} \right|\leq  c_4 \cdot e^{c_5(|w|+|v|)}.
\end{equation}
Also, from (\ref{AssGamma}) we have 
\begin{equation}\label{U24}
\left|\frac{H(w)}{H(v)} \right| \leq c_6 \cdot e^{c_7 (|w| + |v|)}.
\end{equation}

Combining (\ref{U21}), (\ref{U22}), (\ref{U23}) and (\ref{U24}) we conclude that 
\begin{equation}\label{U25}
|F(w)| \leq c_8 e^{ c_9(|w| +  |v|) -  c_{10}  (M+N) |w|\log|w|}  \cdot \exp \left( \Re \left[ G_{M,N}(v) \right] \right).
\end{equation}
From the above we conclude that 
\begin{equation}\label{KD5}
 \int_{D^1_v}| F(w) | |dw|  \leq c_{11}  e^{c_9 |v|} \cdot \exp \left( \Re \left[ G_{M,N}(v) \right] \right) \leq c_{12}\exp( - c_{13} |v| \log|v|),
\end{equation}
where the last inequality follows from Lemma \ref{HL1}.\\

We next focus our attention on $D_v^2$. Recall that $D_v^2$ consists of straight oriented segments that connect $z_-$ to $v + 2\delta_0 - i\delta_0$ to $v + 2\delta_0 + i\delta_0$ to $z_+$, where $z_-, z_+$ are the points on $C_{b,\pi/4} $ that have imaginary parts $\Im(v) - \delta_0$ and $\Im(v) + \delta_0$ respectively. Explicitly, we have $z_- = b + \Im(v) -\delta_0 + \i \Im(v) - i\delta_0$ and $ z_+ = b + \Im(v) + \delta_0 + \i \Im(v) + \i\delta_0$.

\begin{figure}[h]
\scalebox{0.6}{\includegraphics{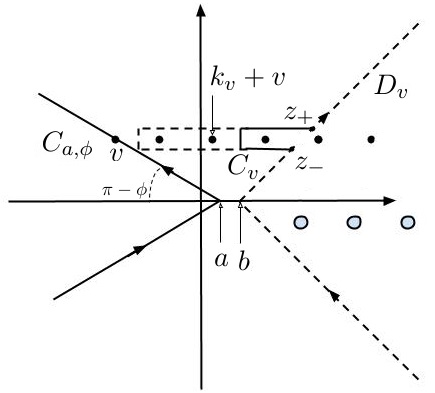}}
\captionsetup{width=\linewidth}
 \caption{The figure depicts the contour $C_{a, \phi}$ and $D_v = D_v(b,\pi/4, \delta_0)$ with $\min(\vec{\alpha}) > b > a > \max(\vec{a})$. The black dots denote the points $v, v+1, v+2, \dots$ and the grey ones denote $\alpha_m + x, \alpha_m + x + 1, \alpha_m +  x + 2, \dots$. The contour $C_v$ is the filled contour that connects $z_-$ and $z_+$. }
\label{S6_1}
\end{figure}
If $|v|$ is bounded, we observe that $F(w)$ remains bounded on $D_v^2$ as it is bounded away from all of its poles, so we only need to focus on large $|v|$. We assume that $|v|$ is sufficiently large so that $D_v^2$ is well separated from the real line and so $z_-, z_+$ lie in the same complex half-plane. For simplicity we will assume that we are in the upper complex half-plane (the other case is completely analogous).

We define a contour connecting $z_-$ and $z_+$, denoted by $C_{v}$ as follows, see also Figure \ref{S6_1}. It starts at $z_-$ and goes horizontally to $v + k_v + 1/2 - \i\delta_0$, then up to $v+k_v + 1/2+ \i\delta_0$ and then horizontally to $z_+$. Here $k_v$ denotes the largest integer such that $\Re[v + k_v ] < - 2a  + \Im(v)$. We assume that $|v|$ is sufficiently large so that $k_v \geq 1$. Notice that the length of $C_v$ does not exceed $3 + 2\delta_0 + 4|a| + 2|b|$. 

An application of the residue theorem gives
\begin{equation}\label{KDsth}
\frac{1}{2\pi \i}\int_{D_v^2} F(w) = \sum_{j = 1}^{k_v}R_j(v)  + \frac{1}{2\pi \i} \int_{C_v} F(w) dw,
\end{equation}
where the residue at $w = v + j$ is denoted by $R_j(v)$ and given by the formula
\begin{equation}\label{res1}
R_j(v) = -(-u)^{j}  \prod_{n = 1}^N \frac{\Gamma(v - a_n)}{\Gamma(v+ j - a_n)} \prod_{m = 1}^M \frac{\Gamma(\alpha_m + x - v- j)}{\Gamma(\alpha_m + x- v)} \frac{e^{\tau v j+ \tau j^2/2}}{v - v' + j}.
\end{equation}
From the functional equation for the gamma function $\Gamma(z+1)  = z \Gamma(z)$ we know that 
$$\left| \frac{\Gamma(v - a_n)}{\Gamma(v+j - a_n)} \right| = \prod_{k = 1}^j \left|\frac{1}{v - a_n + k - 1} \right| \leq \frac{1}{|\Im(v)|^j}\mbox{ and }$$
$$\left| \frac{ \Gamma(\alpha_m + x - v - j)}{\Gamma(\alpha_m + x - v)} \right| =  \prod_{k = 1}^j \left|\frac{1}{\alpha_m + x - v - k} \right|\leq \frac{1}{(|\Im(v)| - \delta_0)^j},$$
where in the last inequality we used that $\Im(x) \in [-\delta_0, \delta_0]$. 

Next we observe that if $v = a + r_v \cos(\phi) + \i r_v \sin(\phi)$ then
$$\Re[v j + j^2/2] = j (a + r_v \cos(\phi)) + j^2/2 \leq  j (a + r_v \cos(\phi) + (r_v \sin(\phi) - 2a - r_v \cos (\phi) ) /2) \leq 0,$$
where in the last inequality we used that $r_v \cos(\phi) + r_v \sin(\phi) \leq 0$ by our choice of $\phi$ and the first inequality used $1 \leq j \leq k_v$ and $\Re[v + k_v ] < - 2a  + \Im(v)$. The conclusion is that for all large enough $|v|$ we have
$$ \left| e^{\tau v j+ \tau j^2/2} \right| \leq 1$$
for all $j = 1, \dots, k_v$ and $\tau \geq 0$. Finally, if we assume that $|v|$ is sufficiently large we can ensure that $\frac{\max(1, |u|)}{(|\Im(v)|- \delta_0)} \leq \frac{1}{2}$. Combining all of  the above estimates we see that
\begin{equation}\label{KD4v2}
\sum_{j = 1}^{k_v } |R_j(v)|  \leq \sum_{j = 1}^{k_v } \frac{|u|^j}{(|\Im(v)| - \delta_0)^{j(M+N)}} \frac{1}{|v - v' + j|} \leq  \frac{c_{14}}{(|\Im(v)| - \delta_0)^{M+N}} \leq \frac{c_{15}}{1 + |v|^2},
\end{equation}
where in the second inequality we used that $v - v'$ is uniformly bounded away from $\mathbb{Z}$ and that the sum of the geometric series is controlled by its first term. The last inequality used that $M + N \geq 2$ and holds for all $|v|$ large enough.

Since the contour $C_v$ is (at least distance $\delta_0$) bounded away from $v + \mathbb{Z}$ the estimates in (\ref{U22}) hold. Provided that $|v|$ is sufficiently large and $w \in C_v$ the estimates in (\ref{U21}) and (\ref{U24}) also holds. We next wish to show that (\ref{U23}) also holds for $w \in C_v$.

Suppose that $w \in C_v$, and write it as $w = y + r_w \cos(\pi/4) + \i r_w \sin(\pi/4)$, where $y$ is a real number between $-2|a| - 2|b| - \delta_0 - 3/2$ and $2|a| + 2|b| + \delta_0 + 3/2$ and $r_w \in [\Im(v) - \delta_0, \Im(v) + \delta_0]$. As done before in  (\ref{U23}) we have
$$\Re[2 v(w-v) + (w-v)^2] =   (-\cos^2(\phi) + \sin^2(\phi))r_v^2 - 2a \cos(\phi) r_v  + y^2 - a^2 + 2y \cos(\pi/4) r_w.$$
 In particular, by our assumption on $\phi$ we have $ (-\cos^2(\phi) + \sin^2(\phi)) \leq 0$ and so (\ref{U23}) also holds.
Combining (\ref{U21}), (\ref{U22}), (\ref{U23}) and (\ref{U24}) we conclude that for all large $v$ and $w \in C_v$ one has the estimate as in (\ref{U25}) and since $C_v$ has length at most $ 3 + 2\delta_0 + 4|a| + 2|b|$, we conclude from Lemma \ref{HL1} that
\begin{equation}\label{remainsKD0}
\left|  \frac{1}{2\pi \i} \int_{C_v} F(w) dw\right| \leq c_{17}\exp( - c_{18} |v| \log|v|).
\end{equation}
The proposition now follows from combining (\ref{KD5}),  (\ref{KDsth}),(\ref{KD4v2}) and  (\ref{remainsKD0}).

\end{proof}

%
\subsection{Proof of Proposition \ref{PKBS3}}\label{Section10.3} In this section we give the proof of  Proposition \ref{PKBS3}, whose statement is recalled here for the reader's convenience. We follow the same notation as in Sections \ref{Section3.1} and \ref{Section3.2}.

\begin{proposition}\label{PKBS6}
Let $A \geq 0$ and $x\geq -A$ be given and let $M, N, u(x, M, N)$ be as in Definition \ref{DefScaleS3}. Then for any $\epsilon > 0$ there exist positive constants $M_0, c_0 ,C_0 > 0$ depending on $A,\epsilon$ and the parameters in  Definition \ref{DefScaleS3} such that if $M \geq M_0$, $v, v' \in C_{a,3\pi/4}$  we have
\begin{equation}\label{KernelBoundV0S6}
\left|K_u (v,v')\right| \leq C_0M^{1/3} \mbox{ and if $|v - a| \geq \epsilon$ we further have }\left|K_u (v,v')\right| \leq C_0 e^{-c_0 M \log (1 + |v|)}.
\end{equation}
If $\rr = \rl = 0$ then $M_0, c_0 ,C_0 > 0$ depend only on $\theta$, $\delta$, $A$ and $\epsilon$. 
\end{proposition}
\begin{proof}

Before we go to the main argument, we give a conceptual roadmap of the proof. For clarity we split the proof into several steps. In the first four steps we prove (\ref{KernelBoundV0S6}) when $v$ is bounded away from $a$ and in the last step we establish it when $v$ is close to $a$. In both cases we deform the contour $D_v$ in the definition of $K_u(v,v')$ to a suitable contour where we can estimate the integrand in the definition of $K_u(v,v')$ using the results from Sections \ref{Section9} and \ref{Section10}. In the case when $v$ is close to $a$ we can deform $D_v$ to a contour that is a horizontal translation of $C_{z_c, \pi/4}$ by $O(M^{-1/3})$ (recall that this contour was defined in Definition \ref{ContV}). Ideally, we would prefer to work with the contour $C_{z_c, \pi/4}$, because Lemma \ref{S6analGS3} and Lemma \ref{ReleL1S5} give us very good control of the absolute value of the integrand on this contour. A mild translation of $C_{z_c, \pi/4}$ is, however, necessary to ensure that the poles at $\alpha_m$ and $a_n$ are on the correct sides of the contour. The idea then is to combine our good estimates on $C_{z_c, \pi/4}$ with Lemma \ref{horG}, which ensures that slight horizontal movements away from $ C_{z_c, \pi/4}$ does not significantly damage our estimates. This idea is also present in the case when $v$ is bounded away from $a$. The new challenge that arises in this case is that we need to deform $D_v$ to a new contour, which avoids the poles at $v + \mathbb{Z}$ coming from the sine function in the definition of $K_u(v,v')$. The idea then is to work with a contour $D_v(b, \pi/4, d)$ as in in Definition \ref{ContE}. This contour consists of two parts -- $D^1$ and $D^2$. Estimating the absolute value of the integrand on $D^1$ is  similar to the case we discussed above and relies only on Lemma \ref{S6analGS3}, Lemma \ref{ReleL1S5} and Lemma \ref{horG}. The analysis of the integral over $D^2$ is more involved and requires further deforming this contour to the half-plane $\{ \Re(z) \geq z_c\}$. Once we are in the half-plane $\{ \Re(z) \geq z_c\}$ we can utilize Lemma \ref{PosLemma} and get the estimates we desire -- this is done in Step 3. Unfortunately, in the process of deforming $D^2$ we pick up a large number of poles, whose residues need to be controlled as well -- this is done in Step 4 and requires a fairly involved case-by-case analysis. We now turn to the proof.\\

{\bf \raggedleft Step 1.}  Let $r$ be as in Lemma \ref{S6analGS3} for the choice of $\delta$ and $\theta$ as in the proposition. We also set $r_\epsilon = \min(r/2,\epsilon,1/2)$,  and for $y > 0$ we write $B_y(z_c) := \{z \in \mathbb{C}: |z - z_c| < y\}$. We also define $d_\epsilon = \min(r_\epsilon/10,1/20, d_0)$, where $d_0 > 0$ is sufficiently small so that
\begin{equation}\label{smallD}
d_0 \leq \frac{\theta - z_c}{4}, \hspace{2mm} d_0 \leq \frac{z_c}{4} \mbox{, and } (1- s)\sqrt{(1-s)^2 + (\theta - z_c-s)^2} \geq 1 + s \mbox{ for $s \in [0, d_0]$}.
\end{equation}
Notice that all of the above constants can be chosen uniformly in $\theta, \epsilon$ and $\delta$ when $\rr = \rl = 0$.

We first suppose that $|v -a| \geq \epsilon$ and establish (\ref{KernelBoundV0S6}) in this and the next three steps. In the sequel we will write $G$ for $G_{\alpha}$, $\sigma$ for $\sigma_\alpha$, $d$ for $d_\epsilon$ and $r$ for $r_\epsilon$ to ease the notation. Then we have by Cauchy's theorem that if $M_0$ is sufficiently large and $M \geq M_0$ that
\begin{equation}\label{DefFS6}
\begin{split}
&K_u(v,v') = \frac{1}{2\pi \i} \int_{D}  F(w,v',v)dw, \hspace{2mm} \mbox{ with } F(w,w',v) = \prod_{n = 1}^{\rr} \frac{\Gamma(v - a_n)}{\Gamma(w - a_n)} \prod_{m = 1}^{\rl} \frac{\Gamma(\alpha_m - w)}{\Gamma(\alpha_m - v)} \times \\
&\frac{\pi  e^{M[G(v) - G(w)]}e^{M^{1/3} x \sigma (v-z_c)}e^{-M^{1/3} x \sigma (w-z_c)} }{\sin(\pi(v-w)) (w- v')},
\end{split}
\end{equation}
where $D$ stands for the contour $D_v(b, \pi/4, d)$ in Definition \ref{ContE} with $d$ as above, $b = a + \rho \sigma^{-1} M^{-1/3}$ and $\rho$ as in Definition \ref{DefScaleS3}. The decay estminates necessary to deform the contour $D_v$ in the definition of $K_u(v,v')$ to $D$ near infinity come from Proposition \ref{PE1v3} applied to $K = [0,\theta]$, $T =0$ and $\vec{a}, \vec{\alpha}, u$ as in the statement of the proposition. In order to ensure that the poles at $\theta, \theta + 1, \theta + 2, \dots$ stay on the right side of $D$ (and hence are not crossed in the process of deformation) we need to ensure that $a + \rho \sigma^{-1}M^{-1/3} = z_c + ( \mu + \rho) \sigma^{-1}M^{-1/3} < \theta$, which is possible by making $M_0$ sufficiently large depending on $\theta$ and $\delta$ alone. Also, by the definition of $\rho$, we have that if $\rl > 0$ then the poles at $\alpha_i$ for $i = 1, \dots, \rl$ are also to the right of $D$ for large enough $M$ in view of $|v-a| \geq r$ and $d \leq r/10$. We denote by $D^1$ and $D^2$ the portions of $D$ that are part of $C_{b, \pi/4}$ or not respectively, see Definition \ref{ContV}.

In the following steps $C,c$ stand for generic positive constants, depending on the parameters in the statement of the proposition, whose value may change from line to line. \\

{\bf \raggedleft Step 2.} In this and the next step we assume that $|v-a| \geq r$. Notice that if $\Re(z) \leq 0$ we have
\begin{equation}\label{ExpX}
|e^{M^{1/3} x \sigma z}| \leq 1 \mbox{ if $x \geq 0$ and so for $x \geq -A$ we have } |e^{M^{1/3} x \sigma z}|  \leq e^{AM^{1/3} \sigma |z|}.
\end{equation}
The latter equation together with Lemma \ref{ReleL1S5} and Lemma \ref{horG} imply for all large enough $M$
\begin{equation}\label{OP11}
\begin{split}
&| e^{MG(v)} e^{M^{1/3} \sigma (v-z_c) x } | \leq Ce^{-c M |v| \log (1 + |v|)} \mbox{ if $|v - a| \geq r$, and $v \in C_{a, 3\pi/4}$} \\
&| e^{-MG(w)} e^{-M^{1/3} \sigma (w-z_c) x } | \leq Ce^{-c M |w| \log (1 + |w|)} \mbox{ if $|w - b| \geq r/2$ and $w \in C_{b, \pi/4}$}. 
\end{split}
\end{equation}
To see why the latter is true, note that Lemma \ref{ReleL1S5} implies the above inequalities when $a, b$ are replaced with $z_c$ and then Lemma \ref{horG} allows us to replace $z_c$ with $a,b$ since such a change affects the right side of the inequalities by $Ce^{c M^{2/3} (1 + |v|)}$ for the first line and $Ce^{c M^{2/3} (1 + |w|)}$ for the second line. Of course, this effect is negligible compared to the exponential in $M$ decay.

In addition, using (\ref{sinIneq}) we have for all $w \in D^1$ that
\begin{equation}\label{OP12}
\left|\frac{1} {\sin(\pi(v-w)) (w- v')} \right| \leq CM^{1/3}.
\end{equation}
The extra $M^{1/3}$ is coming from the $|w-v'|$ in the denominator, which is lower bounded by $M^{-1/3}  \sigma^{-1} \rho$ as $v'$ is allowed to be close to $a$ unlike $v$. We also observe that 
\begin{equation}\label{OP12.1}
\begin{split}
&\left|\prod_{n = 1}^{\rr} \frac{\Gamma(v - a_n)}{\Gamma(w - a_n)} \prod_{m = 1}^{\rl} \frac{\Gamma(\alpha_m - w)}{\Gamma(\alpha_m - v)}\right| = \left|\prod_{n = 1}^{\rr} \frac{\Gamma(v - a_n + 1)}{\Gamma(w - a_n + 1)} \prod_{m = 1}^{\rl} \frac{\Gamma(\alpha_m - w + 1)}{\Gamma(\alpha_m - v + 1)}\right|  \times \\
& \left|\prod_{n = 1}^{\rr} \frac{w - a_n}{v - a_n} \prod_{m = 1}^{\rl} \frac{\alpha_m - v}{\alpha_m - w}\right| \leq C \exp \left( c |v| \log (1 + |v|) + c |w| \log(1 + |w|) + c \log M \right),
\end{split}
\end{equation}
where we used (\ref{AssGamma}) and (\ref{growthGamma}) to upper bound the gamma functions, and the extra $\log M$ term is coming from $\alpha_m - w$ in the denominator, whose absolute value is bounded away from $0$ by $ (1/2) (\min(\vec{y}) - \mu - \rho) M^{-1/3} \sigma^{-1}$ for all large enough $M$, see Definition \ref{DefScaleS3}.

Combining (\ref{OP11}), (\ref{OP12}) and (\ref{OP12.1}) we conclude that for all large enough $M$
\begin{equation}\label{OP13}
\left| \int_{D^1 \cap B^c_{r/2}(b)}  F(w,v',v)dw \right| \leq C  e^{-c M |v| \log (1 + |v|)}.
\end{equation}

Suppose next that $w \in  D^1 \cap B_{r/2}(b)$ (observe that this piece is separated from $D^2$ by our assumption on $v$). Using (\ref{ExpX}), Lemma \ref{S6analGS3} and Lemma \ref{horG} we have
$$| e^{-MG(w)} e^{-M^{1/3} \sigma (w-z_c) x } | \leq C e^{cM^{2/3}}e^{-M \sigma^3 (\sqrt{2}/2)^3|w-z_c|^3/6} e^{AM^{1/3} \sigma |w - z_c|}.$$
In deriving the above inequality we used that our $r$ is smaller than the one in Lemma \ref{S6analGS3} so that the lemma is applicable. Combining the last inequality with the first line in (\ref{OP11}), (\ref{OP12}), (\ref{OP12.1}) we conclude that for all large enough $M$
\begin{equation}\label{OP14}
\begin{split}
&\left| \int_{D^1 \cap B_{r/2}(b)} \hspace{-10mm} F(w,v',v)dw \right| \leq C  e^{-c M |v| \log (1 + |v|)}.
\end{split}
\end{equation}
 Equations (\ref{OP13}) and (\ref{OP14}) provide the desired estimates for the integral of $F(w,v',v)$ on $D^1$. We consider the integral over $D^2$ in the next step.\\

{\bf \raggedleft Step 3.} In this step we establish the following inequality
\begin{equation}\label{OP15}
\left| \int_{D^2 }  F(w,v',v)dw \right| \leq C  e^{-c M \log (1 + |v|)}.
\end{equation}
Combining (\ref{OP13}), (\ref{OP14}) and (\ref{OP15}) we conclude (\ref{KernelBoundV0S6}) in the case when $|v -a| \geq r$, which of course implies the statement when $|v-a| \geq \epsilon$ as $r \leq \epsilon$ by construction. 

To prove (\ref{OP15}) we define a new contour $C_v$ as follows, see also Figure \ref{S6_3}. Let $k_v \in \mathbb{Z}_{\geq 0}$ be the unique integer such that $\Re(v + k_v) \in [ z_c - d, z_c + d] $ if such an integer exists (the uniqueness follows from our assumption that $d \leq 1/20$); if no such integer exits we let $k_v$ be the largest integer such that $\Re(v + k_v) \leq z_c$. If $\Re(v + k_v) \in [ z_c - d, z_c + d] $ we let $C_v$ be the contour that horizontally connects the points $z_-$ to $z_c + 1/2 + \i \Im(z_-)$ to $z_c +1/2 + \i \Im(z_+)$ to $z_+$, where $z_{\pm}$ are the two points shared by $D^1$ and $D^2$, see the left part of Figure \ref{S6_3}. If $\Re(v + k_v) \not \in [ z_c - d, z_c + d] $ we let $C_v$ be the contour that connects $z_-$ to $z_c + \i \Im(z_-)$ to $z_c + \i \Im(z_+)$ to $z_+$, see the right part of Figure \ref{S6_3}. Observe that by construcion we have that if $z = z_c + (\mu + \rho)\sigma^{-1}M^{-1/3} + f + \i g \in C_v$ then $|g| \geq f \geq 0$, $d(z, v + \mathbb{Z}) \geq d$ and $d(z, C_{a, 3\pi/4}) \geq d/2$ for large enough $M$. Also the length of $C_v$ is at most $2 |\Im(v)| + 3d$ for large enough $M$.

\begin{figure}[h]
\centering
\scalebox{0.6}{\includegraphics{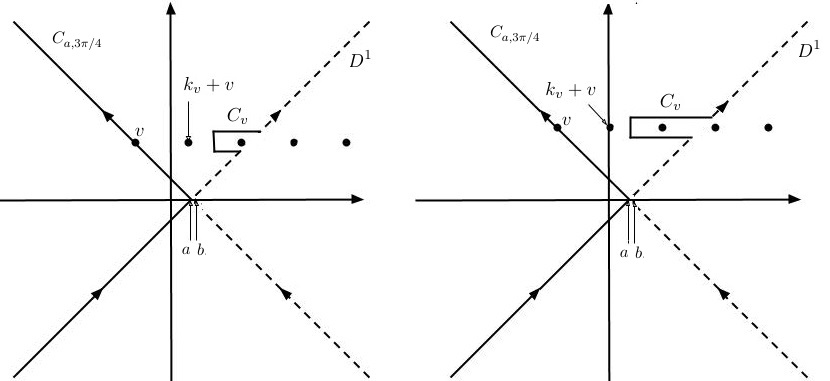}}
 \captionsetup{width=.9\linewidth}
\caption{The left part depicts $C_v$ when $Re(v + k_v) \in [ z_c - d, z_c + d] $ and the right part depicts it when $Re(v + k_v) \not \in [ z_c - d, z_c + d]. $ }
\label{S6_3}
\end{figure}

An application of the residue theorem, see \cite[Corollary 2.3, Chapter 3]{Stein}, gives
\begin{equation}\label{OP16}
\frac{1}{2\pi \i}\int_{D^2} F(w,v',v) = \sum_{j = 1}^{k_v }R_j(v)  + \frac{1}{2\pi \i} \int_{C_v} F(w,v',v) dw,
\end{equation}
where the residue at $w = v + j$ is denoted by $R_j(v)$ and given by the formula
\begin{equation*}
R_j(v) = -(-e)^{-x \sigma M^{1/3} \cdot j} \frac{\Gamma(v)^N}{\Gamma(v+j)^N} \frac{ \Gamma(\theta - v - j)^M}{\Gamma(\theta - v)^M} \frac{1}{v - v' + j} \prod_{n = 1}^{\rr} \frac{\Gamma(v - a_n)}{\Gamma(v + j - a_n)} \prod_{m = 1}^{\rl} \frac{\Gamma(\alpha_m - v - j)}{\Gamma(\alpha_m - v)}.
\end{equation*}
Observe that from (\ref{ExpX}), Lemma \ref{PosLemma} and Lemma \ref{horG} we have for $w \in C_v$ that
\begin{equation}\label{OP16.1}
\left| e^{-MG(w)} e^{-M^{1/3} \sigma (w-z_c) x}   \right| \leq Ce^{c M^{1/3}A|v| + cM^{2/3} |v| }.
\end{equation}
Let us elaborate on equation (\ref{OP16.1}) briefly. Equation (\ref{ExpX}) allows us to bound $|e^{-M^{1/3} \sigma (w-z_c) x}|$ by $e^{c M^{1/3}A |v|}$, which in turn is controlled by $e^{c M^{1/3}A |w|}$ as $|w|$ is upper and lower bounded by constant multiples of $|v|$ for $w \in C_v$. Furthermore, Lemma \ref{PosLemma} bounds $|e^{-MG(w)}|$ by $1$ when $|\Im [ w- z_c] \geq \Re [w - z_c ] \geq 0$. By construction of $C_v$ we know that $\Re [w - z_c ] \geq 0$ is satisfied for all points on $C_v$. The inequality $|\Im [ w- z_c]| \geq \Re [w - z_c ] $ can fail for some points on $C_v$ but only those that are at most distance $(\mu + \rho)\sigma^{-1}M^{-1/3} $ (if the latter is positive) from the points $z^{\pm}$. For these $w$ we know from Lemma \ref{horG} that $\Re[-MG(w)] =  \Re[-MG(w -(\mu + \rho)\sigma^{-1}M^{-1/3} )] + O(M^{2/3}( 1 + |w|)),$ and since by Lemma \ref{PosLemma} the first term is negative, we see that $\Re[-G(w)]  \leq cM^{2/3}|v|$ for all $w \in C_v$.

Since $C_v$ is bounded away from $v + \mathbb{Z}$, and $C_{a, 3\pi/4}$ we know that (\ref{OP12}) holds, and as it is bounded away from $\mathbb{R}$ we also know that (\ref{OP12.1}) holds. Combining the first line of (\ref{OP11}), with (\ref{OP12}), (\ref{OP12.1}), (\ref{OP16.1}) and the fact that the length of $C_v$ is at most $2 |\Im(v)| + 3d$ we conclude that for all large enough $M$
\begin{equation}\label{OP17}
\left| \int_{C_v}  F(w,v',v)dw \right| \leq C  e^{-c M |v| \log (1 + |v|)}.
\end{equation}

We now claim that 
\begin{equation}\label{OP18}
 \sum_{j = 1}^{k_v}|R_j(v)|  \leq C  e^{-c M \log (1 + |v|)}.
\end{equation}
If the latter is true, then (\ref{OP16}), (\ref{OP17}) and (\ref{OP18}) would imply (\ref{OP15}). We establish (\ref{OP18}) in the next step.\\

{\bf \raggedleft Step 4.} We prove (\ref{OP18}) by considering the cases when $|\Im(v)| \leq 1 - d$, $|\Im(v)| \in [1-d, 1+ d]$, $|\Im(v)| \in [ 1-d, 3]$ and $|\Im(v)| \geq 3$. The case when $|\Im(v)| \leq 1 - d$ is trivial since then $k_v = 0$ and the sum in (\ref{OP18}) is empty. The focus is on the remaining three cases.

From the functional equation for the gamma function $\Gamma(z+1)  = z \Gamma(z)$ we know that 
\begin{equation}\label{GammaRatGeom}
\begin{split}
&\left| \frac{\Gamma(v)}{\Gamma(v+j)} \right| = \prod_{k = 1}^j \left|\frac{1}{v + k - 1} \right| \mbox{, } \left| \frac{ \Gamma(\theta - v - j)}{\Gamma(\theta - v)} \right| =  \prod_{k = 1}^j \left|\frac{1}{\theta - v - k} \right| \mbox{ and } \\
& \left|  \prod_{n = 1}^{\rr} \frac{\Gamma(v - a_n)}{\Gamma(v + j - a_n)} \prod_{m = 1}^{\rl} \frac{\Gamma(\alpha_m - v - j)}{\Gamma(\alpha_m - v)} \right| = \prod_{n = 1}^{\rr} \prod_{k = 1}^j \left| \frac{1}{v + k - 1 - a_n}\right|  \prod_{n = 1}^{\rl} \prod_{k = 1}^j \left|\frac{1}{\alpha_m- v - k} \right|.
\end{split}
\end{equation}
If $|\Im(v)| \in [1-d, 1+ d]$ we have that $k_v = 1$ and so the left side of (\ref{OP18}) becomes
\begin{equation*}
\begin{split}
&|R_1(v)| =  \frac{|e^{-x\sigma M^{1/3}}|}{|v - v' + 1|} \frac{1}{|v|^N \cdot |\theta - v - 1|^M}\prod_{n = 1}^{\rr} \left| \frac{1}{v  - a_n}\right|  \prod_{n = 1}^{\rl} \left|\frac{1}{\alpha_m- v - 1} \right| \leq  \\
& \frac{C\exp(A \sigma M^{1/3})}{|\Im(v)|^N (\sqrt{(1-d)^2 + (\theta - a + |\Im(v)| - 1)^2} )^M} \leq \frac{C\exp(A \sigma M^{1/3})}{ ((1-d)\sqrt{(1-d)^2 + (\theta - a  - d)^2})^M} 
\end{split}
\end{equation*}
where we used that $|v| \geq 1-d $, $M \geq N$ and (\ref{ExpX}). By our choice of $d \leq d_0$ with $d_0$ as in (\ref{smallD}), we get for all large enough $M$
$$ |R_1(v)| \leq  \frac{\exp(CM^{1/3})}{(1+d/2)^M},$$
which implies (\ref{OP18}) in the case $|\Im(v)| \in [1-d, 1+ d]$ as long as $M$ is sufficiently large.

Next, suppose that $|\Im(v)| \in[1 + d, 3]$. Then from (\ref{GammaRatGeom}) we have
\begin{equation}\label{GammaRatIneq}
\begin{split}
&\left| \frac{\Gamma(v)}{\Gamma(v+j)} \right| \leq \frac{1}{|\Im(v)|^j}\mbox{ and }\left| \frac{ \Gamma(\theta - v - j)}{\Gamma(\theta - v)} \right| \leq \frac{1}{|\Im(v)|^j}\\ 
& \left|  \prod_{n = 1}^{\rr} \frac{\Gamma(v - a_n)}{\Gamma(v + j - a_n)} \prod_{m = 1}^{\rl} \frac{\Gamma(\alpha_m - v - j)}{\Gamma(\alpha_m - v)} \right| \leq C.
\end{split}
\end{equation}
The above inequalities imply that (observe $k_v \leq 3$)
\begin{equation}\label{OP19}
 \sum_{j = 1}^{k_v }|R_j(v)|  \leq \frac{\exp(CM^{1/3})}{(1 + d)^{N + M}},
\end{equation}
which implies (\ref{OP18}) in the case $|\Im(v)| \in [1+d, 3]$ as long as $M$ is sufficiently large.

Finally, we suppose that $|\Im(v)| \geq 3$. In this case (\ref{GammaRatIneq}) imply
\begin{equation}\label{OP20}
 \sum_{j = 1}^{k_v }|R_j(v)|  \leq \sum_{j = 1}^{k_v} \frac{C \cdot \exp({A \sigma M^{1/3} j})}{|\Im(v)|^{j(M+N)}} \leq C \sum_{j = 1}^{k_v}|\Im(v)|^{-jM} \leq C |\Im(v)|^{-M},
\end{equation}
where in the next to last inequality we used that $|\Im(v)|^{N} \geq \exp({A\sigma M^{1/3}})$ for all large enough $M$, and in the last inequality we bounded the sum by a geometric series. Equation (\ref{OP20}) implies (\ref{OP18}) in the case $|\Im(v)| \geq 3$ since
$$\log (1 + |v|) \leq \log (1 + 2 |\Im(v)|) \leq \log(3) + \log |\Im(v)|\leq 2 \log |\Im(v)|.$$

{\bf \raggedleft Step 5.} In this step we establish (\ref{KernelBoundV0S6}) when $|v - a| \leq r$. Observe that by Cauchy's theorem we may deform the $D_v$ in the definition of $K_u(v,v')$ to the contour $C_{b, \pi/4}$ without affecting the value of the integral. The latter is because $v + 1$ is strictly to the right of $C_{b, \pi/4}$ whenever $|v-z_c| \leq r$, by our choice of $r$. Thus we do not cross any poles in the process of the deformation. The decay estminates necessary to deform the contour $D_v$ in the definition of $K_u(v,v')$ to $C_{b, \pi/4}$ near infinity come from Proposition \ref{PE1v3} applied to $K = [0,\theta]$, $T =0$ and $ \vec{a}, \vec{\alpha}, u$ as in the statement of the proposition. 

Consequently, we have
$$K_u(v,v') = \frac{1}{2\pi \i} \int_{C_{b, \pi/4}}  F(w,v',v)dw,$$
where $F(w,v',v)$ is as in (\ref{DefFS6}). Next, by Lemma \ref{S6analGS3},  Lemma \ref{horG} and (\ref{ExpX}) we have
\begin{equation}\label{OP21}
| e^{MG(v)} e^{M^{1/3}\sigma x (v-z_c)} | \leq \exp( - (\sqrt{2}/2)|v-z_c|^3 \sigma^3 M/6 + CM^{2/3}) \leq e^{CM^{2/3}},
\end{equation}
and also by Lemma \ref{ReleL1S5}, Lemma \ref{horG} and (\ref{ExpX}) provided that $M$ is sufficiently large we have
\begin{equation}\label{OP22}
| e^{-MG(w)} e^{-M^{1/3} \sigma (w-z_c) } | \leq Ce^{-c M |w| \log (1 + |w|)} \mbox{ if $|w - b| \geq r/2$ and $w \in C_{b, \pi/4}$}. 
\end{equation}
One also observes that the inequalities in (\ref{OP12}) and (\ref{OP12.1}) are also satisfied if $|v - a| \leq r$ and $|w - b| \geq r/2$, $w \in C_{b, \pi/4}$, which together imply that for all large $M$
\begin{equation}\label{OP23}
 \left|\int_{C_{b, \pi/4} \cap B_{r/2}^c(b)}  F(w,v',v)dw \right| \leq e^{CM^{2/3}}  \int_{C_{b, \pi/4}  \cap B_{r/2}^c(b)} e^{-c M |w| \log (1 + |w|)}|dw|   \leq Ce^{-c M},
\end{equation}
where $|dw|$ denotes intergration with respect to arc-length. 

On the other hand, if $v \in B_r(a)$ and $w \in B_{r/2}(b)$ we can perform the change of variables $\tilde{v} = (v - z_c)M^{1/3} \sigma$, $\tilde{w} = (w - z_c) M^{1/3} \sigma$ and apply Lemma \ref{S6analGS3} and Lemma \ref{horG} to obtain
$$ \left| \int_{C_{b, \pi/4}  \cap B_{r/2}(b)}  \hspace{-5mm} F(w,v',v)dw \right| \leq C \int_{C_{\mu + \rho, \pi/4}} \left| \frac{1}{\tilde{w}M^{-1/3} \sigma^{-1} + z_c - v'} \right| e^{c \log(1 + |\tilde{v}|) + c \log (1 + |\tilde{w}|)} \times $$
$$ \frac{ \pi \cdot \exp\left({-(\sqrt{2}/2)^3|\tilde{v}|^3/12 - (\sqrt{2}/2)^3|\tilde{w}|^3/12 + |\tilde{w}| A+ |\tilde{v}|A }\right)}{M^{1/3}\sigma | \sin(\pi M^{-1/3} \sigma^{-1}(\tilde{v} - \tilde{w})) |} |dw|.$$
We remark that in deriving the last inequality we bounded the double product of gamma functions in the definition of $F(w,v',v)$ by $C e^{c \log(1 + |\tilde{v}|) + c \log (1 + |\tilde{w}|)}$. To see why such a bound holds we can use (\ref{OP12.1}), and note that the right side of the top line is $O(1)$ for $v \in B_r(a)$ and $w \in B_r(b)$ since $r \leq 1/2$ by assumption, while the rational functions on the second line can be controlled by $Ce^{c \log(1 + |\tilde{v}|) + c \log (1 + |\tilde{w}|)}$. We also mention that the bound on 
$$| \exp(M[G(z_c + \tilde{v} \sigma^{-1} M^{-1/3}) - G(z_c + \tilde{w} \sigma^{-1} M^{-1/3})])|$$
 via 
$$C \exp\left(-(\sqrt{2}/2)^3|\tilde{v}|^3/12 - (\sqrt{2}/2)^3|\tilde{w}|^3/12 \right)$$
 can be obtained using only Lemma \ref{S6analGS3}  in the case
$|\tilde{w} - b| \leq M^{1/3 - 1/10}$ and $|\tilde{v} - a| \leq M^{1/3 - 1/10}$. When $ M^{1/3 - 1/10} \leq |\tilde{w} - b| \leq (r/2) M^{1/3}$ and $ M^{1/3 - 1/10} \leq |\tilde{v} - a| \leq r M^{1/3}$ one needs to further invoke Lemma \ref{horG}. The way this estimate is established is very similar to what we earlier did in Step 4 of the proof of Theorem \ref{mainThmGUE} in Section \ref{Section3.3} so we omit the details.

From Lemma \ref{SineBound} we know that
$$\left| \frac{\pi M^{-1/3}\sigma^{-1} }{\sin(\pi M^{-1/3} \sigma^{-1}(\tilde{v} - \tilde{w})) } \right| \leq C,$$
and also we know by definition of our contours that
$$\left| \frac{1}{\tilde{w}M^{-1/3} \sigma^{-1} + z_c - v'} \right| \leq CM^{1/3}.$$
Combining the last few estimates we see that
\begin{equation}\label{OP24}
\begin{split}
&\left| \int_{C_{b, \pi/4}  \cap B_{r/2}(b)} F(w,v',v)dw \right| \leq CM^{1/3} \int_{C_{\mu + \rho, \pi/4}} e^{c \log(1 + |\tilde{v}|) + c \log (1 + |\tilde{w}|)} \cdot \\
&e^{-(\sqrt{2}/2)^3|\tilde{v}|^3/12 - (\sqrt{2}/2)^3|\tilde{w}|^3/12+ |\tilde{w}| A + |\tilde{v}| A} |d\tilde{w}| \leq CM^{1/3},
\end{split}
\end{equation}
 where the last inequality used the fact that $C_{\mu + \rho, \pi/4}$ does not depend on $M$ and the integral is finite by the cube in the exponential. Equations (\ref{OP23}) and (\ref{OP24}) imply (\ref{KernelBoundV0S6}) when $|v - z_c| \leq r$. This suffices for the proof.
\end{proof}

\bibliographystyle{alpha} 
\bibliography{PD}

\end{document}